\documentclass{jams-l}

\usepackage{amssymb}
\usepackage[makeroom]{cancel}
\usepackage{graphicx}
\usepackage[cmtip,all]{xy}
\usepackage[usenames,dvipsnames]{color}
\usepackage{upgreek}
\usepackage{multirow}
\usepackage{hyperref}
\definecolor{linkcolour}{rgb}{0,0.2,0.6}
\hypersetup{colorlinks,breaklinks,urlcolor=linkcolour, linkcolor=linkcolour}
\usepackage{multicol}
\usepackage{colortbl}
\definecolor{Gray}{gray}{0.85}
\definecolor{LightCyan}{rgb}{0.88,1,1}
\usepackage{arydshln}
\usepackage{soul}

\newtheorem{theorem}{Theorem}[section]

\newtheorem{proposition}[theorem]{Proposition}

\theoremstyle{definition}
\newtheorem{definition}[theorem]{Definition}
\newtheorem{convention}[theorem]{Convention}

\newtheorem{example}[theorem]{Example}

\theoremstyle{remark}
\newtheorem{remark}[theorem]{Remark}

\numberwithin{equation}{section}

\newcommand\dcap{\mathrel{\ooalign{\rotatebox[origin=c]{-90}{$\longrightarrow$}\cr\kern0.4ex\hbox{$\not$}}}}

\usepackage{geometry}
\begin{document}

%\title[Modeling the action of E6 on P53 with pedigrads]{Modeling the action of E6 on P53 with pedigrads}
\large

\title{Category Theory for Genetics}

\author{R\'{e}my Tuy\'{e}ras}
\address{M.I.T., Department of Mathematics, 77 Massachusetts avenue, Cambridge, MA 02139}
\curraddr{}
\email{rtuyeras@mit.edu}
\thanks{This research was supported by the AFOSR grant, \emph{Categorical approach to agent interaction}, FA9550-14-1-0031 and the AFOSR grant, \emph{Pixel matrices and other compositional analyses of interconnected systems}, FA9550-17-1-0058.}

\date{}

\dedicatory{}

\begin{abstract}
We introduce a categorical language in which it is possible to talk about DNA sequencing, alignment methods, CRISPR, homologous recombination, haplotypes, and genetic linkage. This language takes the form of a class of limit-sketches whose categories of models can model different concepts of biology depending on what their categories of values are. We discuss examples of models in the category of sets and in the category of semimodules over the Boolean semi-ring $\{0,1\}$. We identify a subclass of models in sets that models the genetic material of living beings and another subclass of models in semimodules that models haplotypes. We show how the two classes are related via a universal property/construction.
\end{abstract}

\maketitle

\section{Introduction}

\subsection{Short presentation}
The goal of the present article is to define a type of algebraic structure in which it is possible to \emph{do genetics}. Even though the proposed structures are completely algebraic, we will see that they also enable us to talk about well-known statistical tools, such as the mapping functions used in genetic linkage \cite{ZieglerKonig,Speed_GMF,Haldane}, which map the recombination frequency of two chromosomal regions as a function of their distance on the chromosome and which usually take the form of cumulative distribution functions. The present paper should therefore be seen as an effort to clarify the tools of genetics through algebra and, more specifically, category theory, rather than a work that only restricts itself on describing the algebraic aspects of genetics. 

\subsection{Motivations}
Our objective is to construct a bridge between two completely disconnected domains of science, specifically genetics and category theory, through a series of papers. While genetics is well-known for its complexity, category theory is recognized for its clarity and expressive power \cite{SpivakBook,BrownPorter,MacLane}. The goal of the present program would be to reach a level of abstraction that would allow one to tackle questions whose formulation are too complicated to be addressed with the current tools.

The language of the present paper is rather mathematical, but the results and definitions that it contains always try to capture the biological reality. Note that, in the paper, some terms might be used in a biological sense while others might be used in a mathematical one -- this will usually be specified. For instance, the sentence ``a structure in which it is possible to do genetics'' means that we want to define a formal language rather than a model of some particular living body. The need for such an abstraction, in biology, has, for example, been recognized in \cite{Lazebnik}. 

Attempts at linking genetics (or in fact molecular biology) to a categorical thinking are not new. A first example is \cite{Japanese_work}, in which a category-like formalism is used to discuss the algebraic properties of ``DNA wallpapers''. Another work is \cite{slice_bio}, in which Carbone \& Gromov model DNA, RNA and proteins by using topological objects such as surfaces and moduli spaces. On the other hand, the program proposed herein instead tries to understand the mechanisms of genetics in themselves by forgetting the spacial aspect and focusing on the biological operations occurring in the body. Such an algebraic approach has already been discussed, from the point of view of neuroscience, in several unpublished works by Ehresmann (for example, see \cite{A_Ehresmann}) via the concepts of \emph{limit} and \emph{cone}. The present paper takes a step further, in the context of genetics, by providing a precise `limit theory' (in fact, a \emph{limit sketch}) that can be used to formalize precise concepts of genetics.
In this respect, our structures will define formal environments in which one wants to express a problem and say things about its solution (see the discussion of section \ref{ssec:overall_strategy}).

In addition of offering a formalism, the proposed program aims to tackle technical and/or conceptual problems of various parts of genetics. 

For example, phylogenetics have been recognized to need more clarity in order to be endowed with more satisfactory computerized procedures \cite{MorrisonWhy} and the set of possible operations that can underlie alignment methods \cite[Chap. 1]{Rosenberg} remains to be organized in the form of a theoretical framework, which no-one seems to have produced yet \cite{MorrisonFramework}. In this paper, we define a class of limit-sketches that can model the most basic operations of alignment methods, such as insertion of gaps, cutting of DNA patches, concatenation of DNA patches and homologous recombination. More complex operations, such as duplications, transpositions, inversions, deletion, insertion and substitution mutations would then need to be expressed at the level of the models for this theory (see Examples \ref{exa:duplication_iso_pedigrad}; \ref{exa:Transcription_in_set}; \ref{exa:Mutations_in_set_are_spans}; \ref{exa:Mutations_presentable_morphisms} \& \ref{exa:Transcription_translation_presentable_morphisms}). This structural hierarchization, which pertains to the language of category theory, goes in direction of the program suggested in \cite{MorrisonFramework,MorrisonWhy} by trying to ``[recognize] mechanisms rather than assuming that all the variation occurs at random'' \cite[page 156, right col., l. 5]{MorrisonFramework}. 
For instance, distinguishing homologous recombination from mutations by setting them at different levels (namely, that of the theory and that of the models, respectively) translates the fact that homologous recombination is more of a systematic event pertaining to the biological reality while mutations are more of a set of possible events pertaining to Evolution \cite[Chap. 3]{Mount}.

Another example is in genetic linkage, where the current form of mapping functions do not quite fit recombination models, mainly because they do not manage to model cross-over interference \cite[page 3, right col.]{Speed_GMF}. It is also suggested \cite[\emph{Loc. cit.}]{Speed_GMF} that the measuring of cross-over interference should be done between potentially-separated DNA patches rather than necessarily-adjacent ones. In this article, we will be able to specify such disconnected interference relations via the cones of our limit-sketches (see section \ref{Toward_more_formalism}). In section \ref{Genetic_linkage_and_mapping_functions}, we will also see how these cones allow us to specify probability spaces of recombination events so that we will be able to `recover' Haldane's mapping functions \cite{Haldane}. We will conclude that the flexibility of our language makes it a good candidate for providing a framework in which it is possible to talk about multi-locus genetic linkage.

Finally, it is interesting to note that, both the hierachization of biological operations (such as the distinction between homologous recombination and mutations) and the specification of topologies (indicating where the recombination events occur) are shown to be parameters that can significantly determine the shape of evolutionary trees \cite{ArenasPosada}. It is even concluded that methods being able to manage the space of evolutionary trees resulting from the space of these parameters is very much needed. We would here suggest that such a scheme would first require a formalization of the latter. In this article, we choose to formalize this space of parameters in terms of limit-sketches and their models.

%Cross-over interference essentially occurs when some recombination event at a certain position on a DNA strand can influence the likelyhood of the recombination events around that position .

\subsection{Overall strategy}\label{ssec:overall_strategy}
The goal of the overall project, which will consist of a series of papers, is to unify various domains of genetics, molecular biology and chemistry within a same language in order to clarify new concepts, whose complexity could tend to increase due to the incoming of large amount of data, as well as to facilitate the dissemination of knowledge between researchers. Such a unification has been shown to be important \cite{Vidal} for the reason that ``increasingly sophisticated modeling concepts remain to be developed before the promise of systems biology can be fully realized'' (see \cite[section 4]{Vidal}). 

To formalize biology, one first needs to understand what the components that constitute the biological knowledge are. Very broadly, we could here claim that biological knowledge can be represented as a pair of two components: the method $\Omega$ (or the logic) and the observation $P$ (or the analysis).

For instance, it is not rare to see different research models using the same method, but different observations. Here is an easy example: it may happen that different researchers are interested to \texttt{read} or \texttt{ignore} certain patches of a sequence of molecules -- this is the method -- but their observations are different because their respective sequences are made of different building blocks; e.g. DNA nucleotides, RNA nucleotides, methylated nucleotides, amino acids, codons, alleles, etc. Thus, each analysis would provide different objects, say $(\Omega,P_1)$, $(\Omega,P_2)$, and $(\Omega,P_3)$, whose methods are all equal. If these researchers wanted to compare or unify their models, the idea would be that they could do so by using some sort of `pushout' construction along the method $\Omega$ to construct a new model $(\Omega, P_{123})$ in which all of their models could be discussed at the same time.
\[
\xymatrix@-15pt{
(\Omega,\emptyset)\ar[rr]\ar[rd]\ar[dd]&&(\Omega,P_1)\ar@{..>}[dd]\\
&(\Omega,P_2)\ar@{..>}[rd]&\\
(\Omega,P_3)\ar@{..>}[rr]&&(\Omega, P_{123})
}
\]
On the other hand, if different researchers had different methods and different analyses, it is likely that their models can still be related in one way or another via sub-models. Understanding how each of their models can be unified within a same model would require one to study the diagram of relations existing between them.
\[
\xymatrix@-15pt{
(\Omega_1,P_1)&&(\Omega_2,P_2)&&(\Omega_3,P_3)\\
&(\Omega_{12},P_{12})\ar[ru]\ar[lu]&&(\Omega_{23},P_{23})\ar[ru]\ar[lu]&
}
\]
Again, the idea would be that a `pushout' construction of such diagram would give a better picture of the studied system. Category theory is ideal to talk about relations between objects and this is the reason why this language seems to be the best candidate to work towards the unification of biology.

Of course, before being able to express oneself in the previous terms, one needs to formalize the idea of pair $(\Omega,P)$. We will start by doing so, in the present article, from section \ref{ssec:Recombination_schemes_and_pedigrads} via the concept of \emph{recombination scheme}, for which $\Omega$ is a pre-ordered set and $P$ is a functor preserving certain limits. Here, the term `scheme' is not neutral as it refers to somewhat similar pairs $(U,\mathcal{F})$ used in algebraic geometery to study algebraic varieties. While future work will aim to further develop the language of pairs $(\Omega,P)$, the present article mainly focus on introducing the language in a way that should be accessible to any biologist who knows the basics of category theory. Certain propositions given at the end of section \ref{ssec:Recombination_schemes_and_pedigrads} are slightly involved, but the paper works toward making their statements and proofs as accessible as possible.

\subsection{Road map and results}
The goal of the present paper is to define a class of theories, called \emph{chromologies}, whose models, called \emph{pedigrads}, can recover various aspects of genetics. We will start by defining chromologies in section \ref{sec:Chromologies_and_Pedigrads}, from section \ref{sec:pre-ordered_sets} to section \ref{sec:chromologies}, while the models (pedigrads) for these theories will be defined in sections \ref{ssec:Logical_systems} \& \ref{ssec:Pedigrads}. Intuitively, chromologoies will allow us to do all sorts of basic DNA manipulations such as DNA sequencing, alignment methods, CRISPR \cite{Pennisi} and homologous recombination whereas the pedigrads will allow us to give a context to these operations (which can be handled differently depending on the environment in which they are processed).

In section \ref{sec:Examples_of_pedigrads_in_sets}, we define a class of canonical pedigrads taking values in the category of sets. The images of their underlying functors will be seen as sets containing DNA sequences. These canonical pedigrads will be used in section \ref{sec:Pedigrads_in_semimodules_over_semi-rings} to generate new pedigrads in a category of semimodules over a particular semi-ring. The examples and illustrations given in section \ref{sec:Examples_of_pedigrads_in_sets} will be important to understand the content of section \ref{sec:Pedigrads_in_semimodules_over_semi-rings}.

In section \ref{sec:Pedigrads_in_semimodules_over_semi-rings}, we will construct a class of canonical pedigrads (see Definition \ref{def:Canonical_pedigrads_in_semimodules} and Theorem \ref{theo:representable_pedigrad_E_b_varepsilon}) in the category of semimodules over the Boolean semi-ring $B_2=\{0,1\}$. The elements belonging to the images of their underlying functors can be seen as theoretical `haplotypes'. Then, in sections \ref{ssec:presentable_functors} \& \ref{ssec:presentable_pedigrads}, we will define a larger class of pedigrads, called \emph{presentable pedigrads}, for which our class of canonical pedigrads will satisfy a universal property (see Theorem \ref{theo:universal_property}). Presentable pedigrads will be equipped with a type of morphism that can model polymorphic DNA mutations (see Example \ref{exa:Mutations_presentable_morphisms}) as well as the usual transcription operations (see Example \ref{exa:Transcription_translation_presentable_morphisms}). We will see that the expressive power of semimodules, and, more specifically, the equations that they satisfy, will allow us to express biological phenomena such as nullomers and other selective behaviors resulting from RNA translation (see Example \ref{exa:More_quotients}).

Finally, in section \ref{Genetic_linkage_and_mapping_functions}, we will see how it is possible to recover the `mapping functions' \cite{Haldane} expressing the genetic distance between two markers on a given chromosome. Section \ref{Toward_more_formalism} will suggest various ways of refining the obtained mapping functions.

\subsection{Acknowledgments}
I would like to thank David Spivak and Eric Neumann for very useful discussions, remarks and questions regarding the content of this paper.

\section{Chromologies and Pedigrads}\label{sec:Chromologies_and_Pedigrads}

The goal of this section is to introduce a set of theories whose models try to capture the logic of genetics. To justify why our theories look the way they do, we need to recall a few facts regarding the construction of theories in general. First, recall that, classically, models for theories are defined as sets equipped with some operations. For instance, a \emph{ring} is a set $R$ equipped with two operations $\cdot:R \times R \to R$ and $+:R \times R \to R$ making certain diagrams commute.

More categorically, rings are also product-preserving functors from a certain  product sketch\footnote{A small category equipped with a subset of its wide spans.} $\mathtt{Ring}$ (the theory) to the category $\mathbf{Set}$ of sets and functions \cite{Ehresmann}. This functorial point of view was introduced by Lawvere \cite{Lawvere1963} in 1963 via the concept of what is now called a \emph{Lawvere theory} -- the theory $\mathtt{Ring}$ being an example. The advantage of functors over sets equipped with functions is that functors allow us to clearly distinguish between what is intrinsically true in a model (via the theory) and what can occasionally be true in the model (via the images of the functor). Then, the formalism accompanying the language of functors allows us to more carefully think about the mechanisms governing the models.

Since Lawvere theories were meant to capture the logic of algebraic structures equipped with multivariate functions, their objects were taken to be the set of natural numbers in order to specify the arities of the functions. Along those lines, since the goal of the present section is to define a theory that captures the logic of genetics and whose operations take DNA patches as inputs, the objects of our theory will look like DNA segments. 
Note that, while, in rings, one \emph{adds} and \emph{multiplies} terms together, in genetics, one \emph{cuts}, \emph{aligns} and \emph{recombines} DNA strands together. Therefore, our theory will be based on these operations.

For illustration, an integer object in a Lawvere theory can easily be represented as a finite sequence of atoms; e.g the object 6 would be represented by six atoms as follows.
\begin{equation}\label{eq:Lawvere_integers}
6 = \xymatrix@C-30pt{(\bullet&\bullet&\bullet&\bullet&\bullet&\bullet)}
\end{equation}
These atoms can make it easier to see how the functors (models) defined on the Lawvere theory send the integer objects to the product objects of the models; e.g. for a functor $R$, the image $R(6)$ would be sent to a product of the following form.
\[
R(\bullet) \times R(\bullet) \times R(\bullet) \times R(\bullet) \times R(\bullet) \times R(\bullet)
\]
In the case of DNA, the idea is to copy the previous picture, but by adding enough information to be able to do genetics. If one looks at the type of pictures drawn by biologists to explain homologous recombination, alignment methods or even gene linkage, one can often see pictures of chromosomal patches subdivided in terms of selected and masked regions, as shown below.
\[
\includegraphics[height=5cm]{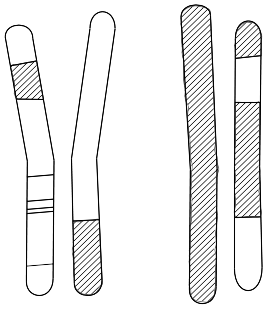}
\]
These colored regions are obviously reminiscent of the term \emph{chromo-some}\footnote{meaning \emph{color-body}} itself. The regional separations are also reminiscent of some sort of topology -- or metric. If one tries to merge these topological and colored components with the type of atomic representation given in (\ref{eq:Lawvere_integers}), we are likely to end up with the following type of picture.
\begin{equation}\label{eq:intro_representation_patch}
\xymatrix@C-30pt{
(\bullet&\bullet&\bullet)&(\circ&\circ)&(\bullet&\bullet&\bullet&\bullet)&(\bullet&\bullet&\bullet&\bullet&\bullet)&(\circ&\circ&\circ)&(\circ)
}
\end{equation}
In picture (\ref{eq:intro_representation_patch}), the black nodes could indicate the regions of the chromosome that one wants to use while the white nodes could indicate the parts of the chromosome that one wants to ignore (mask). For convenience, in the examples and illustrations of this article, we will only consider two colors (black \& white), but the theory will work for more than two colors. Our set of colors will be encoded by pre-ordered sets, whose semantics will allow us to  \emph{select} and \emph{cut}.

\subsection{Pre-ordered sets}\label{sec:pre-ordered_sets}
In this paper, the most basic notions of ordered set are expected to be known by the reader (e.g. partially ordered sets; totally (or linearly) ordered sets; pre-ordered sets; see \cite[Page 11]{MacLane}). However, because pre-orders play an important role later on, it was felt appropriate to recall their definition here. 

A \emph{pre-ordered set} consists of a set $\Omega$ and a binary relation $\leq$ on $\Omega$ satisfying the following logical implications.
\begin{itemize}
\item[1)] (reflexivity) for every $x \in \Omega$, the relation $x \leq x$ holds;
\item[3)] (transitivity) for every $x,y,z \in \Omega$, if $x \leq y$ and $y \leq z$ hold, then so does $x \leq z$.
\end{itemize}

\begin{example}\label{exa:pre-ordered_set_0_and_1}
The set $\{0,1\}$ is a pre-ordered set if one sets $0 \leq 1$; $0 \leq 0$ and $1 \leq 1$. The resulting pre-ordered set is usually known as the \emph{Boolean} pre-ordered set and the values 0 and 1 are usually denoted as $\mathtt{false}$ and $\mathtt{true}$, respectively.
\end{example}

\begin{remark}[Representation]
Pre-ordered sets may happen to be sets of labels (or even sets of structures) instead of being sets of integers. In the case of the Boolean pre-ordered set given in Example \ref{exa:pre-ordered_set_0_and_1}, the labels $\mathtt{false}$ and $\mathtt{true}$ will sometimes be used instead of the integers $0$ and $1$, mainly for the sake of clarity.
\end{remark}

\begin{example}
The set $\{0,1\}$ could also be equipped with the discrete pre-order made of the reflexive relations $0 \leq 0$ and $1 \leq 1$ only.
\end{example}

\begin{example}\label{exa:product_pre-order_set_0_1}
For every integer $n$, the $n$-fold Cartesian product $\{0,1\}^{\times n}$ of the pre-ordered set given in Example \ref{exa:pre-ordered_set_0_and_1} is equipped with a pre-order relation $\leq$ that compares two tuples in $\{0,1\}^{\times n}$, say of the form $(x_1,\dots,x_n) \leq (y_1,\dots,y_n)$, if, and only if, the relation $x_i \leq y_i$ holds for every $1 \leq i \leq n$.
\end{example}

\begin{example}
The interval $[0,1]$ is a pre-ordered set for the usual pre-order ``being less than or equal to'' defined on the set $\mathbb{R}$ of real numbers.
\end{example}

\begin{remark}[Pre-order categories]
A pre-ordered set is equivalently a category in which there exists at most one arrow between every pair of objects. In the sequel, a pre-ordered set will sometimes be called a \emph{pre-order category} to emphasize its categorical nature.
\end{remark}

\subsection{Finite sets of integers}
For every positive integer $n$, we will denote by $[n]$ the finite set of integers $\{1,2,\dots,n\}$. We will also let $[0]$ denote the empty set. In the sequel, for every non-negative integer $n$, the set $[n]$ will implicitly be equipped with the order associated with the set of integers (note that the restriction of this order on $[0]$ is the empty order).

\subsection{Segments}
Let $(\Omega,\preceq)$ denote a pre-ordered set. 
A \emph{segment} over $\Omega$ consists of a pair of non-negative integers $(n_1,n_0)$, an order-preserving\footnote{for every relation $x \leq y$ in the domain, the relation $t(x) \leq t(y)$ holds in the codomain.} surjection $t:[n_1] \to [n_0]$ and a function $c:[n_0] \to \Omega$.

\begin{remark}[Representation]
Segments have all the necessary data to encode the type of pictures given in (\ref{eq:intro_representation_patch}). For a segment $(t,c)$ as defined above, the finite set $[n_1]$ represents the range of elements composing the segment
\[
n_1 = \xymatrix@C-30pt{(\bullet&\bullet &\, \cdots\,  & \bullet)}
\]
while the fibers $t^{-1}(1), \dots, t^{-1}(n_0)$ of the surjection $t:[n_1] \to [n_0]$ gather these elements into patches (see the brackets below).
\[
t = \xymatrix@C-30pt{(\bullet&\bullet&\bullet)&(\bullet&\bullet&\bullet&\bullet)&(\bullet&\bullet&\,\cdots\, &\bullet)&(\bullet& \bullet)}
\]
Finally, the different colors associated with the patches of the segment are specified by the map $c:[n_0] \to \Omega$. For instance, if we take $\Omega$ to be the Boolean pre-ordered set $\{\mathtt{false}\leq \mathtt{true}\}$ of Example \ref{exa:pre-ordered_set_0_and_1} and we choose to associate the white color with the $\mathtt{false}$ value and the black color with the $\mathtt{true}$ value, then an identity of the form $c(1) = \mathtt{false}$ will be represented by coloring all the elements of $[n_1]$ living in the fiber $t^{-1}(1)$ in white.
\[
(t,c) = \xymatrix@C-30pt{(\circ&\circ&\circ)&(\bullet&\bullet&\bullet&\bullet)&(\bullet&\bullet&\,\cdots\, &\bullet)&(\bullet& \bullet)}
\]
Note that if $\Omega$ contains more elements, then we need to use more colors (which can also be represented by numbers). These colors could also mean all sorts of things, including actions such as \texttt{ignore}, \texttt{read}, \texttt{start reading}, \texttt{stop reading}, \texttt{misread} (or \texttt{mutate}). The pre-order on the colors would then specify semantic priorities between the different tasks or functions associated with the colors (see below).
\smallskip

\begin{center}
\begin{tabular}{|c|c|c|}
\hline
\cellcolor[gray]{0.8}2 colors & \cellcolor[gray]{0.8}4 colors & \cellcolor[gray]{0.8}5 colors\\
\hline
$\{0,1\}$ & $\{0,1,2,3\}$ & $\{0,1,2,3,4\}$\\
\hline
{$\xymatrix@-15pt{
\fbox{\texttt{read}}\\
\\
\\
\fbox{\texttt{ignore}}\ar[uuu]}$}
&
{$\xymatrix@-15pt{
&\fbox{\texttt{read}}&\\
\fbox{\texttt{start}}\ar[ru]&&\fbox{\texttt{finish}}\ar[lu]\\
&\fbox{\texttt{ignore}}\ar[ur]\ar[ul]&}$}
&
{$\xymatrix@-15pt{
&\fbox{\texttt{read}}&\\
\fbox{\texttt{start}}\ar[ru]&\fbox{\texttt{misread}}\ar[u]&\fbox{\texttt{stop}}\ar[lu]\\
&\fbox{\texttt{ignore}}\ar[u]\ar[ur]\ar[ul]&}$}\\
\hline
\end{tabular}
\end{center}
\end{remark}

\begin{remark}[Notations]
Note that the specification of the data $n_1$ and $n_0$ is redundant with the data of the function $t$ and $c$. Later on, a segment will often be denoted as a pair $(t,c)$ and, every so often, as an arrow $(t,c):[n_1] \multimap [n_0]$.
\end{remark}

\begin{convention}[Domains, topologies \& types]
For every segment $(t,c):[n_1] \multimap [n_0]$, the data $[n_1]$ will be called the \emph{domain} of $(t,c)$, the data $t$ will be called the \emph{topology} of $(t,c)$ and the data $(n_1,n_0)$ will be called the \emph{type} of $(t,c)$. The type of a segment will always be specified as an arrow of the form $[n_1] \multimap [n_0]$.
\end{convention}

\begin{definition}[Homologous segments]
Two segments $(t,c)$ and $(t',c')$ over $\Omega$ will be said to be \emph{homologous} if their topologies $t$ and $t'$ are equal.
\end{definition}

\begin{definition}[Quasi-homologous segments]
Two segments $(t,c)$ and $(t',c')$ over $\Omega$ will be said to be \emph{quasi-homologous} if their domains $[n_1]$ and $[n_1']$ are equal.
\end{definition}

\subsection{Morphisms of segments}\label{ssec:Morphisms_of_chromosomal_patches}
Let $(\Omega,\preceq)$ be a pre-ordered set and $(t,c):[n_1] \multimap [n_0]$ and $(t',c'):[n_1'] \multimap [n_0']$ be two segments over $\Omega$. A morphism of segments from $(t,c)$ to $(t',c')$ consists of
\begin{itemize}
\item[1)] an order-preserving injection $f_1:[n_1] \to [n_1']$;
\item[2)] an order-preserving function $f_0:[n_0] \to [n_0']$;
\end{itemize}
such that the inequality $c' \circ f_0(i) \preceq c(i)$ holds for every $i \in [n_0]$ and the following diagram commutes.
\[
\xymatrix{
[n_1]\ar@{->>}[r]^{t}\ar@{)->}[d]_{f_1}&[n_0]\ar[d]^{f_0}\\
[n_1']\ar@{->>}[r]^{t'}&[n_0']
}
\]
It is easy to check that the class of morphisms of segments over $\Omega$ is stable under component-wise compositions and admits identities on every segment. We will denote by $\mathbf{Seg}(\Omega)$ the category whose objects are segments over $\Omega$ and whose arrows are morphisms between these.

From now on, we will regard the notations $f_1$ and $f_0$ given above as a conventional notation for morphisms in $\mathbf{Seg}(\Omega)$. Below, we give several examples of typical morphism in $\mathbf{Seg}(\Omega)$ where $\Omega$ is taken to be the Boolean pre-ordered set of Example \ref{exa:pre-ordered_set_0_and_1}.

\begin{example}[Locality]
If both components $f_1$ and $f_0$ are identities, then the inequality $c' \circ f_0\preceq c$ `decreases' the colors of the segment as illustrated below, on the left.
\[
\begin{array}{c}
\xymatrix@C-30pt{
(\bullet&\bullet&\bullet)&(\bullet&\bullet)&(\bullet&\bullet&\bullet&\bullet)&(\bullet&\bullet&\bullet&\bullet&\bullet)&(\circ&\circ&\circ)&(\bullet)
}\\
\rotatebox[origin=c]{-90}{$\longrightarrow$} \\
\xymatrix@C-30pt{
(\circ&\circ&\circ)&(\circ&\circ)&(\bullet&\bullet&\bullet&\bullet)&(\bullet&\bullet&\bullet&\bullet&\bullet)&(\circ&\circ&\circ)&(\circ)
}
\end{array}
\quad\quad\quad\quad\quad
\begin{array}{c}
\xymatrix@C-30pt@R-11pt{
(\dots)&(\circ&\circ\ar[d]|{\xcancel{\quad\,\,}}&\circ)&(\dots)\\
(\dots)&(\bullet&\bullet&\bullet)&(\dots)
}
\end{array}
\]
\textit{Interpretation:} This type of morphism tells us that one is able to select/cut local patches from a segment. This is, for instance, the type of morphism that one may want to use to model CRISPR, namely separating a patch from a segment. Note that, because reading a segment (black color) has a higher semantic priority than ignoring it (white color), turning white regions into black ones, as shown above, on the right, is forbidden. The order relation on the colors can therefore be a way of encoding irreversible (or energy-releasing) events.
\end{example}

\begin{example}[Relativity]\label{exa:Metric_morphism}
If only the component $f_1$ is an identity morphism, then the component $f_0$ can merge the regions defining the topology.
\[
\begin{array}{c}
\xymatrix@C-30pt{
(\bullet&\bullet&\bullet)&(\circ&\circ)&(\bullet&\bullet&\bullet&\bullet)&(\bullet&\bullet&\bullet&\bullet&\bullet)&(\circ&\circ&\circ)&(\circ)
}\\
\rotatebox[origin=c]{-90}{$\longrightarrow$} \\
\xymatrix@C-30pt{
(\circ&\circ&\circ&\circ&\circ)&(\bullet&\bullet&\bullet&\bullet&\bullet&\bullet&\bullet&\bullet&\bullet)&(\circ&\circ&\circ&\circ)
}
\end{array}
\]
\textit{Interpretation:} This type of morphism implies that the way one parses the patches of a segment influences the way one parses the whole segment (e.g. from codons to genes). However, because there is no arrow that increases the number of brackets from its domain to its codomain, the way one parses a segment might not necessarily reflect the way the patches are parsed (e.g from gene to codons).
\end{example}

\begin{example}[Flexibility]\label{exa:Flexibility_morphism}
If the component $f_1$ is not an identity morphism, then the range of the segment increases. Below, we suppose that the identity $c' \circ f_0 = c$ holds.
\[
\begin{array}{r}
\xymatrix@C-30pt{
(\bullet&\bullet&\bullet)&(\circ&\circ)&(\bullet&\bullet&\bullet&\bullet)&(\bullet&\bullet&\bullet&\bullet&\bullet)&(\circ&\circ&\circ)&(\circ)
}\\
\multicolumn{1}{c}{\rotatebox[origin=c]{-90}{$\longrightarrow$}}\\
\xymatrix@C-30pt{
(\bullet&\bullet&\bullet&\bullet)&(\bullet)&(\circ&\circ&\circ&\circ)&(\bullet&\bullet&\bullet&\bullet)&(\bullet&\bullet&\bullet&\bullet&\bullet)&(\circ&\circ&\circ)&(\circ)
}
\end{array}
\]
\textit{Interpretation:} This type of morphism allows one to insert particular nucleobases or spaces in the parsing of a segment. 
%This obviously changes the  tempo of the parsing. 
For instance, spaces become necessary if one wants to recombine segments that are not necessarily (quasi-)homologous. A morphism inserting a space would then correspond to a choice of `sequence alignment' in bioinformatics (this will be illustrated in Example \ref{exa:morphisms_as_inclusion_of_words}).
\end{example}

\begin{remark}[Initial object]
For every pre-ordered set $\Omega$, the segment (over $\Omega$) of type $[0] \multimap [0]$ that is given by the obvious order-preserving surjection $!:\emptyset \to \emptyset$ and the canonical function $!:\emptyset \to \Omega$ is an initial object in $\mathbf{Seg}(\Omega)$. Note that such an object is formal and does not really possess any biological interpretation other than it can help us express the idea of `absence'.
\end{remark}

\subsection{Pre-orders on homologous segments}
Let $(\Omega,\preceq)$ be a pre-ordered set and let $t:[n_1] \to [n_0]$ be an order-preserving surjection. The subcategory of $\mathbf{Seg}(\Omega)$ whose objects are the homologous segments of topology $t$ and whose arrows are the morphisms of segments for which the components $f_0$ and $f_1$ are identities will be denoted by $\mathbf{Seg}(\Omega:t)$
and referred to as the \emph{category of homologous segments (over $\Omega$) of topology $t$}.

\begin{proposition}
For every order-preserving surjection $t:[n_1] \to [n_0]$, the category $\mathbf{Seg}(\Omega:t)$ is a pre-order category.
\end{proposition}
\begin{proof}
The pre-order relations $(t,c) \leq (t,c')$ associated with $\mathbf{Seg}(\Omega:t)$ are induced by the following pre-order relations in $(\Omega,\preceq)$, which directly come from the definition of the arrows of $\mathbf{Seg}(\Omega:t)$.
\[
c'(i) \preceq c(i)\,,\quad\forall i \in [n_0]
\]
It is straightforward to see that this defines a reflexive and transitive relation.
\end{proof}

\subsection{Pre-orders on quasi-homologous segments}
Let $(\Omega,\preceq)$ be a pre-ordered set and let $n_1$ be a non-negative integer. The subcategory of $\mathbf{Seg}(\Omega)$ whose objects are the quasi-homologous segments of domain $[n_1]$ and whose arrows are the morphisms segments for which the component $f_1$ is an identity will be denoted by $\mathbf{Seg}(\Omega\,|\,n_1)$ and called the \emph{category of quasi-homologous segments (over $\Omega$) of domain $n_1$}.

\begin{proposition}\label{prop:quasi_homologous_preordered_category}
If there exists a morphism $(t,c) \to (t',c')$ in $\mathbf{Seg}(\Omega\,|\,n_1)$, then it is the only morphism of type $(t,c) \to (t',c')$ in $\mathbf{Seg}(\Omega)$.
\end{proposition}
\begin{proof}
Let $(\mathrm{id},f_0):(t,c) \to (t',c')$ be the morphism of the statement in $\mathbf{Seg}(\Omega\,|\,n_1)$ and let $(g_1,g_0):(t,c) \to (t',c')$ be another morphism in $\mathbf{Seg}(\Omega)$. Because $g_1$ is an order-preserving inclusion of type $[n_1] \to [n_1]$, it must be an identity, so that the identity $g_0 \circ t = t'$ holds. On the other hand, the identity $f_0 \circ t = t'$ also holds, which means that $g_0 \circ t = f_0 \circ t$. Because $t$ is an epimorphism, the identity $g_0=f_0$ must hold.
\end{proof}

\begin{remark}[Pre-order category]
Proposition \ref{prop:quasi_homologous_preordered_category} implies that the category $\mathbf{Seg}(\Omega\,|\,n_1)$ of quasi-homologous segments is a pre-order category.
\end{remark}

\begin{remark}[Zero domain]
The category $\mathbf{Seg}(\Omega\,|\,0)$ of quasi-homologous segments of empty domain is a terminal category whose only object is the initial object of $\mathbf{Seg}(\Omega)$.
\end{remark}

\subsection{Cones}\label{ssec:cones}
Recall that a \emph{cone} in a category $\mathcal{C}$ consists of an object $X$ in $\mathcal{C}$, a small category $A$, a functor $U:A \to \mathcal{C}$ and a natural transfomation $\Delta_{A}(X) \Rightarrow U$ where $\Delta_{A}(X)$ denotes the constant functor $A \to \mathbf{1} \to \mathcal{C}$ mapping every object in $A$ to the object $X$ in $\mathcal{C}$.

\subsection{Chromologies}\label{sec:chromologies}
A \emph{chromology} is a pre-ordered set $(\Omega,\preceq)$ that is equipped, for every non-negative integer $n$, with a set $D[n]$ of cones in the category $\mathbf{Seg}(\Omega\,|\,n)$. A chromology as above will later be denoted as a pair $(\Omega,D)$.

\begin{remark}[Future examples]
In section \ref{ssec:Distributive_and_exactly_distributive_chromologies}, we will see several examples of chromologies, which will be used throughout this article.
\end{remark}

\subsection{Logical systems}\label{ssec:Logical_systems}
We will speak of a \emph{logical system} to refer to a category $\mathcal{C}$ that is equipped with a subclass of its cones $\mathcal{W}$ (see section \ref{ssec:cones}).

\begin{remark}[Limit sketch]
The difference between a logical system and a limit sketch is that the latter is  defined as a small category that is equipped with a subset of its cones. A logical system is also meant to be the codomain of a functor whose domain is a limit sketch.
\end{remark}

\subsection{Pedigrads} \label{ssec:Pedigrads}
Pedigrads are algebraic structures that model the logical rules of chromologies.  Their name refers to the concept of `pedigree' used in genetics. Let $(\Omega,D)$ be a chromology and $(\mathcal{C},\mathcal{W})$ be a logical system. A \emph{pedigrad} in $(\mathcal{C},\mathcal{W})$ is a functor $\mathbf{Seg}(\Omega) \to \mathcal{C}$ sending, for every non-negative integer $n$, the cones in $D[n]$ to cones in $\mathcal{W}$. 

\begin{convention}[$\mathcal{W}$-pedigrads]
Because we will often consider the same category $\mathcal{C}$ for different classes of cones $\mathcal{W}$, we will often refer to a pedigrad in $(\mathcal{C},\mathcal{W})$ as a $\mathcal{W}$-pedigrad.
\end{convention}

\subsection{Morphisms of pedigrads} \label{ssec:Morphisms_of_pedigrads}
Recall that, for every pair of categories $\mathcal{C}$ and $\mathcal{D}$, the notation $[\mathcal{C},\mathcal{D}]$ denotes the category whose objects are functors $\mathcal{C} \to \mathcal{D}$ and whose arrows are natural transformations in $\mathcal{D}$ over $\mathcal{C}$.
Let $(\Omega, D)$ be a chromology and $(\mathcal{C},\mathcal{W})$ be a logical system. A \emph{morphism of pedigrads} from a pedigrad $A:\mathbf{Seg}(\Omega)\to\mathcal{C}$ in $(\mathcal{C},\mathcal{W})$ for $(\Omega, D)$ to a pedigrad $B:\mathbf{Seg}(\Omega)\to\mathcal{C}$ in $(\mathcal{C},\mathcal{W})$ for $(\Omega, D)$ is an arrow  $A \Rightarrow B$ in the category $[\mathbf{Seg}(\Omega),\mathcal{C}]$. 

\begin{convention}[Category of pedigrads]
The full subcategory of $[\mathbf{Seg}(\Omega),\mathcal{C}]$ whose objects are the pedigrads in a logical system $(\mathcal{C},\mathcal{W})$ for a chromology $(\Omega,D)$ will be called \emph{the category of pedigrads in  $(\mathcal{C},\mathcal{W})$ for $(\Omega,D)$}.
\end{convention}

\section{Examples of pedigrads in sets}\label{sec:Examples_of_pedigrads_in_sets}
The goal of this section is to construct two canonical classes of pedigrads that take their values in the category $\mathbf{Set}$ of sets and functions (see Propositions \ref{prop:mono-pedigrad_in_set} \& \ref{prop:E_b_varepsilon_W_iso_pedigrad_exactly_distributive}). Throughout the section, we shall also let $(E,\varepsilon)$ be a fixed pointed set and $(\Omega,\preceq)$ be a pre-ordered set.

\subsection{Truncation functors}

In this section, we define a truncation operation, which will turn out to be very useful for constructing pedigrads.

\begin{definition}[Truncation]
For every segment $(t,c):[n_1] \multimap [n_0]$ over $\Omega$ and element $b \in \Omega$, we will denote by $\mathsf{Tr}_b(t,c)$ the subset $\{i \in [n_1]~|~ b \preceq c \circ t(i)\}$ of $[n_1]$. This is the set of all elements in $[n_1]$ whose images via $c \circ t$ is greater than or equal to $b$ in $\Omega$.
\end{definition}

\begin{example}[Truncation]
Let $(\Omega,\preceq)$ be the Boolean pre-ordered set $\{0 \leq 1\}$. If we consider a segment over $\Omega$, as given below, on the left, the operation $\mathsf{Tr}_b$ for which $b$ is taken to be equal to $1$ will select all the integers in the domain of $(t,c)$ that are associated with black nodes while the operation $\mathsf{Tr}_b$ for which $b$ is taken to be equal to $0$ will select all the integers in the domain of $(t,c)$.
\[
(t,c)=\xymatrix@C-30pt{
(\bullet&\bullet&\bullet)&(\circ&\circ)&(\bullet&\bullet&\bullet&\bullet)&(\circ&\circ&\circ&\circ&\circ)&(\bullet&\bullet&\bullet)&(\circ)
}
\quad\quad
\begin{array}{l}
\mathsf{Tr}_1(t,c)=\{1,2,3,6,7,8,9,15,16,17\}\\
\mathsf{Tr}_0(t,c)=[18]
\end{array}
\]
\end{example}

\begin{definition}[Sub-objects]
For every non-negative integer $n$, we will speak of a \emph{sub-object} of $[n]$ to refer to a subset of $[n]$. A \emph{morphism of sub-objects of $[n]$} is an inclusion of sets between the two sub-objects.
\end{definition}

\begin{example}[Truncation operations and sub-objects]
Let $(\Omega,\preceq)$ be the Boolean pre-ordered set $\{0 \leq 1\}$. If we consider the morphism of segments of Example \ref{exa:Metric_morphism}, which is recalled below, on the left, we can see that the truncation operation $\mathsf{Tr}_1$ gives, on the right, two sub-objects of the domain $[18]$ that we can relate via a morphism of sub-objects.
\[
\begin{array}{crc}
(t,c)=\xymatrix@C-30pt{
(\bullet&\bullet&\bullet)&(\circ&\circ)&(\bullet&\bullet&\bullet&\bullet)&(\bullet&\bullet&\bullet&\bullet&\bullet)&(\circ&\circ&\circ)&(\circ)
}&&\mathsf{Tr}_1(t,c)=\{1,2,3,6,7,8,9,10,11,12,13,14\}\\
\rotatebox[origin=c]{-90}{$\longrightarrow$} &&\rotatebox[origin=c]{-90}{$\supseteq$}\\
(t',c')=\xymatrix@C-30pt{
(\circ&\circ&\circ&\circ&\circ)&(\bullet&\bullet&\bullet&\bullet&\bullet&\bullet&\bullet&\bullet&\bullet)&(\circ&\circ&\circ&\circ)
}&&\mathsf{Tr}_1(t',c')=\{6,7,8,9,10,11,12,13,14\}
\end{array}
\]
The fact that a morphism of segments of the form $(t,c) \to (t',c')$ gives rise to an inclusion $\mathsf{Tr}_1(t',c') \subseteq \mathsf{Tr}_1(t,c)$ is explained by Proposition \ref{prop:preserve_Tr_b_morphism}.
\end{example}

\begin{proposition}\label{prop:preserve_Tr_b_morphism}
Let $(f_1,f_0):(t,c) \to (t',c')$ be a morphism in $\mathbf{Seg}(\Omega)$. If the relation $f_1(i) \in \mathsf{Tr}_b(t',c')$ holds, then so does the relation $i \in \mathsf{Tr}_b(t,c)$. 
\end{proposition}
\begin{proof}
Recall that, by definition of a morphism in $\mathbf{Seg}(\Omega)$, the inequality $c' \circ f_0 \preceq c$ holds. Now, if the relation $f_1(i) \in \mathsf{Tr}_b(t',c')$ holds, then so do the following pre-order relations.
\[
b \preceq c' \circ t'  \circ f_1(i) = c' \circ f_0 \circ  t (i) \preceq  c \circ  t (i)
\]
By transitivity, we obtain the inequality $b \preceq c \circ  t (i)$, so that $i$ must be in $\mathsf{Tr}_b(t,c)$.
\end{proof}

\begin{proposition}\label{prop:Tr_functor_Set_op}
For every element $b \in \Omega$ and non-negative integer $n_1$, the mapping $(c,t) \mapsto \mathsf{Tr}_b(t,c)$ extends to a functor
$\mathsf{Tr}_b:\mathbf{Seg}(\Omega\,|\,n_1) \to \mathbf{Set}^{\mathrm{op}}$, which factorizes through the opposite category of sub-objects of $[n_1]$.
\end{proposition}
\begin{proof}
By definition, for every segment $(c,t)$ in $\mathbf{Seg}(\Omega\,|\,n_1)$, the set $\mathsf{Tr}_b(t,c)$ is a subset of $[n_1]$. For every morphism $(\mathrm{id},f_0):(t,c) \to (t,c')$ in $\mathbf{Seg}(\Omega\,|\,n_1)$, Proposition \ref{prop:preserve_Tr_b_morphism} shows that there is an inclusion
$\mathsf{Tr}_b(t',c') \subseteq \mathsf{Tr}_b(t,c)$. Since the opposite category of sub-objects of $[n_1]$ is a pre-order category, the statement follows. 
\end{proof}

In fact, Proposition \ref{prop:Tr_functor_Set_op} hides a more general construction if one allows the consideration of the category $\mathbf{Set}_{\ast}$ of pointed sets and point-preserving maps (see Example \ref{exa:explain_pointed_maps}). Recall that there is an adjunction
\[
\xymatrix{
\mathbf{Set} \ar@<+1.2ex>[r]^{\mathbb{F}} \ar@{}[r]|{\bot}\ar@<-1.2ex>@{<-}[r]_{\mathbb{U}} & \mathbf{Set}_{\ast}
}
\]
whose right adjoint $\mathbb{U}:\mathbf{Set}_{\ast} \to \mathbf{Set}$ forgets the pointed structure (i.e. $\mathbb{U}:(X,p) \mapsto X$) and whose left adjoint 
$\mathbb{F}:\mathbf{Set} \to \mathbf{Set}_{\ast}$ maps a set $X$ to the obvious pointed set $(X+\{\ast\},\ast)$ and maps a function $f:X \to Y$ to the coproduct map $f+\{\ast\}:X+\{\ast\} \to Y+\{\ast\}$.

\begin{proposition}\label{prop:Tr_functor_pointed_Set_op}
For every element $b \in \Omega$, the mapping $(c,t) \mapsto \mathbb{F}\mathsf{Tr}_b(t,c)$ extends to a functor $\mathsf{Tr}_b^{\ast}:\mathbf{Seg}(\Omega) \to \mathbf{Set}_{\ast}^{\mathrm{op}}$ mapping every function $(f_1,f_0):(t,c) \to (t',c')$ in $\mathbf{Seg}(\Omega)$ to the following map of pointed sets.
\[
\begin{array}{llllll}
\mathsf{Tr}_b^{\ast}(f_1,f_0)&:&\mathbb{F}\mathsf{Tr}_b(t',c')&\to&\mathbb{F}\mathsf{Tr}_b(t,c)&\\
&&j&\mapsto & i &\textrm{if } \exists i \in \mathsf{Tr}_b(t,c) \,: \, j = f_1(i);\\
&&j&\mapsto & \ast &\textrm{otherwise.}\\
\end{array}
\]
\end{proposition}
\begin{proof}
Follows from Proposition \ref{prop:preserve_Tr_b_morphism}.
\end{proof}

\begin{example}[Truncation operations and pointed sets]\label{exa:explain_pointed_maps}
Let $(\Omega,\preceq)$ be the Boolean pre-ordered set $\{0 \leq 1\}$. If we consider the morphism of segments of Example \ref{exa:Flexibility_morphism}, which is further specifed below, on the left, by using adequate labeling to show how the first segment is mapped to the second one, we can see that the truncation operation $\mathsf{Tr}_1$, displayed on the right, forces us to consider a map of pointed sets.
\[
\begin{array}{cccc}
\xymatrix@C-30pt{
(\mathop{\bullet}\limits^1&\mathop{\bullet}\limits^2&\mathop{\bullet}\limits^3)&(\mathop{\circ}\limits^4&\mathop{\circ}\limits^5)&(\mathop{\bullet}\limits^6&\mathop{\bullet}\limits^7&\mathop{\bullet}\limits^8&\mathop{\bullet}\limits^9)&(\bullet&\bullet&\bullet&\bullet&\bullet)&(\circ&\circ&\circ)&(\circ)
}&\{1,2,3,6,7,8,9,10,11,12,13,14\}&&?\\
\rotatebox[origin=c]{-90}{$\longrightarrow$} &\rotatebox[origin=c]{90}{$\longrightarrow$}&&\rotatebox[origin=c]{90}{$\longrightarrow$}\\
\xymatrix@C-30pt{
(\mathop{\bullet}\limits^1&\mathop{\bullet}\limits^2&\mathop{\bullet}\limits^3&\mathop{\bullet}\limits^{\ast})&(\mathop{\bullet}\limits^{\ast})&(\mathop{\circ}\limits^4&\mathop{\circ}\limits^5&\mathop{\circ}\limits^{\ast}&\mathop{\circ}\limits^{\ast})&(\mathop{\bullet}\limits^6&\mathop{\bullet}\limits^7&\mathop{\bullet}\limits^8&\mathop{\bullet}\limits^9)&(\bullet&\bullet&\bullet&\bullet&\bullet)&(\circ&\circ&\circ)&(\circ)
}&\{1,2,3,10,11,12,\dots,16,17,18\}&\cup&\{4,5\}
\end{array}
\]
\end{example}

\begin{proposition}\label{prop:Tr_ast_on_Chr_Omega_t_equals_F_Tr}
For every element $b \in \Omega$ and non-negative integer $n_1$, the following diagram commutes.
\[
\xymatrix@C-8pt{
*+!R(.4){\mathbf{Seg}(\Omega\,|\,n_1)}\ar[d]_{\mathsf{Tr}_b}\ar[r]^{\subseteq}&*+!L(.4){\mathbf{Seg}(\Omega)}\ar[d]^{\mathsf{Tr}_b^{\ast}}\\
\mathbf{Set}^{\mathrm{op}}\ar[r]_{\mathbb{F}^{\mathrm{op}}}&*+!L(.4){\mathbf{Set}_{\ast}^{\mathrm{op}}}
}
\]
\end{proposition}
\begin{proof}
By definition, if we restrict the functor $\mathsf{Tr}_b^{\ast}:\mathbf{Seg}(\Omega) \to \mathbf{Set}_{\ast}^{\mathrm{op}}$ to the subcategory $\mathbf{Seg}(\Omega\,|\,n_1) \hookrightarrow \mathbf{Seg}(\Omega)$, then every morphism $(t,c) \leq (t,c')$ in $\mathbf{Seg}(\Omega\,|\,n_1)$ is sent to the following map in $\mathbf{Set}_{\ast}$.
\[
\begin{array}{llllll}
\mathsf{Tr}_b^{\ast}(f_1,f_0)&:&\mathbb{F}\mathsf{Tr}_b(t,c')&\to&\mathbb{F}\mathsf{Tr}_b(t,c)&\\
&&j&\mapsto & j & j \in \mathsf{Tr}_b(t,c')\\
&&\ast&\mapsto & \ast &\textrm{otherwise.}\\
\end{array}
\]
This means that the restriction of $\mathsf{Tr}_b^{\ast}$ on $\mathbf{Seg}(\Omega\,|\,n_1)$ can be retrieved from the application of the functor $\mathbb{F}$ on the images of $\mathsf{Tr}_b$.
\end{proof}

\subsection{Example of pedigrads in sets}
In this section, we construct a collection of functors $\mathbf{Seg}(\Omega) \to \mathbf{Set}$ by using any pointed set $(E,\varepsilon)$ and a parameter in $b \in \Omega$ (see Definition \ref{def:set_E_b_varepsilon}). Later on, we will define various classes of cones $\mathcal{W}$ in $\mathbf{Set}$ for which these functors are $\mathcal{W}$-pedigrad.

\begin{convention}[Notation]
In the sequel, the hom-set of a category $\mathcal{C}$ from an object $X$ to an object $Y$ will be denoted as $\mathcal{C}(X,Y)$. For instance, the set of functions from a set $X$ to a set $Y$ will be denoted by $\mathbf{Set}(X,Y)$. Recall that, for any category $\mathcal{C}$, the hom-sets give rise to a functor $\mathcal{C}(\_,\_):\mathcal{C}^{\mathrm{op}}\times \mathcal{C} \to \mathbf{Set}$ called the \emph{hom-functor} \cite[page 27]{MacLane}.
\end{convention}

\begin{definition}[Canonical pedigrads]\label{def:set_E_b_varepsilon}
For every element $b \in \Omega$, we will denote by $E_{b}^{\varepsilon}$ the functor $\mathbf{Seg}(\Omega) \to \mathbf{Set}$ defined as the composition of the following functors.
\[
\xymatrix@C+20pt{
\mathbf{Seg}(\Omega)\ar[r]^{\mathsf{Tr}_b^{\ast}}&\mathbf{Set}_{\ast}^{\mathrm{op}}\ar[rr]^{\mathbf{Set}_{\ast}(\_,(E,\varepsilon))}&&\mathbf{Set}
}
\]
\end{definition}

\begin{remark}\label{rem:E_b_varepsilon_as words_functions}
For every object $(t,c)$ in $\mathbf{Seg}(\Omega)$, an element in $E_b^{\varepsilon}(t,c)$ can be seen as a function of the form $\mathsf{Tr}_b(t,c) \to E$ according to the following series of bijections.
\begin{align*}
E_b^{\varepsilon}(t,c)& = \mathbf{Set}_{\ast}(\mathsf{Tr}_b^{\ast}(t,c),(E,\varepsilon))&\\
&= \mathbf{Set}_{\ast}(\mathbb{F}\mathsf{Tr}_b(t,c),(E,\varepsilon))&(\textrm{Def. of }\mathsf{Tr}_b^{\ast})\\
&\cong \mathbf{Set}(\mathsf{Tr}_b(t,c),\mathbb{U}(E,\varepsilon))&(\mathbb{F} \dashv \mathbb{U})\\
&= \mathbf{Set}(\mathsf{Tr}_b(t,c),E)&(\textrm{Def. of }\mathbb{U})
\end{align*}
Because the set $\mathsf{Tr}_b(t,c)$ is equipped with the natural order of natural numbers, we will represent an element in $E_b^{\varepsilon}(t,c)$ as a word of elements in $E$ (see Example \ref{exa:elements_as_words}).
\end{remark}

\begin{example}[Objects]\label{exa:elements_as_words}
Suppose that $\Omega$ denotes the Boolean pre-ordered set $\{0\leq 1\}$ and let $(E,\varepsilon)$ be the pointed set $\{\mathtt{A},\mathtt{C},\mathtt{G},\mathtt{T},\varepsilon\}$. If we consider the segment
\[
(c,t) = \xymatrix@C-30pt{
(\bullet&\bullet&\bullet)&(\circ&\circ)&(\bullet&\bullet&\bullet&\bullet)&(\bullet&\bullet&\bullet&\bullet&\bullet)&(\circ&\circ&\circ)&(\circ)
}
\]
then the set $E_1^{\varepsilon}(t,c)$ (where $b = 1$) will contain the following words (which have been parenthesized for clarity), among many others.
\[
\begin{array}{c}
\xymatrix@C-30pt{
(\mathtt{A}&\mathtt{G}&\varepsilon)&(\mathtt{T}&\mathtt{C}&\mathtt{A}&\mathtt{A})&(\mathtt{T}&\mathtt{A}&\mathtt{G}&\mathtt{G}&\varepsilon)
}\\
\xymatrix@C-30pt{
(\mathtt{G}&\mathtt{T}&\varepsilon)&(\varepsilon&\varepsilon&\varepsilon&\mathtt{C})&(\mathtt{A}&\mathtt{G}&\mathtt{T}&\mathtt{A}&\mathtt{C})
}\\
\xymatrix@C-30pt{
(\mathtt{T}&\mathtt{A}&\mathtt{A})&(\mathtt{G}&\mathtt{A}&\mathtt{T}&\mathtt{C})&(\mathtt{A}&\mathtt{G}&\mathtt{T}&\mathtt{T}&\mathtt{T})
}\\
\textrm{etc.}\\
\end{array}
\]
\end{example}

\begin{example}[Morphisms]\label{exa:morphisms_as_inclusion_of_words}
Suppose that $\Omega$ denotes the Boolean pre-ordered set $\{0\leq 1\}$ and let $(E,\varepsilon)$ be the pointed set $\{\mathtt{A},\mathtt{C},\mathtt{G},\mathtt{T},\varepsilon\}$. If we consider the morphism of segments given below, in which we use adequate labeling to show how the first segment is included in the second one,
\[
\xymatrix@C-30pt{
(\mathop{\bullet}\limits^1&\mathop{\bullet}\limits^2&\mathop{\bullet}\limits^3)&(\mathop{\circ}\limits^4&\mathop{\circ}\limits^5)&(\mathop{\bullet}\limits^6&\mathop{\bullet}\limits^7&\mathop{\bullet}\limits^8&\mathop{\bullet}\limits^9)&(\mathop{\bullet}\limits^{10\,}&\mathop{\bullet}\limits^{11})
} \to \xymatrix@C-30pt{
(\mathop{\bullet}\limits^1&\mathop{\bullet}\limits^2&\mathop{\bullet}\limits^3&\mathop{\bullet}\limits^{\ast}&\mathop{\bullet}\limits^{\ast})&(\mathop{\circ}\limits^4&\mathop{\circ}\limits^5&\mathop{\circ}\limits^{\ast})&(\mathop{\bullet}\limits^6&\mathop{\bullet}\limits^7&\mathop{\bullet}\limits^8&\mathop{\bullet}\limits^9)&(\mathop{\bullet}\limits^{\ast})&(\mathop{\circ}\limits^{10\,}&\mathop{\circ}\limits^{11})
}
\]
then the function obtained from the application of $E_1^{\varepsilon}$ will have the following mapping rules.
\[
\begin{array}{ccc}
\xymatrix@C-30pt{
(\mathtt{A}&\mathtt{G}&\varepsilon)&(\mathtt{T}&\mathtt{C}&\mathtt{A}&\mathtt{A})&(\mathtt{G}&\mathtt{C})
} &\mapsto &
\xymatrix@C-30pt{
(\mathtt{A}&\mathtt{G}&\varepsilon&\varepsilon&\varepsilon)&(\mathtt{T}&\mathtt{C}&\mathtt{A}&\mathtt{A})&(\varepsilon)
}\\
\xymatrix@C-30pt{
(\mathtt{G}&\mathtt{T}&\varepsilon)&(\varepsilon&\varepsilon&\varepsilon&\mathtt{C})&(\mathtt{T}&\mathtt{A})
} &\mapsto &
\xymatrix@C-30pt{
(\mathtt{G}&\mathtt{T}&\varepsilon&\varepsilon&\varepsilon)&(\varepsilon&\varepsilon&\varepsilon&\mathtt{C})&(\varepsilon)
}\\
\xymatrix@C-30pt{
(\mathtt{T}&\mathtt{A}&\mathtt{A})&(\mathtt{G}&\mathtt{A}&\mathtt{T}&\mathtt{C})&(\mathtt{A}&\mathtt{A})
} &\mapsto& 
\xymatrix@C-30pt{
(\mathtt{T}&\mathtt{A}&\mathtt{A}&\varepsilon&\varepsilon)&(\mathtt{G}&\mathtt{A}&\mathtt{T}&\mathtt{C})&(\varepsilon)
}\\
&\textrm{etc.}&\\
\end{array}
\]
If one restricts oneself to morphisms that only insert symbols $\varepsilon$ and do not turn any black node into white ones, then the mappings associated with this type of morphism can be assimilated to the gap insertion operations used in sequence alignment algorithms (see \cite{Rosenberg} or \cite[Chap. 1, sec. 5]{Mount}) in order to compare two sequences of different lengths \cite{SellersRecog,NeedlemanWunsch}.
\[
\xymatrix@C-30pt{
(\mathop{\bullet}\limits^1&\mathop{\bullet}\limits^2&\mathop{\bullet}\limits^3)&(\mathop{\circ}\limits^4&\mathop{\circ}\limits^5)&(\mathop{\bullet}\limits^6&\mathop{\bullet}\limits^7&\mathop{\bullet}\limits^8&\mathop{\bullet}\limits^9)&(\mathop{\bullet}\limits^{10\,}&\mathop{\bullet}\limits^{11})
} \to \xymatrix@C-30pt{
(\mathop{\bullet}\limits^1&\mathop{\bullet}\limits^2&\mathop{\bullet}\limits^3)&(\mathop{\circ}\limits^4&\mathop{\circ}\limits^5&\mathop{\circ}\limits^{\ast})&(\mathop{\bullet}\limits^6&\mathop{\bullet}\limits^7&\mathop{\bullet}\limits^{\ast}&\mathop{\bullet}\limits^{\ast}&\mathop{\bullet}\limits^8&\mathop{\bullet}\limits^9)&(\mathop{\bullet}\limits^{\ast}&\mathop{\bullet}\limits^{10\,}&\mathop{\bullet}\limits^{11})
}
\]
\[
\begin{array}{ccc}
\xymatrix@C-30pt{
(\mathtt{G}&\mathtt{A}&\mathtt{C})&(\mathtt{A}&\mathtt{T}&\mathtt{T}&\mathtt{C})&(\mathtt{C}&\mathtt{T})
} &\mapsto &
\xymatrix@C-30pt{
(\mathtt{G}&\mathtt{A}&\mathtt{C})&(\mathtt{A}&\mathtt{T}&\varepsilon&\varepsilon&\mathtt{T}&\mathtt{C})&(\varepsilon&\mathtt{C}&\mathtt{T})
}\\
&\textrm{etc.}&\\
\end{array}
\]
\end{example}

\begin{proposition}\label{prop:pedigrad_representable}
For every domain $[n_1]$, the restriction of the functor $E_{b}^{\varepsilon}:\mathbf{Seg}(\Omega) \to \mathbf{Set}$ on $\mathbf{Seg}(\Omega\,|\,n_1)$ is isomorphic to the functor $\mathbf{Set}(\mathsf{Tr}_b(\_),E):\mathbf{Seg}(\Omega\,|\,n_1) \to \mathbf{Set}$. In other words, the following diagram commutes up to an isomorphism of functors.
\[
\xymatrix{
\mathbf{Seg}(\Omega\,|\,n_1)\ar[r]^{\subseteq}\ar[d]_{\mathbf{Tr}_b}&\mathbf{Seg}(\Omega)\ar[d]^{E_{b}^{\varepsilon}}\\
\mathbf{Set}^{\mathrm{op}}\ar[r]_{\mathbf{Set}(\_,E)}&\mathbf{Set}
}
\]
\end{proposition}
\begin{proof}
Note that the following series of isomorphisms hold on $\mathbf{Seg}(\Omega\,|\,n_1)$.
\begin{align*}
E_{b}^{\varepsilon}(\_) & = \mathbf{Set}_{\ast}\Big(\mathsf{Tr}_b^{\ast}(\_),(E,\varepsilon)\Big)&\\
& = \mathbf{Set}_{\ast}\Big(\mathbb{F}\mathsf{Tr}_b(\_),(E,\varepsilon)\Big)&(\textrm{Proposition }\ref{prop:Tr_ast_on_Chr_Omega_t_equals_F_Tr})\\
& \cong \mathbf{Set}\Big(\mathsf{Tr}_b(\_),\mathbb{U}(E,\varepsilon)\Big)&(\mathbb{F} \dashv \mathbb{U})\\
& = \mathbf{Set}\Big(\mathsf{Tr}_b(\_),E\Big)&(\textrm{Def. of }\mathbb{U})
\end{align*}
Because these isomorphisms are natural on $\mathbf{Seg}(\Omega\,|\,n_1)$, the statement follows.
\end{proof}

\subsection{Distributive and exactly distributive chromologies}\label{ssec:Distributive_and_exactly_distributive_chromologies}
Let $(\Omega,\preceq)$ be a pre-ordered set, $b$ be an element in $\Omega$, $A$ be a small category, $\tau$ be an object of $\mathbf{Seg}(\Omega)$ and $\rho:\Delta_{A}(\tau) \Rightarrow \theta$ be a cone in $\mathbf{Seg}(\Omega\,|\,n)$ for some non-negative integer $n$. Note that the application of the truncation functor $\mathsf{Tr}_b:\mathbf{Seg}(\Omega\,|\,n) \to \mathbf{Set}^{\mathrm{op}}$ on the cone $\rho$ gives rise to a cocone in $\mathbf{Set}$ as follows.
\begin{equation}\label{eq:cocone_for_definition_distributive_cones}
\mathsf{Tr}_b(\rho):\mathsf{Tr}_b\theta \Rightarrow \Delta_{A} \circ\mathsf{Tr}_b(\tau)
\end{equation}
According to Proposition \ref{prop:Tr_functor_Set_op}, this cocone can be seen as a diagram in the category of sub-objects of $[n]$. It follows that the colimit adjoint of (\ref{eq:cocone_for_definition_distributive_cones}) in $\mathbf{Set}$ can be factorized, via an epi-mono factorization, through the union of sets $\bigcup_{a \in A}\mathsf{Tr}_b\theta(a)$, as shown below.
\begin{equation}\label{eq:epi_mono_factorization_cocone_for_definition_distributive_cones}
\xymatrix{
\mathsf{colim}_A\mathsf{Tr}_b\theta \ar@/^2pc/[rr]^{\mathsf{lim}_A\mathsf{Tr}_b(\rho)}\ar[r]_-{\textrm{epi.}}^-e&\bigcup_{a \in A}\mathsf{Tr}_b\theta(a) \ar[r]_-{\textrm{mono.}}^-m &  \mathsf{Tr}_b(\tau)
}
\end{equation}

\begin{definition}[Distributive cones]\label{def:distributive_cones}
A cone of the form $\rho:\Delta_{A}(\tau) \Rightarrow \theta$ in $\mathbf{Seg}(\Omega\,|\,n)$  will be said to be \emph{$b$-distributive} if the arrow $m$ of (\ref{eq:epi_mono_factorization_cocone_for_definition_distributive_cones}) is an epimorphism (hence an identity).
\end{definition}

\begin{definition}[Exactly distributive cones]\label{def:exactly_distributive_cones}
A cone of the form $\rho:\Delta_{A}(\tau) \Rightarrow \theta$ in $\mathbf{Seg}(\Omega\,|\,n)$  will be said to be \emph{exactly $b$-distributive} if it is $b$-distributive and the arrow $e$ of (\ref{eq:epi_mono_factorization_cocone_for_definition_distributive_cones}) is a monomorphism (hence an bijection).
\end{definition}

\begin{example}[Homologous cones]\label{exa:Homologous_cones}
Let $\Omega$ denote the Boolean pre-ordered set $\{0 \leq 1\}$. We now give examples of 1-distributive and exactly 1-distributive cones taken in one of the subcategories of homologous segments of $\mathbf{Seg}(\Omega)$. We quickly describe the form of 0-distributive cones at the end of this example. 

First, we can give the following three inequalities as an example of 1-distributive cone in one of the pre-order categories $\mathbf{Seg}(\Omega:t)$ for the obvious topology $t$ of domain $[18]$.
\[
\begin{array}{l}
\xymatrix@C-30pt{
(\bullet&\bullet&\bullet)&(\circ&\circ)&(\bullet&\bullet&\bullet&\bullet)&(\bullet&\bullet&\bullet&\bullet&\bullet)&(\circ&\circ&\circ)&(\circ)
} \leq \xymatrix@C-30pt{
(\circ&\circ&\circ)&(\circ&\circ)&(\bullet&\bullet&\bullet&\bullet)&(\circ&\circ&\circ&\circ&\circ)&(\circ&\circ&\circ)&(\circ)
}\\
\xymatrix@C-30pt{
(\bullet&\bullet&\bullet)&(\circ&\circ)&(\bullet&\bullet&\bullet&\bullet)&(\bullet&\bullet&\bullet&\bullet&\bullet)&(\circ&\circ&\circ)&(\circ)
} \leq \xymatrix@C-30pt{
(\circ&\circ&\circ)&(\circ&\circ)&(\bullet&\bullet&\bullet&\bullet)&(\bullet&\bullet&\bullet&\bullet&\bullet)&(\circ&\circ&\circ)&(\circ)
}\\
\xymatrix@C-30pt{
(\bullet&\bullet&\bullet)&(\circ&\circ)&(\bullet&\bullet&\bullet&\bullet)&(\bullet&\bullet&\bullet&\bullet&\bullet)&(\circ&\circ&\circ)&(\circ)
} \leq \xymatrix@C-30pt{
(\bullet&\bullet&\bullet)&(\circ&\circ)&(\circ&\circ&\circ&\circ)&(\bullet&\bullet&\bullet&\bullet&\bullet)&(\circ&\circ&\circ)&(\circ)
}\\
\end{array}
\]
From a biological point of view, this type of cone could be used to specify various ways of selecting a set of codons or genes (depending on the scale).  

For their part, exactly 1-distributive cones cannot have common black patches that are not related via their diagram $A$. Below is an example of such a cone whose diagram is discrete.
\[
\begin{array}{l}
\xymatrix@C-30pt{
(\bullet&\bullet&\bullet)&(\circ&\circ)&(\bullet&\bullet&\bullet&\bullet)&(\bullet&\bullet&\bullet&\bullet&\bullet)&(\circ&\circ&\circ)&(\circ)
} \leq \xymatrix@C-30pt{
(\circ&\circ&\circ)&(\circ&\circ)&(\bullet&\bullet&\bullet&\bullet)&(\circ&\circ&\circ&\circ&\circ)&(\circ&\circ&\circ)&(\circ)
}\\
\xymatrix@C-30pt{
(\bullet&\bullet&\bullet)&(\circ&\circ)&(\bullet&\bullet&\bullet&\bullet)&(\bullet&\bullet&\bullet&\bullet&\bullet)&(\circ&\circ&\circ)&(\circ)
} \leq \xymatrix@C-30pt{
(\circ&\circ&\circ)&(\circ&\circ)&(\circ&\circ&\circ&\circ)&(\bullet&\bullet&\bullet&\bullet&\bullet)&(\circ&\circ&\circ)&(\circ)
}\\
\xymatrix@C-30pt{
(\bullet&\bullet&\bullet)&(\circ&\circ)&(\bullet&\bullet&\bullet&\bullet)&(\bullet&\bullet&\bullet&\bullet&\bullet)&(\circ&\circ&\circ)&(\circ)
} \leq \xymatrix@C-30pt{
(\bullet&\bullet&\bullet)&(\circ&\circ)&(\circ&\circ&\circ&\circ)&(\circ&\circ&\circ&\circ&\circ)&(\circ&\circ&\circ)&(\circ)
}\\
\end{array}
\]
From a biological point of view, this type of cone could be used to specify which codon or gene (depending on the scale) is separated from the others during homologous recombination.  

On the other hand, the class of $0$-distributive cones in $\mathbf{Seg}(\Omega:t)$ contains all the cones of $\mathbf{Seg}(\Omega:t)$ while the class of exactly $0$-distributive cones in $\mathbf{Seg}(\Omega:t)$ is equal to the set of identities (seen as one-arrow cones) in $\mathbf{Seg}(\Omega:t)$.
\end{example}

\begin{example}[Quasi-homologous cones]\label{exa:Quasi-homologous_cones}
Let $\Omega$ denote the Boolean pre-ordered set $\{0 \leq 1\}$. In this example, we give 1-distributive and exactly 1-distributive cones in one of the categories of quasi-homologous segments of $\mathbf{Seg}(\Omega)$. For instance, the three (obvious) arrows given below (in which the bracketed elements $(\bullet)$ and $(\circ)$ have been shortened, on the left side, to the symbols $\bullet$ and $\circ$, respectively) form a single 1-distributive cone in $\mathbf{Seg}(\Omega\,|\,18)$.
\[
\begin{array}{l}
\xymatrix@C-30pt{
\bullet&\bullet&\bullet&\circ&\circ&\bullet&\bullet&\bullet&\bullet&\bullet&\bullet&\bullet&\bullet&\bullet&\circ&\circ&\circ&\circ
} \leq \xymatrix@C-30pt{
(\circ&\circ)&(\circ)&(\circ&\circ)&(\bullet&\bullet&\bullet&\bullet)&(\circ&\circ&\circ)&(\circ&\circ)&(\circ&\circ&\circ)&(\circ)
}\\
\xymatrix@C-30pt{
\bullet&\bullet&\bullet&\circ&\circ&\bullet&\bullet&\bullet&\bullet&\bullet&\bullet&\bullet&\bullet&\bullet&\circ&\circ&\circ&\circ
}\leq \xymatrix@C-30pt{
(\circ&\circ&\circ)&(\circ&\circ)&(\bullet&\bullet&\bullet&\bullet)&(\bullet&\bullet&\bullet&\bullet&\bullet)&(\circ&\circ&\circ)&(\circ)
}\\
\xymatrix@C-30pt{
\bullet&\bullet&\bullet&\circ&\circ&\bullet&\bullet&\bullet&\bullet&\bullet&\bullet&\bullet&\bullet&\bullet&\circ&\circ&\circ&\circ
} \leq \xymatrix@C-30pt{
(\bullet&\bullet&\bullet)&(\circ&\circ)&(\circ&\circ&\circ&\circ)&(\bullet&\bullet&\bullet&\bullet&\bullet)&(\circ&\circ&\circ)&(\circ)
}\\
\end{array}
\]
For their part, exactly 1-distributive cones cannot have common black patches that are not related via their diagram $A$. Below is an example of such a cone whose diagram is discrete.
\[
\begin{array}{l}
\xymatrix@C-30pt{
\bullet&\bullet&\bullet&\circ&\circ&\bullet&\bullet&\bullet&\bullet&\bullet&\bullet&\bullet&\bullet&\bullet&\circ&\circ&\circ&\circ
} \leq \xymatrix@C-30pt{
(\circ&\circ)&(\circ)&(\circ&\circ)&(\bullet&\bullet&\bullet&\bullet)&(\circ&\circ&\circ)&(\circ&\circ)&(\circ&\circ&\circ)&(\circ)
}\\
\xymatrix@C-30pt{
\bullet&\bullet&\bullet&\circ&\circ&\bullet&\bullet&\bullet&\bullet&\bullet&\bullet&\bullet&\bullet&\bullet&\circ&\circ&\circ&\circ
} \leq \xymatrix@C-30pt{
(\circ&\circ&\circ)&(\circ&\circ)&(\circ&\circ&\circ&\circ)&(\bullet&\bullet&\bullet&\bullet&\bullet)&(\circ&\circ&\circ)&(\circ)
}\\
\xymatrix@C-30pt{
\bullet&\bullet&\bullet&\circ&\circ&\bullet&\bullet&\bullet&\bullet&\bullet&\bullet&\bullet&\bullet&\bullet&\circ&\circ&\circ&\circ
} \leq \xymatrix@C-30pt{
(\bullet&\bullet&\bullet)&(\circ&\circ)&(\circ&\circ)&(\circ&\circ)&(\circ&\circ&\circ&\circ&\circ)&(\circ&\circ&\circ)&(\circ)
}\\
\end{array}
\]
The difference between the cones given in Example \ref{exa:Homologous_cones} and those given above is that the ones given above specify operations that act on sequences that are, \emph{a priori}, not equipped with any particular topology. See the difference in this following example. Suppose that $\mathsf{E}_0$ denotes an enzyme that is to cut a DNA sequence codon by codon. Either one feeds $\mathsf{E}_0$ with a sequence whose topology is already specified, say $(\mathtt{A}\mathtt{T}\mathtt{C})(\mathtt{G}\mathtt{A})$, and $\mathsf{E}_0$ returns $\mathtt{A}\mathtt{T}\mathtt{C}$ and $\mathtt{G}\mathtt{A}$ or one feeds $\mathsf{E}_0$ with a sequence of the form $\mathtt{A}\mathtt{T}\mathtt{C}\mathtt{G}\mathtt{A}$ and only the specification of a particular cone would allow us to know what $\mathsf{E}_0$ returns.
\end{example}

\begin{definition}[Distributive chromologies]\label{def:distributive_chromologies}
For every element $b \in \Omega$, a chromology $(\Omega,D)$ will be said to be \emph{$b$-distributive} if all its cones in $D$ are $b$-distributive.
\end{definition}

\begin{definition}[Exactly distributive chromologies]\label{def:exactly_distributive_chromologies}
For every element $b \in \Omega$, a chromology $(\Omega,D)$ will be said to be \emph{exactly $b$-distributive} if all its cones in $D$ are exactly $b$-distributive.
\end{definition}

\subsection{Logical systems for pedigrads in sets}
In this section, we show that the functor defined in Definition \ref{def:set_E_b_varepsilon} can be seen as a pedigrad for a certain logical system in $\mathbf{Set}$.

\begin{definition}[Logical systems of monomorphisms]\label{def:logical_systems_cones_mono_limits}
We will denote by $\mathcal{W}^{\textrm{inj}}$ the set of cones $\Delta_{A}(X) \Rightarrow F$ in $\mathbf{Set}$ whose limit adjoints $X \to \mathsf{lim}_{A}F$ are injections.
\end{definition}

\begin{proposition}\label{prop:mono-pedigrad_in_set}
For every element $b$ in $\Omega$ and $b$-distributive chromology $(\Omega,D)$, the functor $E_{b}^{\varepsilon}:\mathbf{Seg}(\Omega) \to \mathbf{Set}$ is a $\mathcal{W}^{\textrm{inj}}$-pedigrad for $(\Omega,D)$.
\end{proposition}
\begin{proof}
Let $\rho:\Delta_{A}(\tau) \Rightarrow \theta$ be a cone in $D[n_1]$ for some given non-negative integer $n_1$. By assumption on $(\Omega,D)$ and Definition \ref{def:distributive_cones}, the canonical arrow
\[
\mathsf{colim}_A\mathsf{Tr}_b\theta \to \mathsf{Tr}_b(\tau)
\]
must be an epimorphism, so that its image via the functor $\mathbf{Set}(\_,E):\mathbf{Set}^{\mathrm{op}} \to \mathbf{Set}$ is an injection. By Proposition \ref{prop:pedigrad_representable} and the usual definition of colimits in $\mathbf{Set}$, the resulting injection is (naturally) isomorphic to the following canonical arrow.
\[
E_{b}^{\varepsilon}(\tau) \to \mathsf{lim}_A E_{b}^{\varepsilon}\circ \theta
\]
This precisely shows that $E_{b}^{\varepsilon}:\mathbf{Seg}(\Omega) \to \mathbf{Set}$ is a $\mathcal{W}^{\textrm{inj}}$-pedigrad for $(\Omega,D)$.
\end{proof}

\begin{definition}[Sub-functors]
Recall that a \emph{sub-functor} of a functor $G:D \to \mathbf{Set}$ is a functor $F:D \to \mathbf{Set}$ such that 
\begin{itemize}
\item[1)] for every object $d$ in $D$, the set $F(d)$ is a subset in $G(d)$;
\item[2)] for every morphism $f:d \to d'$ in $D$, the function $F(f):F(d) \to F(d')$ makes the following diagram commute.
\[
\xymatrix{
F(d)\ar[d]_{F(f)}\ar[r]^{\subseteq}&G(d)\ar[d]^{G(f)}\\
F(d')\ar[r]_{\subseteq}&G(d')
}
\]
\end{itemize}
In this case, we will write $F \subseteq G$ to mean that $F$ is a sub-functor of $G$.
\end{definition}

\begin{remark}[Sub-functors are pedigrads]
Every sub-functor $F \subseteq E_{b}^{\varepsilon}$ is a $\mathcal{W}^{\textrm{inj}}$-pedigrad for any $b$-distributive chromology $(\Omega,D)$. This follows from the commutativity of the following diagram for every cone $\rho:\Delta_{A}(\tau) \Rightarrow \theta$ in $D$, which forces the left-most vertical arrow to be an injection.
\[
\xymatrix@C-18pt{
F(\tau)\ar[d]\ar[r]^-{\subseteq}&E_{b}^{\varepsilon}(\tau)\ar[d]^{\mathrm{inj.}}\\
*+!R(.5){\mathsf{lim}_AF \circ \theta}\ar[r]&*+!L(.6){\mathsf{lim}_AE_{b}^{\varepsilon}\circ \theta}
}
\]
\end{remark}

\begin{example}[Living beings]\label{exa:Living_beings_mono}
Let $(\Omega,\preceq)$ be the Boolean pre-ordered set $\{0 \leq 1 \}$ and $(E,\varepsilon)$ be the pointed set $\{\mathtt{A},\mathtt{C},\mathtt{G},\mathtt{T},\varepsilon\}$. We will suppose that $\Omega$ is equipped with a 1-distributive chromology structure $(\Omega,D)$. The sub-functors of $E_{1}^{\varepsilon}:\mathbf{Seg}(\Omega) \to \mathbf{Set}$ could be seen as structures containing the genomes of living beings whose genetic codes are encoded with DNA. Precisely, the fact that the cones of $D$ are sent to 
$\mathcal{W}^{\textrm{inj}}$ means that the DNA strands are \emph{uniquely determined} by the patches living in the codomains of the arrows defining the cones of $D$.
\[
\xymatrix@R-10pt{
&\mathtt{G}\mathtt{G}\mathtt{A}\mathtt{T}\mathtt{A}\mathtt{C}\mathtt{C}\mathtt{G}\mathtt{A}\mathtt{T}\mathtt{T}\mathtt{A}\ar@/_1pc/[ld]\ar[d]\ar@/^1pc/[rd]&\\
(\!-\!-\!-\!)(\mathtt{T}\mathtt{A}\mathtt{C})(\!-\!-\!-\!-\!--\!)&(\mathtt{G}\mathtt{G}\mathtt{A})(\!-\!-\!-\!)(\!-\!-\!-\!-\!--\!)&(\!-\!-\!-\!)(\!-\!-\!-\!)(\mathtt{C}\mathtt{G}\mathtt{A}\mathtt{T}\mathtt{T}\mathtt{A})
}
\]
Note that if two patches $S_1$ and $S_2$ live in a subfunctor $F \subseteq E_{1}^{\varepsilon}$, then their concatenation $S_1 \cdot S_2$ might not exist in the sub-functor $F$ (this makes sense with the fact that if two genes are present in an individual $X$, then the cutting and re-gluing of their internal patches might not exist in $X$). However, the ability of concatenating segments can turn out to be useful if one wants to deal with homologous recombination or even CRISPR. To be able to concatenate DNA strands, we will need the type of pedigrad induced by the logical systems defined in Definition \ref{def:logical_systems_cones_iso_limits}.
\end{example}

\begin{definition}[Logical systems of bijections]\label{def:logical_systems_cones_iso_limits}
We will denote by $\mathcal{W}^{\textrm{bij}}$ the set of cones $\Delta_{A}(X) \Rightarrow F$ in $\mathbf{Set}$ whose limit adjoints $X \to \mathsf{lim}_{A}F$ are bijections.
\end{definition}

\begin{proposition}\label{prop:E_b_varepsilon_W_iso_pedigrad_exactly_distributive}
For every element $b$ in $\Omega$ and exactly $b$-distributive chromology $(\Omega,D)$, the functor $E_{b}^{\varepsilon}:\mathbf{Seg}(\Omega) \to \mathbf{Set}$ is a $\mathcal{W}^{\textrm{bij}}$-pedigrad for $(\Omega,D)$.
\end{proposition}
\begin{proof}
Let $\rho:\Delta_{A}(\tau) \Rightarrow \theta$ be a cone in $D[n_1]$ for some given non-negative integer $n_1$. By assumption on $(\Omega,D)$ and Definitions \ref{def:distributive_cones} \& \ref{def:exactly_distributive_cones}, the canonical arrow
\[
\mathsf{colim}_A\mathsf{Tr}_b\theta \to \mathsf{Tr}_b(\tau)
\]
must be both an epimorphism and a monomorphism (and, in fact, an isomorphism), so that its image via the functor $\mathbf{Set}(\_,E):\mathbf{Set}^{\mathrm{op}} \to \mathbf{Set}$ is a bijection. By Proposition \ref{prop:pedigrad_representable} and the usual definition of colimits in $\mathbf{Set}$, the resulting bijection is (naturally) isomorphic to the following canonical arrow.
\[
E_{b}^{\varepsilon}(\tau) \to \mathsf{lim}_A E_{b}^{\varepsilon}\circ \theta
\]
This precisely shows that $E_{b}^{\varepsilon}:\mathbf{Seg}(\Omega) \to \mathbf{Set}$ is a $\mathcal{W}^{\textrm{bij}}$-pedigrad for $(\Omega,D)$.
\end{proof}

\begin{example}[Algebraic ring of DNA]
The idea behind $\mathcal{W}^{\textrm{bij}}$-pedigrads is that they allow one to model all those operations that an engineer might want to do (such as homologous recombination (see section \ref{sec:Pedigrads_in_semimodules_over_semi-rings}), CRISPR \cite{Pennisi} (see Example \ref{exa:CRISPR_iso_pedigrad} below), DNA sequencing, alignment methods, etc.), but that might not exist in a living being. Therefore $\mathcal{W}^{\textrm{bij}}$-pedigrads should be seen as universes in which it is possible to think about DNA, instead of being seen as models containing the genome of a particular set of living beings.
\end{example}

\begin{example}[Duplication]\label{exa:duplication_iso_pedigrad}
Let $(\Omega,\preceq)$ be the Boolean pre-ordered set $\{0 \leq 1 \}$ and $(E,\varepsilon)$ be the pointed set $\{\mathtt{A},\mathtt{C},\mathtt{G},\mathtt{T},\varepsilon\}$. In this example, we illustrate how flexible $\mathcal{W}^{\textrm{bij}}$-pedigrads are by constructing what one could see as a duplication mutation, which is a type of mutation that is responsible for triggering certain cancers \cite{ReamsRoth}. To do so, first note that the category $\mathbf{Seg}(\Omega)$ contains two morphisms as follows (in which the labeled elements $\bullet$ specify how the domain is sent to the codomain).
\[
\xymatrix@C-30pt@R-15pt{
(\mathop{\bullet}\limits^1&\mathop{\bullet}\limits^2&\mathop{\bullet}\limits^3)\ar[rr]^{f_1}&\quad\quad&(\mathop{\bullet}\limits^1&\mathop{\bullet}\limits^2&\mathop{\bullet}\limits^3)&(\circ&\circ&\circ)
}
\quad\quad\quad
\xymatrix@C-30pt@R-15pt{
(\mathop{\bullet}\limits^1&\mathop{\bullet}\limits^2&\mathop{\bullet}\limits^3)\ar[rr]^{f_2}&\quad\quad&(\circ&\circ&\circ)&(\mathop{\bullet}\limits^1&\mathop{\bullet}\limits^2&\mathop{\bullet}\limits^3)
}
\]
The images of $f_1$ and $f_2$ via the functor $E_1^{\varepsilon}$ here turn out to be identities: they map any word to the same copy of that word; e.g. $\mathtt{A}\mathtt{T}\mathtt{G}\mapsto \mathtt{A}\mathtt{T}\mathtt{G}$. Now, since the functor $E_1^{\varepsilon}$ is a $\mathcal{W}^{\textrm{bij}}$-pedigrad, the image of the obvious exactly 1-distributive cone
\begin{equation}\label{eq:cone_duplication_example}
\begin{array}{c}
\xymatrix@C-30pt@R-30pt{
&&&&&&&(\bullet&\bullet&\bullet)&(\circ&\circ&\circ)\\
(\bullet&\bullet&\bullet)&(\bullet&\bullet&\bullet)\ar[rru]\ar[rrd]&\quad\quad\quad&&&&&&\\
&&&&&&&(\circ&\circ&\circ)&(\bullet&\bullet&\bullet)
}
\end{array}
\end{equation}
via the functor $E_1^{\varepsilon}$ is associated with a concatenation operation of the following form (this is the inverse of the limit adjoint of the image of cone (\ref{eq:cone_duplication_example}) via the functor $E_1^{\varepsilon}$).
\[
\mu:E_1^{\varepsilon}(\!\!\xymatrix@C-30pt{(\bullet&\bullet&\bullet)&(\circ&\circ&\circ)}\!\!)\times E_1^{\varepsilon}(\!\!\xymatrix@C-30pt{(\circ&\circ&\circ)&(\bullet&\bullet&\bullet)}\!\!) \longrightarrow E_1^{\varepsilon}(\!\!\xymatrix@C-30pt{(\bullet&\bullet&\bullet)&(\bullet&\bullet&\bullet)}\!\!)
\]
Here, by concatenation, we mean a map that concatenates words; e.g.
the pair of words $(\mathtt{A}\mathtt{T}\mathtt{G},\mathtt{C}\mathtt{G}\mathtt{G})$ will be sent to the word $\mathtt{A}\mathtt{T}\mathtt{G}\mathtt{C}\mathtt{G}\mathtt{G}$. The post-composition of the morphism $\mu$ with the pairing of morphisms $(E_1^{\varepsilon}(f_1), E_1^{\varepsilon}(f_2))$ then gives rise to a map of the type given below that duplicates any word contained in its domain; e.g. $\mathtt{A}\mathtt{T}\mathtt{G} \mapsto \mathtt{A}\mathtt{T}\mathtt{G}\mathtt{A}\mathtt{T}\mathtt{G}$.
\[
E_1^{\varepsilon}(\!\!\xymatrix@C-30pt{(\bullet&\bullet&\bullet)}\!\!)\to E_1^{\varepsilon}(\!\!\xymatrix@C-30pt{(\bullet&\bullet&\bullet)&(\bullet&\bullet&\bullet)}\!\!)
\]
\end{example}

\begin{example}[CRISPR]\label{exa:CRISPR_iso_pedigrad}
Let $(\Omega,\preceq)$ be the Boolean pre-ordered set $\{0 \leq 1 \}$ and $(E,\varepsilon)$ be the pointed set $\{\mathtt{A},\mathtt{C},\mathtt{G},\mathtt{T},\varepsilon\}$. In this example, we illustrate how $\mathcal{W}^{\textrm{bij}}$-pedigrads can be used to model the CRISPR technology. Recall that CRISPR is a tool that allows one to edit a DNA segment. For illustration,  let us take the segment $\mathtt{ATCGTC}$, which is supposed to live the set given below, on the left.
\[
E_1^{\varepsilon}(\!\!\xymatrix@C-30pt{(\bullet&\bullet)&(\bullet&\bullet&\bullet)&(\bullet)}\!\!)\quad\quad\quad\quad\quad E_1^{\varepsilon}(\!\!\xymatrix@C-30pt{(\circ&\circ)&(\bullet&\bullet&\bullet)&(\circ)}\!\!)
\]
Now, suppose that we would like to replace the patch $\mathtt{CGT}$ of our segment with the sequence $\mathtt{TTC}$, which lives in the set given above, on the right. To do so, we would need an algebraic operation that first cuts the patch that one wants to replace and then use a concatenation operation to insert $\mathtt{TTC}$ in the missing (or crossed out) part of $\mathtt{AT}$(\st{$\mathtt{CGT}$})$\mathtt{C}$. In more categorical terms, we would need to use the arrow resulting from the composition of the pair of arrows
\[
\begin{array}{clc}
E_1^{\varepsilon}(\!\!\xymatrix@C-30pt{(\bullet&\bullet)&(\bullet&\bullet&\bullet)&(\bullet)}\!\!) \times E_1^{\varepsilon}(\!\!\xymatrix@C-30pt{(\circ&\circ)&(\bullet&\bullet&\bullet)&(\circ)}\!\!) & \longrightarrow & E_1^{\varepsilon}(\!\!\xymatrix@C-30pt{(\bullet&\bullet)&(\circ&\circ&\circ)&(\bullet)}\!\!) \times E_1^{\varepsilon}(\!\!\xymatrix@C-30pt{(\circ&\circ)&(\bullet&\bullet&\bullet)&(\circ)}\!\!)\\\vspace{-18pt}
\\\vspace{-4pt}
&&\xymatrix@R-7pt{\ar[d]^{\mu}\\~}\\
 &  & E_1^{\varepsilon}(\!\!\xymatrix@C-30pt{(\bullet&\bullet&\bullet&\bullet&\bullet&\bullet)}\!\!) 
\end{array}
\]
where the horinzontal arrow, at the top, is that induced by the obvious arrow of $\mathbf{Seg}(\Omega)$ that turns three black nodes into white ones and where the vertical arrow, on the right, is the inverse of the limit adjoint of the image of the exactly 1-distributive cone given below.
\[
\xymatrix@C-30pt@R-30pt{
&&&&&&&(\bullet&\bullet)&(\circ&\circ&\circ)&(\bullet)\\
(\bullet&\bullet&\bullet&\bullet&\bullet&\bullet)\ar[rru]\ar[rrd]&\quad\quad\quad&&&&&&\\
&&&&&&&(\circ&\circ)&(\bullet&\bullet&\bullet)&(\circ)
}
\]
On can check that the image of the pair $(\mathtt{ATCGTC},\mathtt{TTC})$ through the resulting composition is the segment $\mathtt{ATTTCC}$.
\end{example}

\subsection{Morphisms of pedigrads}
According to section \ref{ssec:Morphisms_of_pedigrads}, a morphism between two pedigrads is a natural transformation between the underlying functors of the pedigrads. The goal of this section is to show how morphisms of pedigrads can be used to model transcription and mutation processes. We shall mainly use Remark \ref{rem:morhisms_of_pedigrads_in_set}.

\begin{remark}\label{rem:morhisms_of_pedigrads_in_set}
Let $(\Omega,\preceq)$ be a pre-ordered set and $b$ be an element in $\Omega$. For every map of pointed sets $f:(A,\alpha) \to (B,\beta)$ in $\mathbf{Set}_{\ast}$, there is a natural transformation $f_b^{\ast}:A_b^{\alpha} \Rightarrow B_b^{\beta}$ given by the evaluation of the bifunctor $\mathbf{Set}_{\ast}(\mathsf{Tr}_b^{\ast}(\_),\_)$ at $f$ on the second variable.
\[
\mathbf{Set}_{\ast}(\mathsf{Tr}_b^{\ast}(\_),f):\mathbf{Set}_{\ast}(\mathsf{Tr}_b^{\ast}(\_),(A,\alpha)) \Rightarrow \mathbf{Set}_{\ast}(\mathsf{Tr}_b^{\ast}(\_),(B,\beta))
\]
The naturality of the resulting transformation $f_b^{\ast}:A_b^{\alpha} \Rightarrow B_b^{\beta}$ directly follows from the functoriality of the bifunctor in the first variable.
\end{remark}

\begin{example}[Transcription]\label{exa:Transcription_in_set}
For convenience, let $(\Omega,\preceq)$ be the Boolean pre-ordered set $\{0 \leq 1\}$ and take $b$ to be $1 \in \Omega$. If we denote by $(A,\varepsilon)$ and $(B,\varepsilon)$ the pointed sets given by the alphabets $\{\mathtt{A},\mathtt{C},\mathtt{G},\mathtt{T},\varepsilon\}$ and $\{\mathtt{A},\mathtt{C},\mathtt{G},\mathtt{U},\varepsilon\}$, respectively, then we can construct an isomorphism $f:(A,\varepsilon) \to (B,\varepsilon)$ in $\mathbf{Set}_{\ast}$ by considering the mappings $\mathtt{A} \mapsto \mathtt{U}$; $\mathtt{T} \mapsto \mathtt{A}$; $\mathtt{G} \mapsto \mathtt{C}$; $\mathtt{C} \mapsto \mathtt{G}$. The morphism of pedigrads $f_b^{\ast}:A_b^{\varepsilon} \Rightarrow B_b^{\varepsilon}$ resulting from Remark \ref{rem:morhisms_of_pedigrads_in_set} then sends any words of the form given below, on the left, to the corresponding words, on the right.
\[
\begin{array}{ccc}
A_b^{\varepsilon}(\!\!\xymatrix@C-30pt{(\bullet&\bullet&\bullet)&(\bullet&\bullet&\bullet)&(\bullet&\bullet&\bullet)}\!\!) &\to &B_b^{\varepsilon}(\!\!\xymatrix@C-30pt{(\bullet&\bullet&\bullet)&(\bullet&\bullet&\bullet)&(\bullet&\bullet&\bullet)}\!\!)\\
\mathtt{T}\mathtt{G}\mathtt{T}\mathtt{A}\mathtt{G}\mathtt{T}\mathtt{A}\mathtt{G}\mathtt{C}&\mapsto&\mathtt{A}\mathtt{C}\mathtt{A}\mathtt{U}\mathtt{C}\mathtt{A}\mathtt{U}\mathtt{C}\mathtt{G}\\
\mathtt{A}\mathtt{A}\mathtt{A}\mathtt{C}\mathtt{T}\mathtt{T}\mathtt{A}\mathtt{C}\mathtt{A}&\mapsto&\mathtt{U}\mathtt{U}\mathtt{U}\mathtt{G}\mathtt{A}\mathtt{A}\mathtt{U}\mathtt{G}\mathtt{U}\\
\end{array}
\]
As can be seen, this type of transformation models RNA transcription by sending every nucleobase to its RNA anti-nucleobase (where $\mathtt{U}$ (uracil) stands for the anti-nucleobase of adenine ($\mathtt{A}$)).
\end{example}

\begin{remark}[Coarse mutations]
As one can imagine, the type of morphism constructed in Remark \ref{rem:morhisms_of_pedigrads_in_set} can be used to represent DNA mutations. However, note that if one restricts oneself to morphisms of pointed sets of the form $(E,\varepsilon) \to (E,\varepsilon)$ where $(E,\varepsilon)$ is our usual pointed set $\{\mathtt{A},\mathtt{C},\mathtt{G},\mathtt{T},\varepsilon\}$, then these mutations will obviously be too systematic to be realistic. For instance, see the example given below for the mappings $\mathtt{A} \mapsto \varepsilon$; $\mathtt{T} \mapsto \mathtt{T}$; $\mathtt{G} \mapsto \mathtt{G}$; $\mathtt{C} \mapsto \mathtt{A}$, which delete any adenine ($\mathtt{A}$) and substitute any cytosine ($\mathtt{C}$) with an adenine ($\mathtt{A}$).
\[
\begin{array}{ccc}
E_1^{\varepsilon}(\!\!\xymatrix@C-30pt{(\bullet&\bullet&\bullet)&(\bullet&\bullet&\bullet)&(\bullet&\bullet&\bullet)}\!\!) &\to &E_1^{\varepsilon}(\!\!\xymatrix@C-30pt{(\bullet&\bullet&\bullet)&(\bullet&\bullet&\bullet)&(\bullet&\bullet&\bullet)}\!\!)\\
\mathtt{A}\mathtt{G}\mathtt{C}\mathtt{A}\mathtt{G}\mathtt{T}\mathtt{A}\mathtt{G}\mathtt{C}&\mapsto&\varepsilon\mathtt{G}\mathtt{A}\varepsilon\mathtt{G}\mathtt{T}\varepsilon\mathtt{G}\mathtt{A}\\
\mathtt{T}\mathtt{A}\mathtt{A}\mathtt{C}\mathtt{C}\mathtt{T}\mathtt{A}\mathtt{C}\mathtt{A}&\mapsto&\mathtt{T}\varepsilon\varepsilon\mathtt{A}\mathtt{A}\mathtt{T}\varepsilon\mathtt{A}\mathtt{A}\\
\end{array}
\]
Note that a parameterization of the alphabet could specify the context in which these mutations happen, which would make them more realistic (see Example \ref{exa:Mutations_in_set_are_spans}).

Finally, note that the definition of $E_b^{\varepsilon}$ forces us to always map the element $\varepsilon$ to itself, so that insertion mutations cannot directly be viewed as morphisms of the previous type. In fact, insertion mutations could be encoded as the fibers (i.e. lifts) of a morphism of the form described above. This idea is further discussed in Example \ref{exa:Mutations_in_set_are_spans} via the concept of span.
\end{remark}

\begin{example}[Mutations are spans]\label{exa:Mutations_in_set_are_spans}
The present example shows that mutations can be recovered from spans in the underlying category of pedigrads. Let $(\Omega,\preceq)$ be the Boolean pre-ordered set $\{0 \leq 1\}$ and take $b$ to be $1 \in \Omega$. Denote by $(E,\varepsilon)$ the pointed set given by the alphabet $\{\mathtt{A},\mathtt{C},\mathtt{G},\mathtt{T},\varepsilon\}$. Since $\mathbf{Set}_{\ast}$ is complete, we can form the Cartesian product of $(E,\varepsilon)$, which, for convenience, will be denoted as $(A,(\varepsilon,\varepsilon))$, where $A$ is the set of pairs of elements in $E$; e.g. $(\mathtt{A},\varepsilon)$, $(\mathtt{A},\mathtt{T})$, etc. By definition, we are given a span of projection maps $(A,(\varepsilon,\varepsilon)) \rightrightarrows (E,\varepsilon)$ in $\mathbf{Set}_*$ that forget the second and first components of the pairs in $A$, respectively. The corresponding span resulting from Remark \ref{rem:morhisms_of_pedigrads_in_set} can then be viewed as a binary relation describing all the possible ways a DNA strands can be mutated (the binary relation appears when reading the span of mappings from left to right, as shown below).
\[
\begin{array}{ccccc}
E_b^{\varepsilon}(\!\!\xymatrix@C-30pt{(\bullet&\bullet&\bullet)&(\bullet&\bullet&\bullet) &(\bullet&\bullet&\bullet)}\!\!)& \leftarrow &A_b^{(\varepsilon,\varepsilon)}(\!\!\xymatrix@C-30pt{(\bullet&\bullet&\bullet)&(\bullet&\bullet&\bullet)&(\bullet&\bullet&\bullet)}\!\!)&\rightarrow&E_b^{\varepsilon}(\!\!\xymatrix@C-30pt{(\bullet&\bullet&\bullet)&(\bullet&\bullet&\bullet)&(\bullet&\bullet&\bullet)}\!\!)\vspace{-9pt}\\
&&&&\\
\mathtt{T}\mathtt{G}\mathtt{C}\mathtt{A}\mathtt{G}\varepsilon\mathtt{A}\mathtt{G}\varepsilon&\rotatebox[origin=c]{180}{$\mapsto$}&\binom{\mathtt{T}}{\mathtt{T}}\binom{\mathtt{G}}{\mathtt{G}}\binom{\mathtt{C}}{\mathtt{C}}\binom{\mathtt{A}}{\mathtt{A}}\binom{\mathtt{G}}{\mathtt{G}}\binom{\varepsilon}{\mathtt{T}}\binom{\mathtt{A}}{\mathtt{A}}\binom{\mathtt{G}}{\mathtt{C}}\binom{\varepsilon}{\varepsilon}&\mapsto&\mathtt{T}\mathtt{G}\mathtt{C}\mathtt{A}\mathtt{G}\mathtt{T}\mathtt{A}\mathtt{C}\varepsilon\vspace{-4pt}\\
&&&&\\
\mathtt{T}\mathtt{G}\mathtt{C}\mathtt{A}\mathtt{G}\varepsilon\mathtt{A}\mathtt{G}\varepsilon&\rotatebox[origin=c]{180}{$\mapsto$}&\binom{\mathtt{T}}{\mathtt{A}}\binom{\mathtt{G}}{\varepsilon}\binom{\mathtt{C}}{\mathtt{C}}\binom{\mathtt{A}}{\varepsilon}\binom{\mathtt{G}}{\mathtt{G}}\binom{\varepsilon}{\mathtt{A}}\binom{\mathtt{A}}{\mathtt{A}}\binom{\mathtt{G}}{\mathtt{G}}\binom{\varepsilon}{\mathtt{C}}&\mapsto&\mathtt{A}\varepsilon\mathtt{C}\varepsilon\mathtt{G}\mathtt{A}\mathtt{A}\mathtt{G}\mathtt{C}\\
\end{array}
\]
Note that the previous mappings describe the three most common types of mutation, namely deletion, insertion and substitution mutations.
We will finally conclude by pointing out that another span $A \rightrightarrows E$ could have been chosen so that the mutations could have depended on some environmental contexts or other similar conditions.
%, the punchline being that mutations can be studied from the point of view of spans in a category of pedigrads.
\end{example}

\subsection{Inversible pedigrads}
Let $(\Omega,\preceq)$ be a pre-ordered set. The goal of this section is to show how one can use a category of pedigrads to model certain biological phenomena, such as the phenomenon of inversion. %we will show that inversions of DNA patches can be encoded in terms of algebra structures for a certain monad.

\begin{convention}[Notation]
For every non-negative integer $n$, we will denote by $\mathsf{rv}_n$ the function $[n] \to [n]$ sending every element $x \in [n]$ to the interger $n+1-x$ in $[n]$.
\end{convention}

\begin{example}
The following picture shows what the map $\mathsf{rv}_6:[6] \to [6]$ looks like.
\[
\xymatrix@C-18pt@R-15pt{
1\ar@{|->}[d]\ar@{}[r]|{\leq}&2\ar@{|->}[d]\ar@{}[r]|{\leq}&3\ar@{|->}[d]\ar@{}[r]|{\leq}&4\ar@{|->}[d]\ar@{}[r]|{\leq}&5\ar@{|->}[d]\ar@{}[r]|{\leq}&6\ar@{|->}[d]\\
6\ar@{}[r]|{\geq}&5\ar@{}[r]|{\geq}&4\ar@{}[r]|{\geq}&3\ar@{}[r]|{\geq}&2\ar@{}[r]|{\geq}&1
}
\]
This type of map reverses the order of the finite set on which it is defined.
\end{example}

\begin{remark}
For every non-negative integer $n$, the function $\mathsf{rv}_n:[n] \to [n]$ is an involution\footnote{i.e. a bijection that is its own inverse.}. 
\end{remark}

\begin{definition}[Inversion operation]
For every object $(t,c):[n_1] \multimap [n_0]$ in $\mathbf{Seg}(\Omega)$, we define the \emph{inversion} of $(t,c)$ as the segment of type $[n_1] \multimap [n_0]$, over $\Omega$, that consists of the order-preserving surjection $\mathsf{rv}_{n_0}\circ t \circ \mathsf{rv}_{n_1}:[n_1] \to [n_0]$ 
and the function $c \circ \mathsf{rv}_{n_0}:[n_0] \to \Omega$. The inversion  of $(t,c)$ will be denoted as $(t,c)^{\dagger}$.
\end{definition}

\begin{example}
Let $(\Omega,\preceq)$ be the Boolean pre-ordered set $\{0 \leq 1\}$. The inversion of the segment, over $\Omega$, displayed below, on the left, is shown on the right.
\[
(c,t)=\xymatrix@C-30pt{(\bullet&\bullet)&(\circ)&(\bullet&\bullet&\bullet)&(\bullet)&(\circ)&(\circ)}
\quad\quad\mapsto\quad\quad
(t,c)^{\dagger}=\xymatrix@C-30pt{(\circ)&(\circ)&(\bullet)&(\bullet&\bullet&\bullet)&(\circ)&(\bullet&\bullet)}
\]
\end{example}

\begin{convention}[Inversion functor]
The mapping $(t,c) \mapsto (t,c)^{\dagger}$ defines an obvious endo\-functor on $\mathbf{Seg}(\Omega)$ mapping a morphism of segments $(f_1,f_0)$ to the morphism of segments $(f_1^{\dagger},f_0^{\dagger})$ defined by the pair $(\mathsf{rv}_{n_1'} \circ f_1 \circ \mathsf{rv}_{n_1},\mathsf{rv}_{n_0'} \circ f_0 \circ \mathsf{rv}_{n_0})$, as shown below.
\[
\begin{array}{c}
\xymatrix{
[n_1]\ar@{->>}[r]^{t}\ar@{)->}[d]_{f_1}&[n_0]\ar[d]^{f_0}\\
[n_1']\ar@{->>}[r]^{t'}&[n_0']
}
\end{array}
\quad\quad\mapsto\quad\quad
\begin{array}{c}
\xymatrix{
[n_1]\ar@{)->}[d]_{f_1^{\dagger}}\ar[r]^{\mathsf{rv}_{n_1}}&[n_1]\ar@{->>}[r]^{t}\ar@{)->}[d]_{f_1}&[n_0]\ar[d]^{f_0}\ar[r]^{\mathsf{rv}_{n_0}}&[n_0]\ar[d]^{f_0^{\dagger}}\\
[n_1']\ar[r]_{\mathsf{rv}_{n_1'}}&[n_1']\ar@{->>}[r]^{t'}&[n_0']\ar[r]_{\mathsf{rv}_{n_0'}}&[n_0']
}
\end{array}
\]
This endofunctor will be denoted as $\mathsf{Inv}:\mathbf{Seg}(\Omega) \to \mathbf{Seg}(\Omega)$ and called the \emph{inversion functor} on $\mathbf{Seg}(\Omega)$.
\end{convention}

\begin{remark}[Involution]
The inversion functor on $\mathbf{Seg}(\Omega)$ is an involution functor and hence an isomorphism of categories.
\end{remark}

\begin{definition}[Inversible chromologies]
A chromology $(\Omega,D)$ will be said to be \emph{inversible} if the image of every cone in $D$ via the inversion functor is a cone in $D$.
\end{definition}

\begin{definition}[Inversible pedigrads]
Let $(\mathcal{C},\mathcal{W})$ be a logical system. A $\mathcal{W}$-pedigrad $X$ on an inversible chromology $(\Omega,D)$ will be said to be \emph{inversible} if it is equipped with an isomorphism $X \Rightarrow X \circ \mathsf{Inv}$ in $[\mathbf{Seg}(\Omega),\mathcal{C}]$.
\end{definition}

\begin{example}
Let $(\Omega,\preceq)$ be a the Boolean pre-ordered set $\{0 \leq 1\}$. For every element $b \in \Omega$, the functor $E_b^{\varepsilon}:\mathbf{Seg}(\Omega) \to \mathbf{Set}$ defines a inversible $\mathcal{W}^{\textrm{bij}}$-pedigrad on any inversible chromology when it is equipped with the natural transformation that maps every word $\mathtt{X}_1\mathtt{X}_2\dots\mathtt{X}_{n}$ in the domain to its inverse $\mathtt{X}_{n}\dots\mathtt{X}_2\mathtt{X}_{1}$ in the codomain.
\[
E_b^{\varepsilon}(t,c) \to E_b^{\varepsilon}\circ \mathsf{Inv}(t,c)
\]
For instance, we would have the following mappings in the case where $E$ is our usual pointed set $\{\mathtt{A},\mathtt{C},\mathtt{G},\mathtt{T},\varepsilon\}$ and $b$ is taken to be $1 \in \Omega$.
\[
\begin{array}{ccc}
E_1^{\varepsilon}(\!\!\xymatrix@C-30pt{(\bullet&\bullet)&(\bullet)&(\bullet&\bullet&\bullet)}\!\!) &\to &E_1^{\varepsilon}(\!\!\xymatrix@C-30pt{(\bullet&\bullet&\bullet)&(\bullet)&(\bullet&\bullet)}\!\!)\\
\mathtt{A}\mathtt{G}\mathtt{T}\mathtt{A}\mathtt{G}\mathtt{C}&\mapsto&\mathtt{C}\mathtt{G}\mathtt{A}\mathtt{T}\mathtt{G}\mathtt{A}\\
\mathtt{C}\mathtt{T}\mathtt{T}\mathtt{A}\mathtt{C}\mathtt{A}&\mapsto&\mathtt{A}\mathtt{C}\mathtt{A}\mathtt{T}\mathtt{T}\mathtt{C}\\
\end{array}
\quad\quad\quad
\begin{array}{ccc}
E_1^{\varepsilon}(\!\!\xymatrix@C-30pt{(\circ&\circ&\circ)&(\bullet)&(\circ&\circ)}\!\!) &\to &E_1^{\varepsilon}(\!\!\xymatrix@C-30pt{(\circ&\circ)&(\bullet)&(\circ&\circ&\circ)}\!\!)\\
\mathtt{A}&\mapsto&\mathtt{A}\\
\mathtt{C}&\mapsto&\mathtt{C}\\
\end{array}
\]

\end{example}

\begin{remark}[Inversible structure]
The inversible structure of an inversible $\mathcal{W}^{\textrm{bij}}$-pedigrad $X$ can be uniquely determined by the cones of the chromology $(\Omega,D)$ provided that $D$ contains enough cones. For instance, the components of the inversible structure defined at atomic patches (as shown below) will uniquely determine the entire inversible structure if these atomic patches are part of the cones in $D$ (see Example \ref{exa:Living_beings_mono})
\[
X(\!\!\xymatrix@C-30pt{\circ&\circ&\circ&\circ&\circ&\circ&\bullet&\circ&\circ}\!\!) \to X(\!\!\xymatrix@C-30pt{\circ&\circ&\bullet&\circ&\circ&\circ&\circ&\circ&\circ}\!\!)
\]
\end{remark}

\section{Pedigrads in semimodules over semi-rings}\label{sec:Pedigrads_in_semimodules_over_semi-rings}

In this section, we use our pedigrad $E_{b}^{\varepsilon}:\mathbf{Seg}(\Omega) \to \mathbf{Set}$ to generate a new functor in the category of semimodules over the Boolean semi-ring with two elements. We then consider a quotient of this functor to define a pedigrad that models homologous recombination and haplotypes.

A question one might want to ask at this point is why should one consider pedigrads in a category of semimodules specifically? The answer can be given in several points.
\vspace{1pt}

1) The first reason is that the tools of semimodule theory allow us to neatly capture the phenomenon of homologous recombination. When we look at the homologous recombination of two segments, say $\mathtt{x}\mathtt{x}\mathtt{a}\mathtt{x}\mathtt{x}\mathtt{b}\mathtt{x}$ and $\mathtt{x}\mathtt{x}\mathtt{A}\mathtt{x}\mathtt{x}\mathtt{B}\mathtt{x}$, we can observe that each one of these will be separated into two parts, say $\mathtt{x}\mathtt{x}\mathtt{a}\mathtt{x}$, $\mathtt{x}\mathtt{x}\mathtt{A}\mathtt{x}$ and $\mathtt{x}\mathtt{b}\mathtt{x}$, $\mathtt{x}\mathtt{B}\mathtt{x}$, so that the recombination of $\mathtt{x}\mathtt{x}\mathtt{a}\mathtt{x}\mathtt{x}\mathtt{b}\mathtt{x}$ and $\mathtt{x}\mathtt{x}\mathtt{A}\mathtt{x}\mathtt{x}\mathtt{B}\mathtt{x}$ is one of the following re-attachments.
\begin{equation}\label{eq:result_of_recombination}
\mathtt{x}\mathtt{x}\mathtt{a}\mathtt{x}\mathtt{x}\mathtt{b}\mathtt{x}\quad;\quad\quad \mathtt{x}\mathtt{x}\mathtt{a}\mathtt{x}\mathtt{x}\mathtt{B}\mathtt{x}\quad;\quad\quad \mathtt{x}\mathtt{x}\mathtt{A}\mathtt{x}\mathtt{x}\mathtt{b}\mathtt{x}\quad;\quad\quad \mathtt{x}\mathtt{x}\mathtt{A}\mathtt{x}\mathtt{x}\mathtt{B}\mathtt{x}
\end{equation}

\[
\begin{array}{c}
\includegraphics[height=3.5cm]{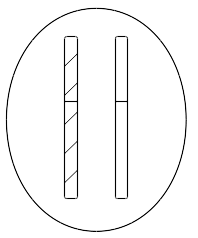} 
\end{array}
\longrightarrow
\begin{array}{c}
\includegraphics[width=5cm]{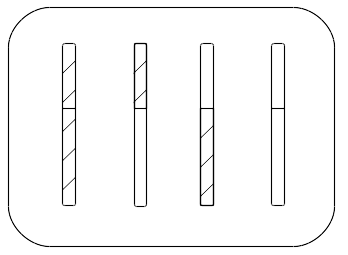} 
\end{array}
\]

Interestingly, the way the different recombined segments are written in (\ref{eq:result_of_recombination}) is reminiscent of the result of a bilinear operation between the sums $\mathtt{x}\mathtt{x}\mathtt{a}\mathtt{x}+\mathtt{x}\mathtt{x}\mathtt{A}\mathtt{x}$ and $\mathtt{x}\mathtt{b}\mathtt{x}+\mathtt{x}\mathtt{B}\mathtt{x}$.
\[
(\mathtt{x}\mathtt{x}\mathtt{a}\mathtt{x}+\mathtt{x}\mathtt{x}\mathtt{A}\mathtt{x})\otimes(\mathtt{x}\mathtt{b}\mathtt{x}+\mathtt{x}\mathtt{B}\mathtt{x}) = \mathtt{x}\mathtt{x}\mathtt{a}\mathtt{x}\otimes\mathtt{x}\mathtt{b}\mathtt{x}+ \mathtt{x}\mathtt{x}\mathtt{a}\mathtt{x}\otimes\mathtt{x}\mathtt{B}\mathtt{x}+ \mathtt{x}\mathtt{x}\mathtt{A}\mathtt{x}\otimes\mathtt{x}\mathtt{b}\mathtt{x}+ \mathtt{x}\mathtt{x}\mathtt{A}\mathtt{x}\otimes\mathtt{x}\mathtt{B}\mathtt{x}
\]
This suggests us that homologous recombination may belong to an algebraic setting in which the mating operation is given by an addition operation while the possible results of homologous recombination are given by a product operation. %The fact that this bilinear operation needs to map the pairs $(a_1,a_2)$ and $(b_1,b_2)$ to their `ancestors' $a$ and $b$ suggests that it  might also be a retraction map. Finally, 
The way one passes from the segments $\mathtt{x}\mathtt{x}\mathtt{a}\mathtt{x}\mathtt{x}\mathtt{b}\mathtt{x}$ and $\mathtt{x}\mathtt{x}\mathtt{A}\mathtt{x}\mathtt{x}\mathtt{B}\mathtt{x}$, or perhaps we should say the sum $\mathtt{x}\mathtt{x}\mathtt{a}\mathtt{x}\mathtt{x}\mathtt{b}\mathtt{x}+\mathtt{x}\mathtt{x}\mathtt{A}\mathtt{x}\mathtt{x}\mathtt{B}\mathtt{x}$, to the sums $\mathtt{x}\mathtt{x}\mathtt{a}\mathtt{x}+\mathtt{x}\mathtt{x}\mathtt{A}\mathtt{x}$ and $\mathtt{x}\mathtt{b}\mathtt{x}+\mathtt{x}\mathtt{B}\mathtt{x}$ would be encoded as a functor from $\mathbf{Seg}(\Omega)$ to a category of semimodules.
\vspace{1pt}

2) Semimodules over semi-rings are very natural structures in which it is possible to talk about all sorts of interactions without contradicting the noticeable behavior of real processes. Specifically, the reason for which we will use semi-rings instead of rings is that quotients of rings are too destructive. While in semimodules over rings, an equation of the form $x+y = x$ implies the identity $y=0$, in semimodules over semi-rings, an equality $x+y = y$ will inform us that the information contained in $x$ is already contained in $y$, which is closer to the type of things one may want to say about biological systems.

3) Finally, the Boolean semi-ring with two elements, call it $B_2$, has the advantage to be suitable for statistical analysis (see section \ref{Genetic_linkage_and_mapping_functions}). For instance,  it is interesting to note that the free $B_2$-semimodule that is generated over a finite set $S$ is isomorphic to the power set of $S$ equipped with its union operation (see Example \ref{exa:power_set_of_finite_sets} and section \ref{ssec:Reminder_on_B_2_semimodules}). This power set, call it $\mathcal{P}(S)$, can then be seen as the set of probable events for a given universe $S$ of outcomes. As a result, the elements of a $B_2$-semimodules will sometimes be seen as statistical events in the sense of a probability space \cite{Loeve}.

\subsection{Semi-rings}
We will call a \emph{semi-ring} \cite[page 1]{Golan} any set $R$ equipped with a commutative monoid structure $(R,+,0)$ and monoid structure $(R,\cdot,1)$ such that $+$ is distributive over $\cdot$ and the operation $\cdot$ annihilates the neutral element 0, that is to say that, for every triple $ a,b,c \in R$, the following identities hold.
\[
(a+b)\cdot c = a \cdot c + b \cdot c\quad\quad\quad   a \cdot (b+c) = a \cdot b + a \cdot c\quad\quad\quad  a \cdot 0 = 0 = 0 \cdot a
\]
We shall denote by $B_2$ the semi-ring consisting of two elements $\{0,1\}$ whose monoid structures are defined according to the following addition and multiplication tables.
\[
\begin{array}{c|cc}
+&0&1\\
\hline
0 & 0 & 1\\
1 & 1 & 1
\end{array}
\quad\quad\quad\quad\quad\quad\quad\quad
\begin{array}{c|cc}
\cdot&0&1\\
\hline
0 & 0 & 0\\
1 & 0 & 1
\end{array}
\]

\subsection{Semimodules over semi-rings}
Let $(R,+,\cdot,0,1)$ be a semi-ring. A \emph{semimodule} over $R$ \cite[page 149]{Golan} is a commutative monoid $(M,\oplus,0_M)$ equipped with a function $\odot:R \times M \to M$  satisfying the following axioms:
\begin{itemize}
\item[1)] for every $r \in R$ and $x,y \in M$, the identity $r \odot (x\oplus y) = (r \odot x) \oplus (r \odot y)$ holds;
\item[2)] for every $r,s \in R$ and $x \in M$, the identity $(r+s)\odot x = (r \odot x) \oplus (s\odot x)$ holds;
\item[3)] for every $r,s \in R$ and $x \in M$, the identity $(r \cdot s) \odot x = r \odot (s \odot x)$ holds;
\item[4)] for every $r \in R$ and $x \in M$, the identities $1 \odot x = x$ and $0 \odot x = 0_M = r \odot 0_M$ hold;
\end{itemize}
In the sequel, the operations $\oplus$ and $\odot$ will be denoted by the same symbols as those used for the semi-ring structure, namely $(\_)+(\_)$ and $(\_)\cdot(\_)$.

\begin{example}[\texorpdfstring{$B_2$}{Lg} as a \texorpdfstring{$B_2$}{Lg}-semimodule]\label{exa:B_2_is_a_B_2_semimodule}
The set $B_2$ has an obvious semimodule structure over itself where its commutative monoid structure is given by the addition $+:B_2 \times B_2 \to B_2$ and the action is given by the multiplication $\cdot:B_2\times B_2 \to B_2$.
\end{example}

\begin{example}[Power sets]\label{exa:power_set_of_finite_sets}
Recall that the \emph{power set} of a set $S$, say $\mathcal{P}(S)$, is the set of all subsets of $S$. This set can also be viewed as the hom-set $\mathbf{Set}(S,\{0,1\})$, in which a function $S \to \{0,1\}$ can be seen as a subset of $S$ by specifying whether it contains an element or not by sending it to 1 or 0. We illustrate below the different ways that can be used to represent an element in the set $\mathcal{P}(S)$ when $S=\{\mathtt{A}\mathtt{G},\mathtt{C}\mathtt{G},\mathtt{C}\mathtt{A},\mathtt{T}\mathtt{A}\}$.
\[
\begin{array}{ccccc}
\cellcolor[gray]{0.8}\textrm{As a subset}&\cellcolor[gray]{0.8}& \cellcolor[gray]{0.8}\textrm{As a function}  &\cellcolor[gray]{0.8}& \cellcolor[gray]{0.8}\textrm{As a formal sum}\\
\hline
&&&&\vspace{-7pt}\\
\begin{array}{ccc}
f=\{\mathtt{A}\mathtt{G},\mathtt{C}\mathtt{G},\mathtt{T}\mathtt{A}\}
\end{array}
&\quad\quad
&
f:
\left[
\begin{array}{cccccc}
\mathtt{A}\mathtt{G}&\mapsto&1&\mathtt{C}\mathtt{A}&\mapsto&0\\
\mathtt{C}\mathtt{G}&\mapsto&1&\mathtt{T}\mathtt{A}&\mapsto&1\\
\end{array}
\right]
&
\quad\quad&
\begin{array}{ccc}
f=\mathtt{A}\mathtt{G}+\mathtt{C}\mathtt{G}+\mathtt{T}\mathtt{A}
\end{array}
\end{array}
\]
Let us now show that $\mathcal{P}(S)$ has a semimodule structure over $B_2$. As will be seen, each representation turn out to be useful depending on what one wants to express. 

First, the power set $\mathcal{P}(S)$ is a commutative monoid for the union of subsets of $S$. Note that this monoid operation can also be defined as the pointwise extension of the addition of $B_2$ on the images of the functions $S \to\{0,1\}$. However, in a calculation, it is better used in terms of formal sums, as shown below.
\[
(\mathtt{A}\mathtt{G}+\mathtt{C}\mathtt{G}) + (\mathtt{T}\mathtt{A}+\mathtt{C}\mathtt{G}) = \mathtt{A}\mathtt{G}+\underbrace{(1+1)}_{=1}\cdot \,\mathtt{C}\mathtt{G} + \mathtt{T}\mathtt{A} =  \mathtt{A}\mathtt{G}+\mathtt{C}\mathtt{G}+\mathtt{T}\mathtt{A}
\]
For its part, the action $B_2 \times \mathcal{P}(S) \to \mathcal{P}(S)$ is induced by the action of the multiplication of $B_2$ on the images of the functions $S \to \{0,1\}$. Interestingly, the resulting $B_2$-semimodule is equipped with an obvious function $S \to \mathcal{P}(S)$ (mapping $x \in S$ to the singleton $\{x\} \subseteq S$) that can be shown to possess a universal property when $S$ is a finite set (see Remark \ref{rem:description_free_b_2_semimodule}).
\end{example}

\subsection{Morphisms of semimodules}
Let $R$ be a semi-ring and $M$ and $N$ be two semimodules over $R$. A \emph{morphism of semimodules} from $M$ to $N$ is a function $f:M \to N$ such that for every pair $r,s \in R$ and pair $x,y \in M$, the following identity holds.
\begin{equation}\label{eq:morphisms_of_semimodules}
f(r \cdot x + s \cdot y ) = r \cdot f(x) + s \cdot f(y)
\end{equation}

\begin{remark}[Boolean coefficients]
In the case of a morphism of $B_2$-semimodules, it suffices to only verify equation (\ref{eq:morphisms_of_semimodules}) when $(r,s)=(1,1)$ and $(r,s)=(0,0)$, namely the equations $f(x+y) = f(x)+f(y)$ and $f(0_M) = 0_N$.
\end{remark}

\begin{example}[Generator]\label{exa:generator_B_2_morphism}
Any morphism $B_2 \to X$ of semimodules over $B_2$ is of the form $1 \mapsto x$ and $0 \mapsto 0$. This is equivalent to picking an element in the $B_2$-semimodule $X$.
\end{example}

\begin{example}\label{exa:power_set_morphism_of_modules_over_B_2}
In this example, we will use the notations of Example
\ref{exa:power_set_of_finite_sets}. Let $S$ and $T$ denote the sets $\{\mathtt{A}\mathtt{G},\mathtt{C}\mathtt{G},\mathtt{C}\mathtt{A},\mathtt{T}\mathtt{A}\}$ and $\{\mathtt{T}\mathtt{G},\mathtt{C}\mathtt{G},\mathtt{G}\mathtt{A},\mathtt{T}\mathtt{A}\}$, 
respectively. The function $f:S \to T$ given below, on the left, induces a morphism $f_*:\mathcal{P}(S) \to \mathcal{P}(T)$ of semimodules over $B_2$ that maps a subset of $S$ to its image via $f$ (see below, on the right).
\[
f:\left(
\begin{array}{ccc}
\mathtt{A}\mathtt{G}&\mapsto&\mathtt{T}\mathtt{G}\\
\mathtt{C}\mathtt{G}&\mapsto&\mathtt{C}\mathtt{G}\\
\mathtt{C}\mathtt{A}&\mapsto&\mathtt{G}\mathtt{A}\\
\mathtt{T}\mathtt{A}&\mapsto&\mathtt{T}\mathtt{A}
\end{array}
\right)
\quad\quad\leadsto\quad\quad
\begin{array}{lll}
f_*(\mathtt{A}\mathtt{G}) &=& \mathtt{T}\mathtt{G}\\
f_*(\mathtt{C}\mathtt{G}) &=& \mathtt{C}\mathtt{G}\\
f_*(\mathtt{A}\mathtt{G}+\mathtt{C}\mathtt{A}) &=& \mathtt{T}\mathtt{G}+\mathtt{G}\mathtt{A}\\
f_*(\mathtt{C}\mathtt{A}+\mathtt{C}\mathtt{G})& =& \mathtt{G}\mathtt{A}+\mathtt{C}\mathtt{G}\\
\end{array}
\]
\end{example}

\begin{convention}[Notation]
For every semi-ring $R$, the obvious category whose objects are semimodules over $R$ (also called \emph{$R$-semimodules}) and whose arrows are morphisms of $R$-semimodules will be denoted by $R\textrm{-}\mathbf{Mod}$.
\end{convention}

\begin{remark}[Power set as a functor]
One can check that the mapping rule $S \mapsto \mathcal{P}(S)$ extends to a functor $\mathcal{P}:\mathbf{Set} \to B_2\textrm{-}\mathbf{Mod}$ (see Examples \ref{exa:power_set_of_finite_sets} \& \ref{exa:power_set_morphism_of_modules_over_B_2}).
\end{remark}

\subsection{Reminder on \texorpdfstring{$B_2$}{Lg}-semimodules}\label{ssec:Reminder_on_B_2_semimodules}
A good reference on semirings and semimodules over semirings is \cite{Golan}. In this section, we recall some useful facts about $B_2$-semimodules. First of all, recall that $B_2\textrm{-}\mathbf{Mod}$ is a locally presentable category
and therefore admits all limits and colimits. It also follows from this fact that we have an adjunction of the form (\ref{adjunction_set_b_2_semimodule}), where the left adjoint $F$ maps a set $S$ to the free $B_2$-semimodules generated over $S$ (see Remark \ref{rem:description_free_b_2_semimodule}) while the right adjoint $U$ is the obvious forgetful functor.
\begin{equation}\label{adjunction_set_b_2_semimodule}
\xymatrix{
\mathbf{Set}\ar@<+1.2ex>[r]^-F\ar@{<-}@<-1.2ex>[r]_-U\ar@{}[r]|-{\bot}&B_2\textrm{-}\mathbf{Mod}
}
\end{equation}

\begin{remark}[Free semimodules]\label{rem:description_free_b_2_semimodule}
For every set $S$, the $B_2$-semimodule $F(S)$ is equivalently:
\begin{itemize}
\item[1)] the set of finite subsets of $S$;
\item[2)] the set of functions $f:S \to \{0,1\}$ for which the fiber $f^{-1}(1)$ is finite;
\item[3)] the set of formal sums of finite collections of distinct elements in $S$.
\end{itemize}
The unit of adjunction (\ref{adjunction_set_b_2_semimodule}) is given, in the first case, by the function $x \mapsto \{x\}$ and, in the second case, by the function $x \mapsto \delta_x$ where $\delta_x$ maps $x$ to $1$ and any other element to 0.
\end{remark}

\begin{convention}[Support]\label{conv:support}
For every set $S$ and element $x \in F(S)$, we will speak of the \emph{support} of $x$ to refer to the set representation of $x$ (see item 1 of Remark \ref{rem:description_free_b_2_semimodule}). This set will be denoted as $\mathsf{Supp}(x)$.
\end{convention}

\begin{example}[Support]
With respect to the notation of Example \ref{exa:power_set_of_finite_sets}, the support of the formal sum $f=\mathtt{A}\mathtt{G}+\mathtt{C}\mathtt{G}+\mathtt{T}\mathtt{A}$ is the set $\{\mathtt{A}\mathtt{G},\mathtt{C}\mathtt{G},\mathtt{T}\mathtt{A}\}$.
\end{example}

\begin{remark}[Support]
For every set $S$, if we see an element $x \in F(S)$ as a function $S \to \{0,1\}$, then the set $\mathsf{Supp}(x)$ is equal to the fiber $x^{-1}(1)$. 
\end{remark}

\begin{definition}[Sub-element]\label{def:sub-element_leq}
Let $S$ be a set. An element $y \in F(S)$ will be said to be a \emph{sub-element} of another element $x \in F(S)$ if the support of $y$ is included in the support of $x$. We will then write $y \leq x$.
\end{definition}

\begin{remark}[Reformulation]\label{rem:reformulation_leq_formal_sum_function_subsets}
From the point of view of the functions $S \to \{0,1\}$, the relation $\leq$ given in Definition \ref{def:sub-element_leq} can be seen as the pointwise pre-order induced by the Boolean pre-ordered set $\{0\leq 1\}$ on the following product set (as in Example \ref{exa:product_pre-order_set_0_1}).
\[
\mathbf{Set}(S,\{0,1\}) \cong \prod_{s \in S}\{0,1\}
\]
\end{remark}

\begin{example}
Let $S$ denote the set $\{\mathtt{A}\mathtt{G},\mathtt{C}\mathtt{G},\mathtt{C}\mathtt{A},\mathtt{T}\mathtt{A}\}$. If one wants to think of the relation $\leq$ on $F(S)$ in terms of formal sums, then one deals with inequalities as follows.
\[
\mathtt{A}\mathtt{G}+\mathtt{C}\mathtt{A} \leq \mathtt{A}\mathtt{G}+\mathtt{C}\mathtt{G}+\mathtt{C}\mathtt{A}
\]
\end{example}

\begin{proposition}\label{prop:sub-element_B_2_semimodules_partial_order}
The binary relation $\leq$ is a partial order.
\end{proposition}
\begin{proof}
Because the inclusion of sets is a partial order.
\end{proof}

\subsection{Reminder on monomorphisms and epimorphisms}
In this section, we recall a few facts on how to detect epimorphisms and monomorphisms. The content of this section is therefore about certain concepts of category theory (see \cite[def. 7.2.1.4]{SpivakBook} and/or below) and not biology. 

First, recall that a monomorphism in a category $\mathcal{C}$ is an arrow $m:A \to B$ such that for every pair of parallel arrows $f,g:X \rightrightarrows A$ for which the equation $m \circ f = m \circ g$ holds, the two arrows $f$ and $g$ must be equal.

\begin{proposition}\label{prop:characterization_mono}
If $\mathcal{C}$ has pullbacks, then an arrow $m:A \to B$ is a monomorphism in $\mathcal{C}$ if the pullback $p_1,p_2:P \rightrightarrows  A$ of two copies of $m$ is such that $p_1$ equals $p_2$.
\end{proposition}
\begin{proof}
For every pair of parallel arrows $f,g:X \rightrightarrows A$ for which the equation $m \circ f = m \circ g$ holds, the universality of the pullback $P$ gives an arrow $h:X \to P$ for which the identities $f = p_1 \circ h$ and $g = p_2 \circ h$ hold. Because $p_1 = p_2$, we have $f= g$ and the statement follows.
\end{proof}

Now, recall that an epimorphism in a category $\mathcal{C}$ is an arrow $e:A \to B$ such that for every pair of parallel arrows $f,g:B \rightrightarrows X$ for which the equation $f \circ e = g \circ e$ holds, the two arrows $f$ and $g$ must be equal.

\begin{definition}[Orthogonality]
A morphism $f:X \to Y$ will be said to be \emph{orthogonal} to an object $I$ if for every arrow $i:I \to Y$, there exists a dashed arrow (called the \emph{lift}) making the diagram given below commute.
\[
\xymatrix@R-5pt{
&X\ar[d]^{e}\\
I\ar[r]_{i}\ar@{-->}[ru]&Y
}
\]
\end{definition}

\begin{proposition}\label{prop:characterization_epi}
Every morphism in $B_2\textrm{-}\mathbf{Mod}$ that is orthogonal with respect to the $B_2$-semimodule $B_2$ (Example \ref{exa:B_2_is_a_B_2_semimodule}) is an epimorphism. 
\end{proposition}
\begin{proof}
We can use Example \ref{exa:generator_B_2_morphism} to see that if an arrow $e:X \to Y$ is orthogonal with respect to the $B_2$-semimodule $B_2$, then it is surjective: for every $y \in Y$, there exists $x \in X$ for which the identity $e(x) = y$ holds. We can directly check that surjective morphisms in $B_2\textrm{-}\mathbf{Mod}$ are epimorphisms in $B_2\textrm{-}\mathbf{Mod}$.
\end{proof}

\begin{example}[Coequalizers]\label{exa:characterization_epi_coequalizer_maps}
Let us show that coequalizer maps \cite[def. 3.3.3.1]{SpivakBook} in $B_2\textrm{-}\mathbf{Mod}$ are orthogonal to the $B_2$-semimodule $B_2$. Let $f,g:X \rightrightarrows Y$ be a pair of morphisms in $B_2\textrm{-}\mathbf{Mod}$. Since $B_2\textrm{-}\mathbf{Mod}$ is cocomplete, we can form its coequalizer $e:Y \to Q$. We can then show, from \cite[Example 15.1]{Golan}, that this map is a quotient map. To show this, define the equivalence relation $R$ containing the pairs of elements $m$ and $m'$ in $Y$ such that for every object $Z$ and morphism $h:Y \to Z$ in $B_2\textrm{-}\mathbf{Mod}$ for which the equation $h \circ f = h \circ g$ holds, the relation $h(m) = h(m')$ holds too. We can verify that $R$ defines a congruence \cite{Golan} in $B_2\textrm{-}\mathbf{Mod}$ (i.e. an equivalence relation living in $B_2\textrm{-}\mathbf{Mod}$). The form of the definition of $R$ is such that all the assumptions of \cite[Example 15.1]{Golan} can be easily verified. It then follows that the map $e:Y \to Q$ can be identified as a quotient map $e_R:M \to M/R$ (see \cite[Page 163]{Golan}) so that there is a bijection making the following diagram commute.
\[
\xymatrix@C-55pt@R-10pt{
&B_2\textrm{-}\mathbf{Mod}(B_2,Y)\ar[rd]^{e_R}\ar[ld]_e&\\
B_2\textrm{-}\mathbf{Mod}(B_2,Q) &\cong& B_2\textrm{-}\mathbf{Mod}(B_2,Y)/R
}
\] 
This bijection implies that for every arrow $i:B_2 \to Q$, there exists an arrow $h:B_2 \to Y$ for which the identity $e \circ h = i$ holds. In other words, the coequalizer map $e$ is orthogonal to the $B_2$-semimodule $B_2$.
\end{example}

\subsection{Wide spans}
Let $\mathcal{C}$ be a category. We will speak of a \emph{wide span} in $\mathcal{C}$ to refer to a pair $(k,\mathbf{S})$ where $k$ is a positive integer and $\mathbf{S}$ is a $[k]$-indexed collection of arrows in $\mathcal{C}$ whose domains are all equal (see below).
\[
\begin{array}{cccc}
\xymatrix@R-18pt@C-20pt{
&&S\ar[lld]\ar[ld]\ar[rd]\ar[rrd]&&\\
S_1&S_2&\dots&S_{k-1}&S_k
}
\end{array}
\]
Later on, a wide span $(k,\mathbf{S})$ will often be denoted as $\mathbf{S}$ only.

\begin{example}
See Example \ref{exa:Relative_definition_families} for an example of wide span related to pedigrads.
\end{example}

\begin{remark}[Wide spans are cones]
A wide span is a cone defined over a finite discrete small category whose objects are equipped with a total order.
\end{remark}

\begin{definition}[Cardinality]
For every wide span $\mathbf{S}=\{S \to S_i\}_{i \in [k]}$ in $\mathcal{C}$, the integer $k$ and ordered set $[k]$ will be denoted as $|\mathbf{S}|$ and $[\mathbf{S}]$, respectively. Both data will be referred to as the \emph{cardinality} of $\mathbf{S}$ (one being an integer, the other being a set).
\end{definition}

\begin{convention}[Product of a wide span]
Suppose $\mathcal{C}$ has products. For every wide span $\mathbf{S}=\{S \to S_i\}_{i \in [k]}$ in $\mathcal{C}$, we will denote by $\mathbf{S}^{\times}$ the product object $S_1 \times \dots\times S_n$.
\end{convention}

\begin{remark}[Canonical arrows]
Every wide span $\mathbf{S}=\{S \to S_i\}_{i \in [k]}$ in a category $\mathcal{C}$ that has products is equipped with an obvious arrow $S \to \mathbf{S}^{\times}$ (the limit adjoint of the underlying cone) and an obvious projection map $\mathbf{S}^{\times} \to S_i$ whose composition is equal to the arrow $S \to S_i$ for every $i \in [k]$.
\end{remark}

\begin{definition}[Cartesian]
A wide span $\mathbf{S}=\{S \to S_i\}_{i \in [k]}$ in a category $\mathcal{C}$ that has products $\mathcal{C}$ will be said to be \emph{Cartesian} if the canonical arrow $S \to \mathbf{S}^{\times}$ is an isomorphism.
\end{definition}

\subsection{Finite chromologies}\label{ssec:finite_chromologies}
We shall speak of a \emph{finite chromology} to refer to a chromology $(\Omega,D)$ whose sets $D[n]$ are finite sets of wide spans, for every non-negative integer $n$. 

\begin{example}
The cones given in Examples \ref{exa:Homologous_cones} \& \ref{exa:Quasi-homologous_cones} are suitable for defining a finite chromology. These cones are defined over a finite discrete small category consisting of three objects. An example of non-suitable cone is given below, where the small category on which it is defined is a cospan $A=\{\cdot \rightarrow \cdot \leftarrow \cdot\}$ and is therefore not discrete.
\[
\begin{array}{l}
\xymatrix@C-30pt@R-25pt{
&&&&&&&&(\bullet&\bullet&\bullet&\bullet)&(\circ&\circ&\circ)\ar[rrd]&&&&&&&&\\
(\bullet&\bullet&\bullet&\bullet)&(\bullet&\bullet&\bullet)\ar[rru]\ar[rrd]&\quad\quad\quad\quad\quad\quad&&&&&&&&\quad\quad&(\circ&\circ&\circ&\circ)&(\circ&\circ&\circ)\\
&&&&&&&&(\circ&\circ&\circ&\circ)&(\bullet&\bullet&\bullet)\ar[rru]&&&&&&&&
}
\end{array}
\quad\quad(\textrm{non-suitable})
\]
\end{example}

\subsection{Recombination cones}\label{ssec:Relative_topology_families}
In this section, we shall let $(\Omega,D)$ denote a finite chromology. The goal of this section is to show how the cones of $D$ can be used to associate all segments over $\Omega$ with a span of sets that describes the range of recombination operations related to this object from the point of view of a given pedigrad in $\mathbf{Set}$ (see section \ref{ssec:Recomb_congruences}).

\begin{definition}[Recombination cones]\label{def:Relative_definition_families}
Let $X$ be a functor $\mathbf{Seg}(\Omega) \to \mathbf{Set}$. For every wide span $\rho:\Delta_A(\tau) \Rightarrow \theta$ in $D$, the wide span of sets that is the image of $\rho$ via $X$
will be called the \emph{$\rho$-recombination cone} of $X$ and denoted by $X(\rho)$.
\end{definition}

\begin{example}[Recombination cones]\label{exa:Relative_definition_families}
Let $\Omega$ be the Boolean pre-ordered set $\{0 \leq 1\}$ and let $(E,\varepsilon)$ be our usual pointed set $\{\mathtt{A},\mathtt{C},\mathtt{G},\mathtt{T},\varepsilon\}$. We will take $b$ to be equal to 1 and $\rho$ to be the exactly 1-distributive cone given in Example \ref{exa:Homologous_cones} (see below); the small discrete category $A$ on which $\rho$ is defined will be taken to be equal to the finite set $\{a_1,a_2,a_3\}$.
\[
\rho:
\begin{array}{l}
\xymatrix@C-30pt@R-20pt{
&&&&&&&&&&&&&&&&&
&&
(\circ&\circ&\circ)&(\circ&\circ)&(\bullet&\bullet&\bullet&\bullet)&(\circ&\circ&\circ&\circ&\circ)&(\circ&\circ&\circ)&(\circ)&\quad&\theta(a_1)
\\
(\bullet&\bullet&\bullet)&(\circ&\circ)&(\bullet\ar@<-3ex>@{}[rrrrrrr]|{\underbrace{\quad\quad\quad\quad\quad\quad\quad\quad\quad\quad\quad\quad\quad}_{}}_{\tau}&\bullet&\bullet&\bullet)&(\bullet&\bullet&\bullet&\bullet&\bullet)&(\circ&\circ&\circ)&(\circ)\ar[rru]\ar[rr]\ar[rrd]
&\quad\quad\quad&
(\circ&\circ&\circ)&(\circ&\circ)&(\circ&\circ&\circ&\circ)&(\bullet&\bullet&\bullet&\bullet&\bullet)&(\circ&\circ&\circ)&(\circ)&\quad&\theta(a_2)
\\
&&&&&&&&&&&&&&&&&
&&
(\bullet&\bullet&\bullet)&(\circ&\circ)&(\circ&\circ&\circ&\circ)&(\circ&\circ&\circ&\circ&\circ)&(\circ&\circ&\circ)&(\circ)&\quad&\theta(a_3)
}
\end{array}
\] 
The $\rho$-recombination cone of $E_b^{\varepsilon}$ is a wide span consisting of three arrows
\[
E_b^{\varepsilon}(\tau) \to E_b^{\varepsilon}(\theta(a_1)),\quad E_b^{\varepsilon}(\tau) \to E_b^{\varepsilon}(\theta(a_2)),\quad E_b^{\varepsilon}(\tau) \to E_b^{\varepsilon}(\theta(a_3))
\]
whose respective codomains are, according to Remark \ref{rem:E_b_varepsilon_as words_functions}, of the form given below, in the rightmost column of the displayed table (the left columns are only given for the sake of exposition) and whose domains $E_b^{\varepsilon}(\tau)$ are isomorphic to $E^{\times 12}$.
\[
\begin{array}{l|l|l|ll}
\cellcolor[gray]{0.8}i & \multicolumn{1}{c|}{\cellcolor[gray]{0.8}\mathsf{Tr}_1(\theta(a_i))} &\cellcolor[gray]{0.8} \leadsto & \multicolumn{2}{c}{\cellcolor[gray]{0.8}E_b^{\varepsilon}(\theta(a_i))}\\
\hline
1 & \{6,7,8,9\} &\leadsto & \mathbf{Set}(\{6,7,8,9\},E)&\cong E^{\times 4}\\
2 & \{10,11,12,13,14\}&\leadsto & \mathbf{Set}(\{10,11,12,13,14\},E) &\cong E^{\times 5}\\
3 & \{1,2,3\}&\leadsto &\mathbf{Set}(\{1,2,3\},E) &\cong E^{\times 3}\\
\end{array}
\]
By Proposition \ref{prop:E_b_varepsilon_W_iso_pedigrad_exactly_distributive}, 
the canonical arrow $E_b^{\varepsilon}(\tau) \to E_b^{\varepsilon}(\rho)^{\times}$ must be a bijection, which implies that the wide span $E_b^{\varepsilon}(\rho)$ is Cartesian. From the point of view of the isomorphisms given in the previous table, the bijection $E_b^{\varepsilon}(\tau) \to E_b^{\varepsilon}(\rho)^{\times}$ is of the form 
\[
E^{\times 12} \to E^{\times 4} \times E^{\times 5} \times E^{\times 3}
\]
and sends the words given below, on the left, (to which parenthesis have been added to facilitate the recognition of the construction) to the corresponding tuples given on the right.
\[
\begin{array}{lll}
(\mathtt{T}\mathtt{A}\mathtt{G})(\mathtt{A}\mathtt{C}\mathtt{G}\mathtt{A})(\mathtt{C}\mathtt{G}\varepsilon\mathtt{T}\mathtt{T})& \mapsto & (\mathtt{A}\mathtt{C}\mathtt{G}\mathtt{A},\mathtt{C}\mathtt{G}\varepsilon\mathtt{T}\mathtt{T},\mathtt{T}\mathtt{A}\mathtt{G})\\
(\varepsilon\mathtt{C}\mathtt{A})(\mathtt{G}\mathtt{G}\mathtt{T}\mathtt{A})(\mathtt{C}\mathtt{C}\mathtt{T}\mathtt{A}\mathtt{T}) & \mapsto & (\mathtt{G}\mathtt{G}\mathtt{T}\mathtt{A},\mathtt{C}\mathtt{C}\mathtt{T}\mathtt{A}\mathtt{T},\varepsilon\mathtt{C}\mathtt{A})\\
(\mathtt{C}\varepsilon\varepsilon)(\mathtt{G}\mathtt{G}\mathtt{C}\mathtt{C})(\mathtt{T}\mathtt{A}\mathtt{G}\mathtt{T}\mathtt{T}) & \mapsto & (\mathtt{G}\mathtt{G}\mathtt{C}\mathtt{C},\mathtt{T}\mathtt{A}\mathtt{G}\mathtt{T}\mathtt{T},\mathtt{C}\varepsilon\varepsilon)\\
&\textrm{etc.}&
\end{array}
\]
\end{example}

\subsection{Recombination congruences}\label{ssec:Recomb_congruences}
The goal of this section is to define congruences that will allow us to identify two populations sharing the same haplotype. These congruences will be defined with respect to wide spans as given in section \ref{ssec:Relative_topology_families}. In this section, the functor $F:\mathbf{Set} \to B_2\textrm{-}\mathbf{Mod}$ refers the left adjoint of (\ref{adjunction_set_b_2_semimodule}).

\begin{convention}[Wide spans of semimodules]
For every wide span $\mathbf{S}=\{S \to S_k\}_{k \in [\mathbf{S}]}$ in $\mathbf{Set}$, we will denote by $F(\mathbf{S})$ the corresponding wide span $\{F(S) \to F(S_k)\}_{k \in [\mathbf{S}]}$ in $B_2\textrm{-}\mathbf{Mod}$.
\end{convention}

\begin{convention}\label{conv:definition_pi_S}
For every wide span $\mathbf{S}=\{S \to S_k\}_{k \in [\mathbf{S}]}$ in $\mathbf{Set}$, we will denote by $\pi_{\mathbf{S}}$ the canonical morphism $F(S) \to F(\mathbf{S})^{\times}$ in $B_2\textrm{-}\mathbf{Mod}$ associated with the wide span $F(\mathbf{S})$ in $B_2\textrm{-}\mathbf{Mod}$.
\end{convention}

\begin{definition}[Recombination congruences]\label{def:Recomb_congruences}
For every wide span $\mathbf{S}=\{S \to S_k\}_{k \in [\mathbf{S}]}$ in $\mathbf{Set}$, the pullback $G(\mathbf{S})\rightrightarrows F(S)$ defined below will be called the \emph{recombination congruence} of $\mathbf{S}$.
\[
\xymatrix{
G(\mathbf{S})\ar@{}[rd]|<<<{\rotatebox[origin=c]{90}{\huge{\textrm{$\llcorner$}}}}\ar[r]^{\mathsf{prj}_1}\ar[d]_{\mathsf{prj}_2}&F(S)\ar[d]^{\pi_{\mathbf{S}}}\\
F(S)\ar[r]_{\pi_{\mathbf{S}}}&F(\mathbf{S})^{\times}
}
\]
\end{definition}

\begin{remark}\label{rem:Recombination_congruence_cartesian_wide_span}
If we let $\mathbf{S}$ be a Cartesian wide span in $\mathbf{Set}$, then the pullback $G(\mathbf{S})\rightrightarrows F(S)$ can also be seen as a pullback $G(\mathbf{S})\rightrightarrows F(\mathbf{S}^{\times})$ by post-composition with the isomorphism $F(S) \to F(\mathbf{S}^{\times})$ since the following diagram commutes by universality of the product $F(\mathbf{S})^{\times}$.
\[
\xymatrix{
F(S)\ar@/^20pt/[rr]^{\pi_{\mathbf{S}}}\ar[r]^{\cong}&F(\mathbf{S}^{\times})\ar[r]&F(\mathbf{S})^{\times}
}
\]
\end{remark}

\begin{example}[Homologous recombination]\label{exa:Recombination_congruences_pi_family}
For convenience, we shall let $\mathbf{S}$ denote the Cartesian wide span $E_b^{\varepsilon}(\rho)$ given in Example \ref{exa:Relative_definition_families}. The associated recombination congruence is of the form $G(\mathbf{S})\rightrightarrows FE_b^{\varepsilon}(\tau)$. As explained in Remark \ref{rem:Recombination_congruence_cartesian_wide_span}, this recombination congruence can also be viewed as a congruence of the form $G(\mathbf{S})\rightrightarrows F(\mathbf{S}^{\times})$, which will be much easier for the eye (see below). To see what the elements of this congruence look like, take $x$ and $y$ to be the sums given below in the top and bottom rows, respectively.
\[
\begin{array}{c|c|c|c}
\cellcolor[gray]{0.8}&\cellcolor[gray]{0.8}\textrm{seen in }F(\mathbf{S}^{\times})&\cellcolor[gray]{0.8}\cong&\cellcolor[gray]{0.8}\textrm{seen in }FE_b^{\varepsilon}(\tau)\\
\hline
&&&\vspace{-5pt}\\
x&(\mathtt{A}\mathtt{C}\mathtt{G}\mathtt{A},\mathtt{C}\mathtt{G}\varepsilon\mathtt{T}\mathtt{T},\mathtt{T}\mathtt{A}\mathtt{G}) + (\mathtt{G}\mathtt{G}\mathtt{T}\mathtt{A},\mathtt{C}\mathtt{C}\mathtt{T}\mathtt{A}\mathtt{T},\varepsilon\mathtt{C}\mathtt{A})
&&
\mathtt{T}\mathtt{A}\mathtt{G}\mathtt{A}\mathtt{C}\mathtt{G}\mathtt{A}\mathtt{C}\mathtt{G}\varepsilon\mathtt{T}\mathtt{T} + \varepsilon\mathtt{C}\mathtt{A}\mathtt{G}\mathtt{G}\mathtt{T}\mathtt{A}\mathtt{C}\mathtt{C}\mathtt{T}\mathtt{A}\mathtt{T}\\
y&(\mathtt{G}\mathtt{G}\mathtt{T}\mathtt{A},\mathtt{C}\mathtt{G}\varepsilon\mathtt{T}\mathtt{T},\varepsilon\mathtt{C}\mathtt{A}) + (\mathtt{A}\mathtt{C}\mathtt{G}\mathtt{A},\mathtt{C}\mathtt{C}\mathtt{T}\mathtt{A}\mathtt{T},\mathtt{T}\mathtt{A}\mathtt{G})
&&
\varepsilon\mathtt{C}\mathtt{A}\mathtt{A}\mathtt{C}\mathtt{G}\mathtt{A}\mathtt{C}\mathtt{C}\mathtt{T}\mathtt{A}\mathtt{T}+\mathtt{T}\mathtt{A}\mathtt{G}\mathtt{G}\mathtt{G}\mathtt{T}\mathtt{A}\mathtt{C}\mathtt{G}\varepsilon\mathtt{T}\mathtt{T}\vspace{5pt}
\end{array}
\]
We can check that the images $\pi_{\mathbf{S}}(x)$ and $\pi_{\mathbf{S}}(y)$ are equal to the following tuple in $F(\mathbf{S})^{\times}$.
\[
\Big(\,\mathtt{A}\mathtt{C}\mathtt{G}\mathtt{A}+\mathtt{G}\mathtt{G}\mathtt{T}\mathtt{A}\,,\,\mathtt{C}\mathtt{G}\varepsilon\mathtt{T}\mathtt{T}+\mathtt{C}\mathtt{C}\mathtt{T}\mathtt{A}\mathtt{T}\,,\,\mathtt{T}\mathtt{A}\mathtt{G}+\varepsilon\mathtt{C}\mathtt{A}\,\Big)
\]
In other words, the pair $(x,y)$ belongs to the set $G(\mathbf{S})$. Here, the intuition is that the set $G(\mathbf{S})$ contains pairs of elements that are the same up to homologous recombination with respect to the topology specified by the cone $\rho$.
\end{example}

\begin{example}[Haplogroups and haplotypes]\label{exa:Haplogroups_biology}
In genetics, a \emph{haplotype} is a given set of genes or DNA strands, say $\mathtt{ACGA}$ and $\mathtt{TAG}$, for particular loci on a chromosome while a \emph{haplogroup} for this haplotype can be viewed as a group of DNA segments sharing these strands at the specified locations whereas the other locations may contain single-nucleotide  polymorphisms (SNP), which are one-nucleobase long mutations that are noticeable within a non-negligible percentage of the population (see \cite{Kivisild,Chen}).
\[
\xymatrix@R-10pt@C-45pt{
\fbox{$\begin{array}{c}\textrm{haplotype }\\\mathtt{ACGA}:\mathtt{TAG}\end{array}$}\,~\,~\,~\,~\,~\,~\,~\,~&\fbox{$\begin{array}{c}\textrm{haplogroup X}\\\dots\mathtt{ACGA}(\mathtt{A}\,\textrm{or}\,\mathtt{C})\mathtt{TAG}\dots\end{array}$}\ar[ld]\ar[rd]&\\
\fbox{$\begin{array}{c}\textrm{haplogroup  XA}\\\dots\mathtt{ACGA}\underline{\mathtt{A}}\mathtt{TAG}\dots\end{array}$}&&\fbox{$\begin{array}{c}\textrm{haplogroup XC}\\\dots\mathtt{ACGA}\underline{\mathtt{C}}\mathtt{TAG}\dots\end{array}$}
}
\]
From the point of view of Example \ref{exa:Recombination_congruences_pi_family}, the haplogroup X that is given in the previous picture can be viewed as an element in the fiber of the map $\pi_{\mathbf{S}}:F(S) \to F(\mathbf{S})^{\times}$ above the following tuple in $F(\mathbf{S})^{\times}$.
\[
(\mathtt{ACGA},\mathtt{A}+\mathtt{C},\mathtt{TAG})
\]
In this paper, we shall therefore speak of a `haplotype' to refer to an element in the $B_2$-module $F(\mathbf{S})^{\times}$ while we shall speak of a `haplogroup' for this haplotype to refer to an element in the fiber of the map $\pi_{\mathbf{S}}:F(S) \to F(\mathbf{S})^{\times}$ above this element. In Example \ref{exa:Recombination_congruences_pi_family}, the two elements $x$ and $y$ could be viewed as two haplogroups of the same haplotype $(\mathtt{A}\mathtt{C}\mathtt{G}\mathtt{A}+\mathtt{G}\mathtt{G}\mathtt{T}\mathtt{A}\,,\,\mathtt{C}\mathtt{G}\varepsilon\mathtt{T}\mathtt{T}+\mathtt{C}\mathtt{C}\mathtt{T}\mathtt{A}\mathtt{T}\,,\,\mathtt{T}\mathtt{A}\mathtt{G}+\varepsilon\mathtt{C}\mathtt{A})$. Here, genetic polymorphism is no longer unary.
\end{example}

\begin{remark}[Haplotypes \emph{versus} recombination]
While the concept of haplotype is possible because of the absence of recombination in the transfer of the mtDNA \cite{Kivisild}, it makes perfect sense to see it arise from an object that is meant to model recombination. As illustrated in Example \ref{exa:Haplogroups_biology}, this absence of recombination is modelled by singleton elements, which characterize the haplotype, whereas the non-singleton elements (i.e. the sums) encode the genetic polymorphism associated with a population. In general, the classification of haplogroups relies on single-nucleotide polymorphisms (SNP) while the polymorphism  pertaining to rather-long DNA intervals is usually the type of variations that is studied from the point of view of homologous recombination \cite{Haldane}.
\end{remark}

\begin{remark}[Congruence]
For every wide span $\mathbf{S}$ in $\mathbf{Set}$, the recombination congruence $G(\mathbf{S})\rightrightarrows F(S)$ defines an actual congruence (i.e. an equivalence relation) in $B_2\textrm{-}\mathbf{Mod}$. Checking that the pullback $G(\mathbf{S})$ is a sub-object of the product $F(S)\times F(S)$ is straightforward. The reflexivity, symmetry and transitivity axioms follow after noticing that the pullback is defined over a cospan made of two copies of the same arrow (see Definition \ref{def:Recomb_congruences}).
\end{remark}

\begin{convention}[Notation]
Let $\mathbf{S}$ be a wide span in $\mathbf{Set}$. For every pair of elements $x,y \in F(S)$, we shall write $x \sim_{\mathbf{S}} y$ to mean that the pair $(x,y)$ belongs to $G(\mathbf{S})$.
\end{convention}

\begin{example}[Haplotypes]\label{exa:Haplogroups_x_x_x+y}
By definition of $G(\mathbf{S})$, the two elements $x$ and $y$ given in Example \ref{exa:Recombination_congruences_pi_family} are also equivalent to the sum $x+y$ displayed below.
\[
(\mathtt{A}\mathtt{C}\mathtt{G}\mathtt{A},\mathtt{C}\mathtt{G}\varepsilon\mathtt{T}\mathtt{T},\mathtt{T}\mathtt{A}\mathtt{G}) + (\mathtt{G}\mathtt{G}\mathtt{T}\mathtt{A},\mathtt{C}\mathtt{C}\mathtt{T}\mathtt{A}\mathtt{T},\varepsilon\mathtt{C}\mathtt{A})+(\mathtt{G}\mathtt{G}\mathtt{T}\mathtt{A},\mathtt{C}\mathtt{G}\varepsilon\mathtt{T}\mathtt{T},\varepsilon\mathtt{C}\mathtt{A}) + (\mathtt{A}\mathtt{C}\mathtt{G}\mathtt{A},\mathtt{C}\mathtt{C}\mathtt{T}\mathtt{A}\mathtt{T},\mathtt{T}\mathtt{A}\mathtt{G})
\]
To see this, observe that we have the relations $x \sim_{\mathbf{S}} y$ and $x \sim_{\mathbf{S}} x$. The fact that $G(\mathbf{S})$ is a $B_2$-semimodule then gives us the relation $x+x \sim_{\mathbf{S}} x + y$. Since the equation $x+x = x$ holds, we obtain the relation $x \sim_{\mathbf{S}} x+y$. Similarly, we can show that the relation $y \sim_{\mathbf{S}} x+y$ holds too.
Intuitively, this means that if $x$ and $y$ are two populations of the same haplotype (i.e. $x \sim_{\mathbf{S}} y$), then the union of these is of the same haplotype.
\end{example}

The argument of Example \ref{exa:Haplogroups_x_x_x+y} could be formalized in terms of a partial order on the so-called haplogroups so that the sum of all the elements contained in the equivalence class of a recombination congruence is the maximum haplogroup, which could be viewed as the haplotype itself. The rest of this section shows that this representative is obtained from a product operation (see Convention \ref{conv:recombination_congruence}) on the local patches of every haplogroup contained in the equivalence class.

\begin{convention}[Notation]\label{conv:recombination_congruence}
For every wide span $\mathbf{S}$, we will denote by $\beta_{\mathbf{S}}$ the function $U(F(\mathbf{S})^{\times}) \to UF(\mathbf{S}^{\times})$ that maps a tuple $(x_1,x_2,\dots,x_{|\mathbf{S}|})$ in $F(\mathbf{S})^{\times}$ to the element of $F(\mathbf{S}^{\times})$ represented by the finite subset $\mathsf{Supp}(x_1)\times \mathsf{Supp}(x_2) \times \dots \times \mathsf{Supp}(x_{|\mathbf{S}|}) \subseteq \mathbf{S}^{\times}$. %(see Example \ref{exa:Recombination_congruences_pi_family_cross}).
\end{convention}

\begin{example}[Sequel of Example \ref{exa:Recombination_congruences_pi_family}]\label{exa:Recombination_congruences_pi_family_cross}
Let $\mathbf{S}$ be the Cartesian wide span $E_b^{\varepsilon}(\rho)$ given in Example \ref{exa:Relative_definition_families}. If we take $x \in F(\mathbf{S}^{\times})$ to be the element considered in Example \ref{exa:Recombination_congruences_pi_family}, then the element
\[
\beta_{\mathbf{S}}(\pi_{\mathbf{S}}(x)) = \{\mathtt{A}\mathtt{C}\mathtt{G}\mathtt{A},\mathtt{G}\mathtt{G}\mathtt{T}\mathtt{A}\} \times \{\mathtt{C}\mathtt{G}\varepsilon\mathtt{T}\mathtt{T},\mathtt{C}\mathtt{C}\mathtt{T}\mathtt{A}\mathtt{T}\}\times\{\mathtt{T}\mathtt{A}\mathtt{G},\varepsilon\mathtt{C}\mathtt{A}\}
\]
is represented in $F(\mathbf{S}^{\times})$ by the following subset of $\mathbf{S}^{\times}$.
\[
\left\{
\begin{array}{llll}
(\mathtt{A}\mathtt{C}\mathtt{G}\mathtt{A},\mathtt{C}\mathtt{G}\varepsilon\mathtt{T}\mathtt{T},\mathtt{T}\mathtt{A}\mathtt{G})&
(\mathtt{A}\mathtt{C}\mathtt{G}\mathtt{A},\mathtt{C}\mathtt{G}\varepsilon\mathtt{T}\mathtt{T},\varepsilon\mathtt{C}\mathtt{A})&
(\mathtt{A}\mathtt{C}\mathtt{G}\mathtt{A},\mathtt{C}\mathtt{C}\mathtt{T}\mathtt{A}\mathtt{T},\mathtt{T}\mathtt{A}\mathtt{G})&
(\mathtt{A}\mathtt{C}\mathtt{G}\mathtt{A},\mathtt{C}\mathtt{C}\mathtt{T}\mathtt{A}\mathtt{T},\varepsilon\mathtt{C}\mathtt{A})\\
(\mathtt{G}\mathtt{G}\mathtt{T}\mathtt{A},\mathtt{C}\mathtt{G}\varepsilon\mathtt{T}\mathtt{T},\mathtt{T}\mathtt{A}\mathtt{G})&
(\mathtt{G}\mathtt{G}\mathtt{T}\mathtt{A},\mathtt{C}\mathtt{G}\varepsilon\mathtt{T}\mathtt{T},\varepsilon\mathtt{C}\mathtt{A})&
(\mathtt{G}\mathtt{G}\mathtt{T}\mathtt{A},\mathtt{C}\mathtt{C}\mathtt{T}\mathtt{A}\mathtt{T},\mathtt{T}\mathtt{A}\mathtt{G})&
(\mathtt{G}\mathtt{G}\mathtt{T}\mathtt{A},\mathtt{C}\mathtt{C}\mathtt{T}\mathtt{A}\mathtt{T},\varepsilon\mathtt{C}\mathtt{A})\\
\end{array}
\right\}
\]
Interestingly, the image of this element via the canonical arrow $F(\mathbf{S}^{\times}) \to F(\mathbf{S})^{\times}$, which can here be viewed as the map $\pi_{\mathbf{S}}:FE_b^{\varepsilon}(\tau) \to F(\mathbf{S})^{\times}$ by Remark \ref{rem:Recombination_congruence_cartesian_wide_span}, is equal to the following tuple.
\[
\Big(\,\mathtt{A}\mathtt{C}\mathtt{G}\mathtt{A}+\mathtt{G}\mathtt{G}\mathtt{T}\mathtt{A}\,,\,\mathtt{C}\mathtt{G}\varepsilon\mathtt{T}\mathtt{T}+\mathtt{C}\mathtt{C}\mathtt{T}\mathtt{A}\mathtt{T}\,,\,\mathtt{T}\mathtt{A}\mathtt{G}+\varepsilon\mathtt{C}\mathtt{A}\,\Big)
\]
By Example \ref{exa:Recombination_congruences_pi_family}, this tuple is equal to the element $\pi_{\mathbf{S}}(x)$, which implies that the identity $\pi_{\mathbf{S}}(x) = \pi_{\mathbf{S}}\beta_{\mathbf{S}}(\pi_{\mathbf{S}}(x))$ holds. In terms of haplogroups, this means that the two populations represented by $x$ and $\beta_{\mathbf{S}}(\pi_{\mathbf{S}}(x))$ have the same haplotype (i.e. $x \sim_{\mathbf{S}} \beta_{\mathbf{S}}(\pi_{\mathbf{S}}(x))$). Proposition \ref{prop:maximum _in_the_class}, given below, explains why such a relation holds in a greater generality.
\end{example}

\begin{convention}[Notation]\label{conv:definition_bot_function}
For every wide span $\mathbf{S}$ in $\mathbf{Set}$, we will denote by $\bot$ the function $U(F(\mathbf{S})^{\times}) \to U(F(\mathbf{S})^{\times})$ that maps a tuple $(x_1,x_2,\dots,x_{|\mathbf{S}|}) \in F(\mathbf{S})^{\times}$ to 0 if the tuple contains a zero component and that maps the tuple to itself otherwise.
\end{convention}

\begin{remark}\label{rem:bottom_and_pi_S}
It follows from Conventions \ref{conv:definition_bot_function} \& \ref{conv:definition_pi_S} that the following diagram commutes.
\[
\xymatrix{
UF(S) \ar[r]^-{U\pi_{\mathbf{S}}}\ar@/^20pt/[rr]^-{U\pi_{\mathbf{S}}} & U(F(\mathbf{S})^{\times}) \ar[r]^-{\bot} & U(F(\mathbf{S})^{\times})
}
\]
This comes from the fact that the only element of $F(S)$ whose image, in $F(\mathbf{S})^{\times}$, contains a zero component is the zero element of $F(S)$.
\end{remark}

\begin{proposition}\label{prop:maximum _in_the_class}
Let $\mathbf{S}$ be a wide span of sets. The following diagram commutes in $\mathbf{Set}$, where the rightmost vertical arrow is the image of the obvious arrow $F(\mathbf{S}^{\times}) \to F(\mathbf{S})^{\times}$ via the functor $U$.
\[
\xymatrix@-10pt{
U(F(\mathbf{S})^{\times}) \ar[rd]_{\bot}\ar[r]^-{\beta_{\mathbf{S}}}&  UF(\mathbf{S}^{\times})\ar[d]\\
&U(F(\mathbf{S})^{\times})
}
\]
\end{proposition}
\begin{proof}
Suppose that $\mathbf{S}$ is of the form $\{S \to S_i\}_{i \in [\mathbf{S}]}$ 
and take $\tilde{x}=(x_1,x_2,\dots,x_{|\mathbf{S}|})$ to be a tuple of $F(\mathbf{S})^{\times}$. 
If the components of $\tilde{x}$ are all non-zero elements, the arrow $F(\mathbf{S}^{\times}) \to F(S_i)$ sends the element $\beta_{\mathbf{S}}(\tilde{x})$ represented by the subset $\mathsf{Supp}(x_1)\times \mathsf{Supp}(x_2) \times \dots \times \mathsf{Supp}(x_{|\mathbf{S}|}) \subseteq \mathbf{S}^{\times}$ to the element represented by the subset $\mathsf{Supp}(x_i) \subseteq S_i$. Because this element can be identified as $x_i$ itself, this means that the arrow $F(\mathbf{S}^{\times}) \to F(\mathbf{S})^{\times}$ sends $\beta_{\mathbf{S}}(\tilde{x})$ to the tuple $\tilde{x}$. 
If there exists $k \in [\mathbf{S}]$ such that $x_k$ is zero, then the definition of a Cartesian product on an empty set implies that $\beta_{\mathbf{S}}(\tilde{x})$ is zero, so that its image via the arrow $F(\mathbf{S}^{\times}) \to F(S_i)$ is also zero. In this case, the arrow $F(\mathbf{S}^{\times}) \to F(\mathbf{S})^{\times}$ sends $\beta_{\mathbf{S}}(\tilde{x})$ to zero.
\end{proof}

The reader can easily verify the following statement in the case of Examples \ref{exa:Recombination_congruences_pi_family} \& \ref{exa:Recombination_congruences_pi_family_cross}.

\begin{proposition}\label{prop:Link_supports_pi_times}
For every wide span $\mathbf{S}$ of sets and element $x \in F(\mathbf{S}^{\times})$, the inclusion $\mathsf{Supp}(x) \subseteq \beta_{\mathbf{S}}(\pi_{\mathbf{S}}(x))$ holds.
\end{proposition}
\begin{proof}
Suppose that $\mathbf{S}$ is of the form $\{S \to S_i\}_{i \in [k]}$. For any element $s \in \mathbf{S}^{\times}$, we will use the notation $s(j)$ to denote the image of $s$ via the $j$-th projection $\mathbf{S}^{\times} \to S_j$. Without loss of generality, every element $x \in F(\mathbf{S}^{\times})$ can be supposed to be of the form $\sum_{k=1}^n s_k$ where $s_k \in \mathbf{S}^{\times}$. Now, observe that, for every $k \in \{1,\dots,n\}$, the tuple $(s_k(1),\dots,s_k(|\mathbf{S}|))$, which can be viewed as the element $s_k \in \mathbf{S}^{\times}$ itself, must belong to the set 
\[
\beta_{\mathbf{S}}(\pi_{\mathbf{S}}(x))=\{s_k(1)~|~k \in [n]\}\times \{s_k(2)~|~k \in [n]\} \times \dots \times \{s_k(|\mathbf{S}|)~|~k \in [n]\},
\]
so that the inclusion $\mathsf{Supp}(x) \subseteq \beta_{\mathbf{S}}(\pi_{\mathbf{S}}(x))$ holds.
\end{proof}

The following definition recalls what a maximum, in a partially ordered set, is.

\begin{definition}[Maximum]
Let $(G,\leq)$ be a partially ordered set. An element $u \in G$ will be called a \emph{maximum} if for all $x \in S$, the relation $x \leq u$ holds;
\end{definition}

\begin{remark}[Maxima are unique]
It follows from the anti-symmetry axiom of partial orders that if a partially ordered set $(G,\leq)$ admits a maximum, then there is no other maximum. In other words, the maximum of $G$ is unique.
\end{remark}

\begin{theorem}[Haplotypes]\label{prop:maximum_equivalence_class_recombination_congruence}
Let $\mathbf{S}$ be a Cartesian wide span in $\mathbf{Set}$. Every equivalence class of $G(\mathbf{S})$ admits a maximum for the partial order of Proposition \ref{prop:sub-element_B_2_semimodules_partial_order}.
The maximum in the equivalence class of an element $x$ is given by the element in $F(S)$ that is represented by the subset $\beta_{\mathbf{S}}\pi_{\mathbf{S}}(x) \subseteq \mathbf{S}^{\times}$ in $F(\mathbf{S}^{\times})$ throught the isomorphism $F(S) \to F(\mathbf{S}^{\times})$.
\end{theorem}
\begin{proof}
Proposition \ref{prop:maximum _in_the_class} and Remark \ref{rem:bottom_and_pi_S} implies the equation $\pi_{\mathbf{S}}(x) = \pi_{\mathbf{S}}(\beta_{\mathbf{S}}\pi_{\mathbf{S}}(x))$ and hence the relation $x \sim_{\mathbf{S}}\beta_{\mathbf{S}}\pi_{\mathbf{S}}(x)$. The fact that $\beta_{\mathbf{S}}\pi_{\mathbf{S}}(x)$ is a maximum follows from Definition \ref{def:sub-element_leq} and Proposition \ref{prop:Link_supports_pi_times}.
\end{proof}

\begin{example}[Haplogroups]\label{exa:saturated_haplogroup}
From the point of view of Examples \ref{exa:Recombination_congruences_pi_family_cross} \& \ref{exa:Haplogroups_x_x_x+y}, the maximum element of Proposition \ref{prop:maximum_equivalence_class_recombination_congruence} is what we could call the \emph{saturated} haplogroup of $x$ (or $y$), that is to say the element represented by the following sum in $F(\mathbf{S}^{\times})$.
\[
\begin{array}{c}
(\mathtt{A}\mathtt{C}\mathtt{G}\mathtt{A},\mathtt{C}\mathtt{G}\varepsilon\mathtt{T}\mathtt{T},\mathtt{T}\mathtt{A}\mathtt{G})+
(\mathtt{A}\mathtt{C}\mathtt{G}\mathtt{A},\mathtt{C}\mathtt{G}\varepsilon\mathtt{T}\mathtt{T},\varepsilon\mathtt{C}\mathtt{A})+
(\mathtt{A}\mathtt{C}\mathtt{G}\mathtt{A},\mathtt{C}\mathtt{C}\mathtt{T}\mathtt{A}\mathtt{T},\mathtt{T}\mathtt{A}\mathtt{G})+
(\mathtt{A}\mathtt{C}\mathtt{G}\mathtt{A},\mathtt{C}\mathtt{C}\mathtt{T}\mathtt{A}\mathtt{T},\varepsilon\mathtt{C}\mathtt{A}) +\\
(\mathtt{G}\mathtt{G}\mathtt{T}\mathtt{A},\mathtt{C}\mathtt{G}\varepsilon\mathtt{T}\mathtt{T},\mathtt{T}\mathtt{A}\mathtt{G})+
(\mathtt{G}\mathtt{G}\mathtt{T}\mathtt{A},\mathtt{C}\mathtt{G}\varepsilon\mathtt{T}\mathtt{T},\varepsilon\mathtt{C}\mathtt{A})+
(\mathtt{G}\mathtt{G}\mathtt{T}\mathtt{A},\mathtt{C}\mathtt{C}\mathtt{T}\mathtt{A}\mathtt{T},\mathtt{T}\mathtt{A}\mathtt{G})+
(\mathtt{G}\mathtt{G}\mathtt{T}\mathtt{A},\mathtt{C}\mathtt{C}\mathtt{T}\mathtt{A}\mathtt{T},\varepsilon\mathtt{C}\mathtt{A})\\
\end{array}
\]
\end{example}

\begin{remark}\label{rem:if_wide_span_not_Cartesian}
If the wide span $\mathbf{S}$ of Proposition \ref{prop:maximum_equivalence_class_recombination_congruence} were not Cartesian, then nothing would ensure us that the saturated haplogroup of Example \ref{exa:saturated_haplogroup} can be lifted to $FE_b^{\varepsilon}(\tau)$. The maximum representative of the equivalence class would then be missing certain elements of the saturated haplogroup so that it could not be seen as coming from a product operation.
\end{remark}

\subsection{Recombination semimodules}\label{ssec:recombination_semimodules}
Let $(\Omega,\preceq)$ be a pre-ordered set. We shall let $(\Omega,D)$ denote a finite chromology and $X$ be a functor $\mathbf{Seg}(\Omega) \to \mathbf{Set}$. 

Recall that, on the one hand, for every wide span $\mathbf{S}$ of sets, we have the following pair of projections associated with the congruence $G(\mathbf{S})$.
\[
\xymatrix{
G(\mathbf{S}) \ar@<-1ex>[r]_{\mathsf{prj}_2}\ar@<+1ex>[r]^{\mathsf{prj}_1}&F(S) 
}
\]
On the other hand, we have the recombination cones of $X$ (Definition \ref{def:Relative_definition_families}), which give us a wide span $X(\rho)$ of sets for every cone in $\rho$ in $D$. It therefore comes quite naturally to our mind that we could coequalize the previous pair of arrows with respect to the wide spans $X(\rho)$ to create a functor $\mathbf{Seg}(\Omega) \to B_2\textrm{-}\mathbf{Mod}$. Unfortunately, the mapping $\rho \mapsto X(\rho)$ is unlikely to be functorial, so that we need to work a little bit before obtaining such a functor. Specifically, we need to consider the pair of arrows resulting from the following diagram, in which the symbols $\oplus$ denote coproducts in $B_2\textrm{-}\mathbf{Mod}$ over the finite set of triples $(\upsilon,\rho,f)$ where $\upsilon$ is an object in $\mathbf{Seg}(\Omega)$, $\rho$ is a cone of the form $\Delta_A(\upsilon) \Rightarrow \theta$ in $D$ and $f$ is an arrow $\upsilon \to \tau$ in $\mathbf{Seg}(\Omega)$.
\begin{equation}\label{eq:coequalizer_mod_pedigrad}
\xymatrix@C+18pt{
\mathop{\bigoplus}\limits_{\upsilon,\rho,f:\upsilon \to \tau} GX(\rho)\ar@<-1ex>[r]_-{\mathsf{prj}_2}\ar@<+1ex>[r]^-{\mathsf{prj}_1}&\mathop{\bigoplus}\limits_{\upsilon,\rho,f:\upsilon \to \tau}  FX(\upsilon) \ar[rr]^-{\mathop{\oplus}\limits_{\upsilon,\rho,f}FX(f)}&& FX(\tau)
}
\end{equation}
Then, for every morphism $g:\tau \to \tau'$, we can define the following diagram, which makes the coequalizer of (\ref{eq:coequalizer_mod_pedigrad}) an obvious functor on $\mathbf{Seg}(\Omega)$. 
\[
\xymatrix@C+18pt{
\mathop{\bigoplus}\limits_{\upsilon,\rho,f:\upsilon \to \tau} GX(\upsilon)\ar[d]_{\subset}\ar@<-1ex>[r]_-{\mathsf{prj}_2}\ar@<+1ex>[r]^-{\mathsf{prj}_1}&\mathop{\bigoplus}\limits_{\upsilon,\rho,f:\upsilon \to \tau}  FX(\upsilon) \ar[rr]^-{\mathop{\oplus}\limits_{\upsilon,\rho,f}FX(f)}\ar[d]_{\subset}\ar@{..>}[rrd]|{\mathop{\oplus}\limits_{\upsilon,\rho,f}FX(g \circ f)}&& FX(\tau)\ar[d]^{FX(g)}\\
\mathop{\bigoplus}\limits_{\upsilon,\rho,f:\upsilon \to \tau'} GX(\upsilon)\ar@<-1ex>[r]_-{\mathsf{prj}_2}\ar@<+1ex>[r]^-{\mathsf{prj}_1}&\mathop{\bigoplus}\limits_{\upsilon,\rho,f:\upsilon \to \tau'}  FX(\upsilon) \ar[rr]_-{\mathop{\oplus}\limits_{\upsilon,\rho,f}FX(f)}&& FX(\tau')
}
\]

\begin{definition}[Recombination semimodules]\label{def:Canonical_pedigrads_in_semimodules}
For every finite chromology $(\Omega,D)$ and object $\tau$ in $\mathbf{Seg}(\Omega)$, we will denote by $DX(\tau)$ the coequalizer of (\ref{eq:coequalizer_mod_pedigrad}). The associated functor $DX:\mathbf{Seg}(\Omega) \to B_2\textrm{-}\mathbf{Mod}$ will be called the \emph{recombination semimodule over $X$}
\end{definition}

\begin{convention}[Coequalizer map]\label{conv:coequalizer_map}
For every functor $X:\mathbf{Seg}(\Omega) \to \mathbf{Set}$, the coequalizer map $FX \to DX$ associated with the coequalizer of diagram (\ref{eq:coequalizer_mod_pedigrad}) will be denoted as $q_{X}$.
\end{convention}

\begin{example}[Recombination semimodule for DNA]\label{exa:Recombination_semimodule_for_DNA}
Let $(E,\varepsilon)$ be our usual pointed set $\{\mathtt{A},\mathtt{C},\mathtt{G},\mathtt{T},\varepsilon\}$ and suppose that $\Omega$ is the Boolean pre-ordered set $\{0 \leq 1\}$. As usual, we will let $b \in \Omega$ be equal to 1. We here discuss the form of the recombination semimodule defined over the functor $E_{b}^{\varepsilon}:\mathbf{Seg}(\Omega) \to \mathbf{Set}$.
\smallskip

If we suppose that the cone $\rho:\Delta_A(\tau) \Rightarrow \theta$ given in Example \ref{exa:Relative_definition_families} is one of the cones contained in $D$, then the coequalizer map $q_{E_{b}^{\varepsilon}}:FE_{b}^{\varepsilon}(\tau) \to DE_{b}^{\varepsilon}(\tau)$ will identify the two elements $x$ and $y$ of Example \ref{exa:Recombination_congruences_pi_family}, which are distinct elements of $FE_{b}^{\varepsilon}(\tau)$, as the same element in $DE_{b}^{\varepsilon}(\tau)$. In other words, the following identity will hold
in $DE_{b}^{\varepsilon}(\tau)$.
\[
(\mathtt{A}\mathtt{C}\mathtt{G}\mathtt{A},\mathtt{C}\mathtt{G}\varepsilon\mathtt{T}\mathtt{T},\mathtt{T}\mathtt{A}\mathtt{G}) + (\mathtt{G}\mathtt{G}\mathtt{T}\mathtt{A},\mathtt{C}\mathtt{C}\mathtt{T}\mathtt{A}\mathtt{T},\varepsilon\mathtt{C}\mathtt{A})=(\mathtt{G}\mathtt{G}\mathtt{T}\mathtt{A},\mathtt{C}\mathtt{G}\varepsilon\mathtt{T}\mathtt{T},\varepsilon\mathtt{C}\mathtt{A}) + (\mathtt{A}\mathtt{C}\mathtt{G}\mathtt{A},\mathtt{C}\mathtt{C}\mathtt{T}\mathtt{A}\mathtt{T},\mathtt{T}\mathtt{A}\mathtt{G})
\]
More generally, the coequalizer of diagram (\ref{eq:coequalizer_mod_pedigrad}) will force any pair $(x,y)$ contained in the recombination congruences resulting from the recombination cones of $E_{b}^{\varepsilon}$ (see Definition \ref{def:Relative_definition_families}) to be identified (i.e. $x = y$) in the recombination semimodule $DE_{b}^{\varepsilon}(\tau)$.

When the functor $E_{b}^{\varepsilon}:\mathbf{Seg}(\Omega) \to \mathbf{Set}$ is a $\mathcal{W}^{\textrm{bij}}$-pedigrad and $\rho$ is the only cone of the ambient chromology, Theorem \ref{prop:maximum_equivalence_class_recombination_congruence}  ensures that the equivalence classes associated with these identifications can be represented by a product operation on the local patches of every element contained in these equivalence classes (see Example \ref{exa:saturated_haplogroup}). For chromologies containing more than one cone, the recombination congruences would mix each other so that they would generate recombination relations that could be more refined than those specified by the cones themselves. %An argument similar to the one used in Example \ref{exa:Haplogroups_x_x_x+y} can show that the representative of such an equivalence class would be given by the sum of all its elements.
\end{example}

\subsection{Logical systems for homologous recombination}
We will denote by $\mathcal{W}^{\textrm{mon}}$ the class of wide spans $\mathbf{S}=\{S \to S_i\}_{i \in [k]}$ in $B_2\textrm{-}\mathbf{Mod}$ whose product adjoint arrows $S \to \mathbf{S}^{\times}$ is a monomorphism in $B_2\textrm{-}\mathbf{Mod}$.

\begin{remark}[Homologous recombination]
A $\mathcal{W}^{\textrm{mon}}$-pedigrad is a pedigrad in which the recombination congruences resulting from the logical system $(B_2\textrm{-}\mathbf{Mod},\mathcal{W}^{\textrm{mon}})$ can be seen as identities (see Proposition \ref{prop:representable_pedigrad_quotiented_by_recombinations}). Another way to put it is to say that a $\mathcal{W}^{\textrm{mon}}$-pedigrad is a pedigrad in which homologous recombination happens (see Example \ref{exa:Recombination_congruences_pi_family}).
\end{remark}

\subsection{Recombination schemes and pedigrads}\label{ssec:Recombination_schemes_and_pedigrads}
This section determines the right conditions to make a recombination semimodule, as given in Definition \ref{def:Canonical_pedigrads_in_semimodules}, a $\mathcal{W}^{\textrm{mon}}$-pedigrad.

\begin{definition}[Recombination prescheme]
We will call a \emph{recombination prescheme} any triple $(\Omega,D,X)$ where $(\Omega,D)$ is a finite chromology and $X$ is a functor $\mathbf{Seg}(\Omega) \to \mathbf{Set}$.
%\footnote{The qualifier `recombination' emphasizes that the fact that the triple $(\Omega,D,X)$ is defined with respect to a finite chromology. Any triple of the form $(\Omega,D,X)$ where $(\Omega,D)$ is not necessarily finite should be called a prescheme.} 
\end{definition}

\begin{example}[Recombination prescheme]\label{exa:Canonical_recombination_prescheme}
For every finite chromology $(\Omega,D)$, element $b \in \Omega$ and pointed set $(E,\varepsilon)$, the triple $(\Omega,D,E_b^{\varepsilon})$ is an obvious recombination prescheme.
\end{example}

\begin{definition}[Irreducibility]\label{def:irreducible_object}
Let $(\Omega,D,X)$ be a recombination prescheme. An object $\tau$ in $\mathbf{Seg}(\Omega)$ will be said to be \emph{irreducible} for $(\Omega,D,X)$ if for every arrow $f:\upsilon \to \tau$ in $\mathbf{Seg}(\Omega)$, the image $FX(f):FX(\upsilon) \to FX(\tau)$ coequalizes the following pair of arrows for every cone $\rho:\Delta_A(\upsilon) \Rightarrow \theta$ in $D$.
\[
\xymatrix@C+18pt{
GX(\rho)\ar@<-1ex>[r]_-{\mathsf{prj}_2}\ar@<+1ex>[r]^-{\mathsf{prj}_1}& FX(\upsilon) 
}
\]
\end{definition}

\begin{remark}[Coequalizing arrows]\label{rem:Coequalizing_arrows}
Observe that, from the point of view of Definition \ref{def:Recomb_congruences}, when $\mathbf{S}$ denotes a general wide span, the obvious composite $F(S) \to F(\mathbf{S})^{\times} \to F(S_i)$ coequalizes the recombination congruence $G(\mathbf{S}) \rightrightarrows F(S)$. This means that if $\mathbf{S}$ is now the wide span given by the $\rho$-recombination cone of a functor $X:\mathbf{Seg}(\Omega) \to \mathbf{Set}$ (Definition \ref{def:Relative_definition_families}) for some cone $\rho$ in $D$, then the morphism 
\[
FX(\rho_a):FX(\tau) \to FX(\theta(a))
\]
coequalizes the recombination congruence $GX(\rho) \rightrightarrows FX(\tau)$ for every $a \in A$. 
\end{remark}

\begin{example}[Coequalizing arrows]
If we let $\mathbf{S}$ denote the recombination cone $E_b^{\varepsilon}(\rho)$ of Example \ref{exa:Relative_definition_families}, then Remark \ref{rem:Coequalizing_arrows} implies that the recombination congruence $GE_b^{\varepsilon}(\rho) \rightrightarrows FE_b^{\varepsilon}(\tau)$ is coequalized by the image of the following arrows via the functor $FE_b^{\varepsilon}:\mathbf{Seg}(\Omega) \to \mathbf{Set}$.
\[
\begin{array}{l}
\rho_{a_1}:\xymatrix@C-30pt@R-20pt{
(\bullet&\bullet&\bullet)&(\circ&\circ)&(\bullet&\bullet&\bullet&\bullet)&(\bullet&\bullet&\bullet&\bullet&\bullet)&(\circ&\circ&\circ)&(\circ)\ar[rr]
&\quad\quad\quad&
(\circ&\circ&\circ)&(\circ&\circ)&(\bullet&\bullet&\bullet&\bullet)&(\circ&\circ&\circ&\circ&\circ)&(\circ&\circ&\circ)&(\circ)
}\\
\rho_{a_2}:\xymatrix@C-30pt@R-20pt{
(\bullet&\bullet&\bullet)&(\circ&\circ)&(\bullet&\bullet&\bullet&\bullet)&(\bullet&\bullet&\bullet&\bullet&\bullet)&(\circ&\circ&\circ)&(\circ)\ar[rr]
&\quad\quad\quad&
(\circ&\circ&\circ)&(\circ&\circ)&(\circ&\circ&\circ&\circ)&(\bullet&\bullet&\bullet&\bullet&\bullet)&(\circ&\circ&\circ)&(\circ)
}\\
\rho_{a_3}:\xymatrix@C-30pt@R-20pt{
(\bullet&\bullet&\bullet)&(\circ&\circ)&(\bullet&\bullet&\bullet&\bullet)&(\bullet&\bullet&\bullet&\bullet&\bullet)&(\circ&\circ&\circ)&(\circ)\ar[rr]
&\quad\quad\quad&
(\bullet&\bullet&\bullet)&(\circ&\circ)&(\circ&\circ&\circ&\circ)&(\circ&\circ&\circ&\circ&\circ)&(\circ&\circ&\circ)&(\circ)
}
\end{array}
\]
Indeed, we can check that these arrows would send the pair of equivalent elements
\[
x=(\mathtt{A}\mathtt{C}\mathtt{G}\mathtt{A},\mathtt{C}\mathtt{G}\varepsilon\mathtt{T}\mathtt{T},\mathtt{T}\mathtt{A}\mathtt{G}) + (\mathtt{G}\mathtt{G}\mathtt{T}\mathtt{A},\mathtt{C}\mathtt{C}\mathtt{T}\mathtt{A}\mathtt{T},\varepsilon\mathtt{C}\mathtt{A})\quad y=(\mathtt{G}\mathtt{G}\mathtt{T}\mathtt{A},\mathtt{C}\mathtt{G}\varepsilon\mathtt{T}\mathtt{T},\varepsilon\mathtt{C}\mathtt{A}) + (\mathtt{A}\mathtt{C}\mathtt{G}\mathtt{A},\mathtt{C}\mathtt{C}\mathtt{T}\mathtt{A}\mathtt{T},\mathtt{T}\mathtt{A}\mathtt{G})
\]
that was given in Example \ref{exa:Recombination_congruences_pi_family} to the same elements, which are displayed below.
\[
\begin{array}{ccccc}
\cellcolor[gray]{0.8}\textrm{via }FE_b^{\varepsilon}(\rho_{a_1})\cellcolor[gray]{0.8}&\cellcolor[gray]{0.8}\quad\quad\quad&\cellcolor[gray]{0.8}\textrm{via }FE_b^{\varepsilon}(\rho_{a_2})&\cellcolor[gray]{0.8}\quad\quad\quad&\cellcolor[gray]{0.8}\textrm{via }FE_b^{\varepsilon}(\rho_{a_3})\\
\hline
&&&&\vspace{-5pt}\\
\mathtt{A}\mathtt{C}\mathtt{G}\mathtt{A}+\mathtt{G}\mathtt{G}\mathtt{T}\mathtt{A}&&\mathtt{C}\mathtt{G}\varepsilon\mathtt{T}\mathtt{T}+\mathtt{C}\mathtt{C}\mathtt{T}\mathtt{A}\mathtt{T}&&\mathtt{T}\mathtt{A}\mathtt{G}+\varepsilon\mathtt{C}\mathtt{A}
\end{array}
\]
As will be shown in Proposition \ref{prop:single_cone_irreducible}, this means that the codomains of the arrows $\rho_{a_1}$, $\rho_{a_2}$ and $\rho_{a_3}$ may be good candidates for being irreducible objects with respect to a recombination prescheme of the form given in Example \ref{exa:Canonical_recombination_prescheme}.
\end{example}

\begin{proposition}\label{prop:single_cone_irreducible}
Let $(\Omega,D,X)$ be a recombination prescheme. Suppose that $D$ contains a unique cone of the form $\rho:\Delta_A(\tau) \Rightarrow \theta$. For every object $a \in A$, the object $\theta(a)$ in $\mathbf{Seg}(\Omega)$ is irreducible.
\end{proposition}
\begin{proof}
Let us now show that if $\rho$ is the only cone of $D$, then the object $\theta(a)$ is irreducible for $(\Omega,D,X)$. By definition of a chromology (section \ref{sec:chromologies}), the arrow $\rho_a:\tau \to \theta(a)$ is an arrow in a category of quasi-homologous segments (see Definition \ref{sec:chromologies}). By Proposition \ref{prop:quasi_homologous_preordered_category}, this means that this arrow is the only arrow of type $\tau \to \theta(a)$ in $\mathbf{Seg}(\Omega)$. By Remark \ref{rem:Coequalizing_arrows}, this means that for every arrow $\tau \to \theta(a)$ in $\mathbf{Seg}(\Omega)$, its image $FX(\tau) \to FX(\theta(a))$ coequalizes the pair $GX(\rho) \rightrightarrows FX(\tau)$. Since $\rho$ is the only cone of $D$, this means that the image of every arrow $\upsilon \to \theta(a)$ in $\mathbf{Seg}(\Omega)$ via the functor $FX:\mathbf{Seg}(\Omega) \to \mathbf{Set}$ coequalizes the pair $GX(\rho') \rightrightarrows FX(\upsilon)$ for every cone $\rho':\Delta_{A'}(\upsilon) \Rightarrow \theta'$ in $D$.
\end{proof}

\begin{proposition}\label{prop:irreducible_coequalizer_map_isomorphism}
Let $(\Omega,D,X)$ be a recombination prescheme. For every irreducible object $\tau$ in $\mathbf{Seg}(\Omega)$, the coequalizer map $q_{X}:FX(\tau) \to DX(\tau)$ (Convention \ref{conv:coequalizer_map})  associated with the recombination semimodule over $X$ (Definition \ref{def:Canonical_pedigrads_in_semimodules}) is an isomorphism.
\end{proposition}
\begin{proof}
By Definition \ref{def:irreducible_object} and universality of a coequalizer.
\end{proof}

\begin{definition}[Recombination scheme]
A \emph{recombination scheme} is a recombination prescheme $(\Omega,D,X)$ such that for every cone $\rho:\Delta_A(\tau) \Rightarrow \theta$ in $D$ and object $a \in A$, the object $\theta(a)$ is irreducible in $\mathbf{Seg}(\Omega)$.
\end{definition}

The following theorem shows that the recombination semimodule associated with a recombination scheme is a $\mathcal{W}^{\textrm{mon}}$-pedigrad.

\begin{theorem}\label{theo:representable_pedigrad_E_b_varepsilon}
Let $(\Omega,D,X)$ be a recombination scheme.
For every cone $\rho:\Delta_A(\tau) \Rightarrow \theta$ in $D$, the canonical arrow $\iota:DX(\tau) \to \prod_{a \in A}DX(\theta(a))$ is a monomorphism in $B_2\textrm{-}\mathbf{Mod}$.
\end{theorem}
\begin{proof}
By Definition \ref{def:Recomb_congruences}, the pullback of two copies of the canonical arrow $\pi_{X(\rho)}$ is the recombination congurence $GX(\rho)\rightrightarrows FX(\tau)$. If we denote by $p_1,p_2:P \rightrightarrows DX(\tau)$ the pullback of two copies of the arrow $\iota$ (see statement), the naturality of the coequalizer map $q_X:FX \to DX$ gives us an arrow $\lambda:FX(\rho) \to P$ making the following diagram commute.
\begin{equation}\label{eq:proof_mono_pedigrad}
\xymatrix@-10pt{
&GX(\rho)\ar[dd]|\hole_<<<<<{\mathsf{prj}_2}\ar[ld]_{\mathsf{prj}_1}\ar@{-->}[rr]^{\lambda}&&P\ar[dd]_{p_2}\ar[ld]_{p_1}\\
FX(\tau)\ar[rr]^>>>>>>>>{q_X}\ar[dd]_{\pi_{X(\rho)}}&&DX(\tau)\ar[dd]^>>>>>>{\iota}&\\
&FX(\tau)\ar[ld]_{\pi_{X(\rho)}}\ar[rr]|\hole^<<<<<<<<<{q_X}&&DX(\tau)\ar[ld]^{\iota}\\
\prod_{a \in A}FX(\theta(a))\ar[rr]_{\cong}^{\prod_a q_{X}}&&\prod_{a \in A}DX(\theta(a))&
}
\end{equation}
Let us show that $\lambda$ is an epimorphism in $B_2\textrm{-}\mathbf{Mod}$ by showing it is orthogonal with respect to the $B_2$-semimodule $B_2$ (see Proposition \ref{prop:characterization_epi}). 
First, because $(\Omega,D,X)$ is a recombination scheme, Proposition \ref{prop:irreducible_coequalizer_map_isomorphism} implies that the bottom front arrow of diagram (\ref{eq:proof_mono_pedigrad}) is an isomorphism. 
Second, because $q_X:FX(\tau) \to DX(\tau)$ is a coequalizer map, it is  orthogonal with respect to the $B_2$-semimodule $B_2$ (Example \ref{exa:characterization_epi_coequalizer_maps}).
These two facts imply that, for every arrow $x:B_2 \to P$, the composite arrows $p_1 \circ x:B_2 \to DX(\tau)$ and  $p_2 \circ x:B_2 \to DX(\tau)$ admit lifts $h_1:B_2 \to FX(\tau)$ and $h_2:B_2 \to FX(\tau)$ for which the following diagram commutes.
\[
\xymatrix{
B_2\ar[r]^{h_1}\ar[d]_{h_2}&FX(\tau)\ar[d]^{\pi_{X(\rho)}}\\
FX(\tau)\ar[r]_{\pi_{X(\rho)}}&*+!L(.7){\prod_{a \in A}FX(\theta(a))}
}
\]
Applying the universality property of $GX(\rho)$ on the previous diagram provides an arrow $h:B_2 \to GX(\rho)$ for which the equation $\lambda \circ h = x$ holds. In other words, the arrow $\lambda$ is orthogonal to the $B_2$-semimodule $B_2$. Now, because the equation $q_X \circ \mathsf{prj}_1 = q_X \circ \mathsf{prj}_2$ holds by definition of $q_X$ and because $\lambda$ is an epimorphism (Proposition \ref{prop:characterization_epi}), the two arrows $p_1,p_2:P \rightrightarrows DX(\tau)$ must be equal (see diagram (\ref{eq:proof_mono_pedigrad})). Because this pair of arrows is also the pullback of two copies of $\iota$, the arrow $\iota$ is a monomorphism (Proposition \ref{prop:characterization_mono}).
\end{proof}

\subsection{Presentable functors}\label{ssec:presentable_functors}
Here, the terminology `presentable' refers to the idea of being the quotient of a free object of a given type. A presentable functor will usually encompass the type of information that can be observed from a \emph{set} of data with which one wants to analyze in $B_2\textrm{-}\mathbf{Mod}$. Throughout this section, we will let $(\Omega,\preceq)$ denote a pre-ordered set.

\begin{definition}[Presentable functors]
A functor $P:\mathbf{Seg}(\Omega) \to B_2\textrm{-}\mathbf{Mod}$ will be said to be \emph{presentable} over a functor $X:\mathbf{Seg}(\Omega) \to \mathbf{Set}$ if it is equipped with a morphism $q:FX \to P$ that is the coequalizer map of a pair of parallel arrows, as shown below, in $[\mathbf{Seg}(\Omega),B_2\textrm{-}\mathbf{Mod}]$.
\begin{equation}\label{eq:coequalizer_representable_pedigrads}
\xymatrix{
Q\ar@<-1ex>[r]_-{p_2}\ar@<+1ex>[r]^-{p_1}&FX
}
\end{equation}
\end{definition}

\begin{example}[Trivial example]
For every pointed set $(E,\varepsilon)$ and element $b \in \Omega$, the functor $FE_{b}^{\varepsilon}$ is obviously presentable for the trivial coequalizer diagram given below.
\[
\xymatrix{
FE_{b}^{\varepsilon}\ar@<-1ex>[r]_-{\mathrm{id}}\ar@<+1ex>[r]^-{\mathrm{id}}&FE_{b}^{\varepsilon}
}
\]
\end{example}

\begin{example}[Recombination]\label{exa:presentable_alphabet_recomb_semimodule}
For every pointed set $(E,\varepsilon)$, element $b \in \Omega$ and finite chromology $(\Omega,D)$, the recombination semimodule $DE_{b}^{\varepsilon}:\mathbf{Seg}(\Omega) \to B_2\textrm{-}\mathbf{Mod}$ (see Definition \ref{def:Canonical_pedigrads_in_semimodules}) is the coequalizer for the pair given in (\ref{eq:coequalizer_mod_pedigrad}). It is therefore presentable for the coequalizer map $q:FE_{b}^{\varepsilon} \to DE_{b}^{\varepsilon}$
\end{example}

\begin{definition}[Presentable morphisms]\label{def:presentable_morphisms}
Let $(P_1,q_1)$ and $(P_2,q_2)$ be two presentable functors over two functors $X_1$ and $X_2$, respectively. A \emph{presentable morphism} from $(P_1,q_1)$ to $(P_2,q_2)$ is a commutative square on the following form in $[\mathbf{Seg}(\Omega),B_2\textrm{-}\mathbf{Mod}]$.
\[
\xymatrix{
FX_1\ar[d]_{g}\ar[r]^-{q_1}&P_1\ar[d]^f\\
FX_2\ar[r]^-{q_2}&P_2\\
}
\]
\end{definition}

\begin{remark}\label{rem:b_presentable_category}
Any component-wise composition of presentable morphisms is presentable.
\end{remark}

\begin{example}[Mutations as presentable morphisms]\label{exa:Mutations_presentable_morphisms}
Let $(E,\varepsilon)$ be a pointed set, $b$ an element in $\Omega$ and $(\Omega,D)$ be a finite chromology. One of the most trivial\footnote{More exotic examples are possible.} types of presentable morphisms of the form $DE_{b}^{\varepsilon} \to DE_{b}^{\varepsilon}$ (for the structure given in Example \ref{exa:presentable_alphabet_recomb_semimodule}) can arise from morphisms $FE_b^{\varepsilon} \to FE_b^{\varepsilon}$ whose mappings are generated by functions $\nu:E \to F(E)$ that assign each element $x \in E$ to a (finite) sum of elements in $F(E)$. The morphism $FE_b^{\varepsilon} \to FE_b^{\varepsilon}$ then sends any word of the form $\mathtt{X}_1\dots \mathtt{X}_n$ in $E^{\times n}$ to the element of $F(E^{\times n})$ represented by the following finite subset of $E^{\times n}$ (as in Example \ref{exa:Recombination_congruences_pi_family_cross}). 
\[
\mathsf{Supp}(\nu(\mathtt{X}_1)) \times \dots \times \mathsf{Supp}(\nu(\mathtt{X}_n))
\]
The following table gives an example when $(E,\varepsilon)$ is our usual pointed set $\{\mathtt{A},\mathtt{C},\mathtt{G},\mathtt{T},\varepsilon\}$.
\smallskip

\begin{center}
\begin{tabular}{|c|c|ccc|}
\hline
\cellcolor[gray]{0.8}$\mathtt{X}$ & \cellcolor[gray]{0.8}$\nu(\mathtt{X})$ & \cellcolor[gray]{0.8}$\mathtt{X}_1\dots\mathtt{X}_n$& \cellcolor[gray]{0.8} $\mapsto$ & \cellcolor[gray]{0.8}$\mathsf{Supp}(\nu(\mathtt{X}_1))\times \dots \times \mathsf{Supp}(\nu(\mathtt{X}_n))$\\
\hline
$\mathtt{A}$ & $\varepsilon$ & \multirow{2}{*}{$\mathtt{A}\mathtt{C}$}& \multirow{2}{*}{$\mapsto$}& \multirow{2}{*}{$\varepsilon\mathtt{A} + \varepsilon\mathtt{C} + \varepsilon\mathtt{T}$}\\
$\mathtt{C}$ & $\mathtt{A} + \mathtt{C} + \mathtt{T}$ & &&\\
\cline{3-5}
$\mathtt{G}$ & $\mathtt{T} + \mathtt{C}$ &\multirow{3}{*}{$\varepsilon\mathtt{G}\mathtt{T}$} & \multirow{3}{*}{$\mapsto$} & \multirow{3}{*}{$\varepsilon\mathtt{T}\varepsilon + \varepsilon\mathtt{G}\varepsilon +\varepsilon\mathtt{T}\mathtt{T} + \varepsilon\mathtt{G}\mathtt{T}$}\\
$\mathtt{T}$ & $\varepsilon + \mathtt{T}$&&&\\
$\varepsilon$&$\varepsilon$&&&\\
\hline 
\end{tabular}
\vspace{5pt}
\end{center}
As one can imagine, this type of morphism can be used to represent DNA mutations. As noticed in Example \ref{exa:Mutations_in_set_are_spans}, taking $(E,\varepsilon)$ to be the pointed set $\{\mathtt{A},\mathtt{C},\mathtt{G},\mathtt{T},\varepsilon\}$ would result in mutations that are too systematic to be realistic. As mentioned thereof, a better parameterization of the alphabet could then be used to make these mutations more realistic. The main difference with Example \ref{exa:Mutations_in_set_are_spans} is that $B_2$-modules now allow us to incorporate some polymorphism in the mutation. The table given below associate examples of mappings (on the left) with certain mutation types (on the right) and shows how the sums can be used to generate polymorphic substitution mutations.
\smallskip

\begin{center}
\begin{tabular}{|c|c|c|}
\hline
\cellcolor[gray]{0.8}Examples of mappings&\cellcolor[gray]{0.8}$\Rightarrow$ &\cellcolor[gray]{0.8}Types of mutations\\
\hline
$\mathtt{A} \mapsto \varepsilon$; $\quad\mathtt{C} \mapsto \varepsilon$; $\quad\mathtt{G} \mapsto \varepsilon$; $\quad\mathtt{T} \mapsto \varepsilon$&$\Rightarrow$ & Deletion mutations\\
\hline
$\mathtt{A} \mapsto \mathtt{G}+\mathtt{T}$; $\quad\mathtt{A} \mapsto \mathtt{A}+\mathtt{T}$; $\quad\mathtt{A} \mapsto \mathtt{T}$; etc.&$\Rightarrow$ & Substitution mutations\\
\hline
\end{tabular}
\vspace{5pt}
\end{center}
As pointed out in Example \ref{exa:Mutations_in_set_are_spans}, the fact that the element $\varepsilon$ always needs to be mapped to itself (by definition of the concept of presentability over $E_b^{\varepsilon}$), implies that insertion mutations (if not all) should be studied through spans of presentable morphisms.
\end{example}

\begin{example}[Transcription as presentable morphisms]\label{exa:Transcription_translation_presentable_morphisms}
Another interpretation of presentable morphisms could be transcription processes (see Example \ref{exa:Transcription_in_set}).
\end{example}

\subsection{Presentable pedigrads and their universal properties}\label{ssec:presentable_pedigrads}
Let $(\Omega,\preceq)$ be a pre-ordered set equipped with a finite chromology structure $(\Omega,D)$.

The proposition given below says that if a presentable functor $\mathbf{Seg}(\Omega) \to B_2\textrm{-}\mathbf{Mod}$ is a $\mathcal{W}^{\textrm{mon}}$-pedigrad, then every pair of elements contained in any of its associated recombination congruences will be equated in the images of $P$. In Example \ref{exa:More_quotients}, we will show how, in general, equations can be used to describe biological phenomena.

\begin{proposition}\label{prop:representable_pedigrad_quotiented_by_recombinations}
 Let $P$ denote a $\mathcal{W}^{\textrm{mon}}$-pedigrad that is presentable over a functor $X:\mathbf{Seg}(\Omega) \to \mathbf{Set}$ and $\rho$ be a cone in $D$. The coequalizer map $FX \to P$ coequalizes the recombination congruence associated with the $\rho$-recombination cone of $X$ (Definition \ref{def:Relative_definition_families}).
\end{proposition}
\begin{proof}
Let $\rho$ be a cone of the form $\Delta_A(\tau) \Rightarrow \theta$. The naturality of coequalizer map $q:FX \to P$ gives us the commutative square displayed below, on the right, while the recombination congruence associated with the $\rho$-recombination cone of $X$ is given next to it, on the left.
\[
\xymatrix{
GX(\rho) \ar@<+1.2ex>[r]^{\mathsf{prj}_1}\ar@<-1.2ex>[r]_{\mathsf{prj}_2}&FX(\tau)\ar[d]_{\pi_{X(\rho)}}\ar[r]^{q_{\tau}}&P(\tau)\ar[d]^{\iota}\\
&*+!R(.7){\prod_{a \in A}FX(\theta(a))} \ar[r]_{\prod_a q_{\theta(a)}}&*+!L(.7){\prod_{a \in A}P(\theta(a))}
}
\]
Since, by definition, the arrow $\pi_{X(\rho)}$ coequalizes the pair $(\mathsf{prj}_1,\mathsf{prj}_2)$, so does the composite $\iota \circ q_{\tau}$. Since $P$ is a $\mathcal{W}^{\textrm{mon}}$-pedigrad, the arrow $\iota$ is a monomorphism, which implies that $q_{\tau}$ coequalizes the $(\mathsf{prj}_1,\mathsf{prj}_2)$ itself (see previous diagram).
\end{proof}

\begin{example}[More equations]\label{exa:More_quotients}
A pedigrad may satisfy various types of equation. For illustration, if we let $\Omega$ be the pre-ordered set $\{0 \leq 1\}$, then any pedigrad $P:\mathbf{Seg}(\Omega) \to B_2\textrm{-}\mathbf{Mod}$ is equipped with a morphism of the following form, where the codomain may contain more equations than the domain (while the equations satisfied in its domain must be sent to similar equations in its codomain).
\begin{equation}\label{eq:quotients_pedigrad}
P\left(
\xymatrix@C-30pt{(\bullet&\bullet&\bullet)&(\bullet&\bullet&\bullet)&(\bullet&\bullet&\bullet)}
\right)
\longrightarrow
P\left(
\xymatrix@C-30pt{(\bullet&\bullet&\bullet&\bullet&\bullet&\bullet&\bullet&\bullet&\bullet)}
\right)
\end{equation}
Interestingly, if we let $P$ be presentable over a functor $E_b^{\varepsilon}$ where $(E,\varepsilon)$ is taken to be the $\{\mathtt{A},\mathtt{C},\mathtt{G},\mathtt{T},\varepsilon\}$, we can give a biological interpretation to the following equations.

1) \textbf{Nullomers:} take the  word $\mathtt{G}\mathtt{A}\mathtt{C}\mathtt{C}\mathtt{G}\mathtt{T}\mathtt{A}\mathtt{A}\mathtt{G}$, which has representatives in each of the objects displayed in (\ref{eq:quotients_pedigrad}).
Even though these representatives look the same, they may be read differently in the domain and the codomain of (\ref{eq:quotients_pedigrad}). For instance, the word $\mathtt{G}\mathtt{A}\mathtt{C}\mathtt{C}\mathtt{G}\mathtt{T}\mathtt{A}\mathtt{A}\mathtt{G}$ could be torsion-free in the domain of (\ref{eq:quotients_pedigrad}) but could be subject to the following equation in the codomain.
\[
(\mathtt{G}\mathtt{A}\mathtt{C}\mathtt{C}\mathtt{G}\mathtt{T}\mathtt{A}\mathtt{A}\mathtt{G}) = 0
\]
Such an equation could, for example, mean that the segment $\mathtt{G}\mathtt{A}\mathtt{C}\mathtt{C}\mathtt{G}\mathtt{T}\mathtt{A}\mathtt{A}\mathtt{G}$ is what one calls a \emph{nullomer} \cite{Nullomer}, namely a short DNA sequence that cannot appear in the genome of a species. On the other hand, the element $\mathtt{G}\mathtt{A}\mathtt{C}\mathtt{C}\mathtt{G}\mathtt{T}\mathtt{A}\mathtt{A}\mathtt{G}$ would not be equal to 0 in the domain of (\ref{eq:quotients_pedigrad}) because the codons that compose it would not be specific to nullomers.

2) \textbf{Translation:} take the words $\mathtt{A}\mathtt{G}\mathtt{T}\mathtt{C}\mathtt{A}\mathtt{T}\mathtt{G}\mathtt{G}\mathtt{G}$ and $\mathtt{A}\mathtt{G}\mathtt{C}\mathtt{C}\mathtt{A}\mathtt{C}\mathtt{G}\mathtt{A}\mathtt{T}$, which can be viewed as elements living in the objects of (\ref{eq:quotients_pedigrad}).
These two elements could be distinct in the domain of (\ref{eq:quotients_pedigrad}), but could be sent to the same element in its codomain.
\[
(\mathtt{A}\mathtt{G}\mathtt{T}\mathtt{C}\mathtt{A}\mathtt{T}\mathtt{G}\mathtt{G}\mathtt{G})=(\mathtt{A}\mathtt{G}\mathtt{C}\mathtt{C}\mathtt{A}\mathtt{C}\mathtt{G}\mathtt{A}\mathtt{T})
\]
Such an equation could, for example, mean that the codon translations of these two words, namely $\mathtt{Ser}\textrm{-}\mathtt{His}\textrm{-}\mathtt{Gly}$ and $\mathtt{Ser}\textrm{-}\mathtt{His}\textrm{-}\mathtt{Asp}$ -- which are, noticeably, different -- give two tripeptides whose properties are seemingly the same \cite{Li_Dipeptide}.
\end{example}

The proposition given below is more of a categorical result showing that presentable pedigrads are endowed with a universal property. This property says that any logical reasoning done in a recombination semimodule $DX$ can be transferred to any pedigrad that is presentable over the underlying functor $X$ (see Remark \ref{rem:Universal_property_presentable_pedigrads}).

\begin{theorem}\label{theo:universal_property}
Let $(P,p)$ be a $\mathcal{W}^{\textrm{mon}}$-pedigrad for $(\Omega,D)$ that is presentable over a functor $X:\mathbf{Seg}(\Omega) \to \mathbf{Set}$. For every morphism $f:P \to Q$ in $[\mathbf{Seg}(\Omega), B_2\textrm{-}\mathbf{Mod}]$, there exists a unique morphism $f':DX \to Q$ making the following diagram commute. 
\[
\xymatrix{
FX\ar[d]_{q_X} \ar[r]^{f \circ p}& Q\\
DX\ar[ru]_{f'}&
}
\]
\end{theorem}
\begin{proof}
According to Proposition \ref{prop:representable_pedigrad_quotiented_by_recombinations}, the coequalizer map $FX \to P$ makes the following diagram commute for every cone $\rho:\Delta_A(\upsilon)\Rightarrow \theta$ in $D$.
\[
\xymatrix{
GX(\rho)\ar@<+1.2ex>[r]\ar@<-1.2ex>[r]&FX(\upsilon)\ar[r]^{p_{\upsilon}}&P(\upsilon)\\
}
\]
Since, for every morphism $g:\upsilon \to \tau$ in $\mathbf{Seg}(\Omega)$, the diagram given below, on the left, commutes, the corresponding diagram on the right must also commute.
\[
\begin{array}{l}
\xymatrix{
FX(\upsilon)\ar[d]_{FX(g)}\ar[r]^{p_{\upsilon}}&P(\upsilon)\ar[d]^{P(g)}\\
FX(\tau)\ar[r]_{p_{\tau}}&P(\tau)
}
\end{array}
\quad
\Rightarrow
\quad
\begin{array}{l}
\xymatrix{
GX(\rho)\ar@<+1.2ex>[r]\ar@<-1.2ex>[r]&FX(\upsilon)\ar[r]^{FX(g)}&FX(\tau)\ar[r]^{p_{\tau}}&P(\tau)\\
}
\end{array}
\]
It then follows that the coequalizer map $p:FX \to P$ coequalizes the pair of arrows of (\ref{eq:coequalizer_mod_pedigrad}). By universal property, there must exists a unique morphism $f':DX \to P$ in $[\mathbf{Seg}(\Omega), B_2\textrm{-}\mathbf{Mod}]$ making the diagram of the statement commute.
\end{proof}

\begin{remark}[Presentable morphisms]\label{rem:Universal_property_presentable_pedigrads}
In Theorem \ref{theo:universal_property}, if one takes $f$ to be the identity on $P$, then the statement implies that there exists a unique presentable morphism $(DX,q_X) \to (P,p)$ of the following form.
\[
\xymatrix{
FX \ar[d]_{q_X} \ar@{=}[r]& FX\ar[d]^p\\
DX\ar[r]_{p'}&P
}
\]
Then, if one takes $f$ to be a presentable morphism $(P,p) \to (Q,q)$, then Remark \ref{rem:b_presentable_category} implies that the composite $f \circ p'$, which must be equal to $f'$ by Theorem \ref{theo:universal_property}, induces a presentable morphism $(DX,q_X) \to (Q,q)$.
\end{remark}

\section{Mapping functions and Pedigrads}\label{Genetic_linkage_and_mapping_functions}

The goal of this section is to show how one can retrieve the mapping functions used in gene linkage \cite{Speed_GMF,McPeek,Haldane} from pedigrads taking their values in the category of $B_2$-semimodules. The idea is that the pedigrads are supposed to help us reconstruct the spaces of events that one wants to measure from a logical reasoning in the chromology. 

First, recall that mapping functions plot the probability of observing a certain number of cross-overs between pairs of markers on distinct chromosomes as a function of their distances on their respective chromosomes. They usually take the form of (non-normalized) cumulative distribution functions e.g. $y=(1-e^{-2x})/2$ (see \cite{Haldane,ZhaoSpeed}). These functions are important because they give a notion of distance between genes by taking into account the frequency of recombination that separate them \cite{Speed_GMF}. This genetic distance can then be used to study genetic diseases \cite{Copeland,Ott}.

\subsection{Probability theory and \texorpdfstring{$B_2$}{Lg}-semimodules}
The goal of this section is to reformulate concepts pertaining to Probability Theory \cite{Loeve} in the language of $B_2$-semimodules in order to reformulate the results of \cite{Haldane} in the context of pedigrads. The notions that we shall use may be slightly weaker or stronger than those used in the literature (see \cite[Page 8]{Loeve}).

\begin{definition}[Event spaces]\label{def:Event_space}
For every set $S$, we will speak of an \emph{event space} over $S$ to refer to a subset $\mathcal{F} \subseteq \mathcal{P}(S)$ of finite subsets of $S$ that contains the empty set and that is stable under binary unions.
\end{definition}

\begin{remark}[Event spaces as \texorpdfstring{$B_2$}{Lg}-semimodules]
Definition \ref{def:Event_space} is equivalent to requiring $\mathcal{F}$ to be a sub-object of the free $B_2$-semimodule $F(S)$ in $B_2\textrm{-}\mathbf{Mod}$ (see Example \ref{exa:power_set_of_finite_sets} and Remark \ref{rem:description_free_b_2_semimodule}), which is to say a $B_2$-semimodule $\mathcal{F}$ that is equipped with a monomorphism $\mathcal{F} \hookrightarrow F(S)$ in $B_2\textrm{-}\mathbf{Mod}$.
\end{remark}

\begin{example}[Haldane's event space]\label{exa:Haldane_event_space}
In \cite{Haldane}, Haldane considers an event space in which he can count the number of cross-overs between two DNA segments during homologous recombination. In this example, we show how to recover this event space in the context of pedigrads. To do so, let $(E,\varepsilon)$ denote our usual pointed set $\{\mathtt{A},\mathtt{C},\mathtt{G},\mathtt{T},\varepsilon\}$ and let $(\Omega,\preceq)$ be the Boolean pre-ordered set $\{0\leq 1\}$ equipped with a finite chromology structure $(\Omega,D)$. The cones of $D$ will be specified later on. We will also let $b$ to be equal to $1 \in \Omega$. 

To count the number of cross-overs occurring between DNA segments, as Haldane did, we need to give ourselves a finite set of DNA segments of the same lengths that can be recombined. In terms of pedigrads, this would amount to picking an object $\tau$ in $\mathbf{Seg}(\Omega)$ and considering a monomorphism in $B_2\textrm{-}\mathbf{Mod}$ of the following type.
\[
B_2 \hookrightarrow DE_b^{\varepsilon}(\tau)
\]
The fact that this arrow is a monomorphism means that the sum of elements picked out by the element $1\in B_2$, say $\sum_{k=1}^n s_k$, is not equal to the zero element of $DE_b^{\varepsilon}(\tau)$. In other words, the recombination of the set of DNA segments $\{s_k~|~k \in [n]\}$ is not degenerate.

Now, the event space considered by Haldane can be seen as the pullback of the arrow $B_2 \hookrightarrow DE_b^{\varepsilon}(\tau)$ along the coequalizer map $FE_{b}^{\varepsilon}(\tau)\to DE_{b}^{\varepsilon}(\tau)$.
\[
\xymatrix{
\mathcal{F}\ar[r]\ar[d]_{\subseteq}\ar@{}[rd]|<<<{\rotatebox[origin=c]{90}{\huge{\textrm{$\llcorner$}}}}&B_2\ar[d]^{\subseteq}\\
FE_{b}^{\varepsilon}(\tau)\ar[r]&DE_{b}^{\varepsilon}(\tau)
}
\]
The resulting sub-object $\mathcal{F} \hookrightarrow FE_{b}^{\varepsilon}(\tau)$ corresponds to the equivalence class of the element $\sum_{k=1}^n s_k$ in $DE_{b}^{\varepsilon}(\tau)$. Let us give an example. Suppose that $\tau$ is the segment appearing at the top of the following (obvious) exactly 1-distributive cone, call it $\rho$, and suppose that $D$ only contains $\rho$.
\[
\xymatrix@C-30pt@R-15pt{
&&&&&&(\bullet)&(\bullet)&(\bullet)&(\bullet)\ar[llld]\ar[rrrd]&(\bullet)&(\bullet)&(\bullet)&&&&&&\\
(\bullet)&(\circ)&(\circ)&(\circ)&(\circ)&(\circ)&(\circ)&&&\dots&&&(\circ)&(\circ)&(\circ)&(\circ)&(\circ)&(\circ)&(\bullet)
}
\]
If the monomorphism $B_2 \hookrightarrow DE_b^{\varepsilon}(\tau)$ picks out the sum $\mathtt{A}\mathtt{G}\mathtt{T}\mathtt{C}\mathtt{C}\mathtt{T}\mathtt{A}+\mathtt{T}\mathtt{C}\mathtt{C}\mathtt{G}\mathtt{A}\mathtt{A}\mathtt{C}$, then the sub-semimodule $\mathcal{F} \hookrightarrow FE_{b}^{\varepsilon}(\tau)$ corresponds to the equivalence class of the recombination congruence associated with the $\rho$-recombination cone of $E_b^{\varepsilon}$ (Definition \ref{def:Relative_definition_families}), that is to say all the recombination events between the two words $\mathtt{A}\mathtt{G}\mathtt{T}\mathtt{C}\mathtt{C}\mathtt{T}\mathtt{A}$ and $\mathtt{T}\mathtt{C}\mathtt{C}\mathtt{G}\mathtt{A}\mathtt{A}\mathtt{C}$ (see examples below).
\[
\begin{array}{c}
\underbrace{\mathtt{A}\mathtt{G}\mathtt{T}\mathtt{C}\mathtt{C}\mathtt{T}\mathtt{A}}_{\textrm{0 cross-over}}+\underbrace{\mathtt{T}\mathtt{C}\mathtt{C}\mathtt{G}\mathtt{A}\mathtt{A}\mathtt{C}}_{\textrm{0 cross-over}}\\
\underbrace{\mathtt{A}\mathtt{G}\mathtt{T}\mathtt{C}\mathtt{C}\mathtt{T}\mathtt{A}}_{\textrm{0 cross-over}}+\underbrace{\mathtt{T}\mathtt{C}\mathtt{T}\mathtt{C}\mathtt{C}\mathtt{T}\mathtt{A}}_{\textrm{1 cross-over}}+\underbrace{\mathtt{T}\mathtt{C}\mathtt{C}\mathtt{G}\mathtt{A}\mathtt{A}\mathtt{C}}_{\textrm{0 cross-over}}\\
\underbrace{\mathtt{T}\mathtt{C}\mathtt{T}\mathtt{C}\mathtt{C}\mathtt{A}\mathtt{C}}_{\textrm{2 cross-overs}}+\underbrace{\mathtt{A}\mathtt{G}\mathtt{C}\mathtt{G}\mathtt{A}\mathtt{T}\mathtt{A}}_{\textrm{2 cross-overs}}+\underbrace{\mathtt{T}\mathtt{C}\mathtt{C}\mathtt{G}\mathtt{A}\mathtt{A}\mathtt{C}}_{\textrm{0 cross-over}}\\
\textrm{etc.}
\end{array}
\]
Since $\mathcal{F}$ is a finite subset of $FE_{b}^{\varepsilon}(\tau)$, the sum of all its elements corresponds to the representative of the equivalence class of $\mathtt{A}\mathtt{G}\mathtt{T}\mathtt{C}\mathtt{C}\mathtt{T}\mathtt{A}+\mathtt{T}\mathtt{C}\mathtt{C}\mathtt{G}\mathtt{A}\mathtt{A}\mathtt{C}$ (see Example \ref{exa:Haplogroups_x_x_x+y}). This representative event is usually called the \emph{sure event} \cite[Page 8]{Loeve}. On the other hand, the intersection of the supports of two events in $\mathcal{F}$ may not exist in $\mathcal{F}$ (see below).
\[
\Big\{
\underbrace{\mathtt{A}\mathtt{G}\mathtt{T}\mathtt{C}\mathtt{C}\mathtt{T}\mathtt{A}}_{\textrm{0 cross-over}},\underbrace{\mathtt{T}\mathtt{C}\mathtt{C}\mathtt{G}\mathtt{A}\mathtt{A}\mathtt{C}}_{\textrm{0 cross-over}}\Big\} \cap \Big\{
\underbrace{\mathtt{T}\mathtt{C}\mathtt{T}\mathtt{C}\mathtt{C}\mathtt{T}\mathtt{A}}_{\textrm{1 cross-over}},\underbrace{\mathtt{A}\mathtt{G}\mathtt{C}\mathtt{G}\mathtt{A}\mathtt{A}\mathtt{C}}_{\textrm{1 cross-over}},\underbrace{\mathtt{A}\mathtt{G}\mathtt{C}\mathtt{G}\mathtt{A}\mathtt{A}\mathtt{C}}_{\textrm{0 cross-over}}
\Big\} = \Big\{\underbrace{\mathtt{T}\mathtt{C}\mathtt{C}\mathtt{G}\mathtt{A}\mathtt{A}\mathtt{C}}_{\textrm{0 cross-over}}\Big\}  \notin \mathcal{F}
\]
The fact that intersections of events do not necessarily belong to $\mathcal{F}$ could translate the idea that natural selection may be able to shape the observable results of homologous recombination (e.g. sperm selection, fetal death, selection of hatchlings, etc.). 

However, note that in the case where the arrow $B_2 \hookrightarrow DE_b^{\varepsilon}(\tau)$ picks out a sum of two elements, such as $\mathtt{A}\mathtt{G}\mathtt{T}\mathtt{C}\mathtt{C}\mathtt{T}\mathtt{A}+\mathtt{T}\mathtt{C}\mathtt{C}\mathtt{G}\mathtt{A}\mathtt{A}\mathtt{C}$, it is always possible to uniquely complete an intersection into an element of $\mathcal{F}$ by considering the complementary recombination operations -- this element is the smallest element of $\mathcal{F}$ whose support contains that intersection. For instance, the elements given below, on the left, can be seen as the corresponding sums, on the right.
\[
\begin{array}{lll}
\mathtt{T}\mathtt{C}\mathtt{C}\mathtt{G}\mathtt{A}\mathtt{A}\mathtt{C} \notin \mathcal{F}\quad\quad\mapsto\quad\quad \mathtt{A}\mathtt{G}\mathtt{T}\mathtt{C}\mathtt{C}\mathtt{T}\mathtt{A}+\mathtt{T}\mathtt{C}\mathtt{C}\mathtt{G}\mathtt{A}\mathtt{A}\mathtt{C} \in \mathcal{F}\\
\mathtt{T}\mathtt{C}\mathtt{C}\mathtt{C}\mathtt{C}\mathtt{T}\mathtt{A} \notin \mathcal{F}\quad\quad\mapsto\quad\quad \mathtt{T}\mathtt{C}\mathtt{C}\mathtt{C}\mathtt{C}\mathtt{T}\mathtt{A}+\mathtt{A}\mathtt{G}\mathtt{T}\mathtt{G}\mathtt{A}\mathtt{A}\mathtt{C} \in \mathcal{F}\\
\end{array}
\]
In the case where the arrow $B_2 \hookrightarrow DE_b^{\varepsilon}(\tau)$ picks out a sum of more than two elements, the completion is not necessarily unique.
\end{example}

\begin{definition}[Bounded]
We will say that an event space $\mathcal{F}$ over a set $S$ is \emph{bounded} if it admits a maximum element for the inclusion of subsets of $S$. Such a maximum will be denoted as $\cup\mathcal{F}$ and called the \emph{sure event}.
\end{definition}

\begin{proposition}
An event space $\mathcal{F}$ over a set $S$ is bounded, if and only if it is a finite subsets of $\mathcal{P}(S)$.
\end{proposition}
\begin{proof}
Let $\mathcal{F}$ be a bounded event space over $S$. Because the element $\cup \mathcal{F}$ belongs to $\mathcal{F}$, it is a finite set. Since $\cup \mathcal{F}$ is also the union of all elements in $\mathcal{F}$, the inclusion $\mathcal{F} \subseteq \mathcal{P}(\cup \mathcal{F})$ must holds, so that $\mathcal{F}$ is finite. Conversely, if $\mathcal{F}$ is finite, then the union of all its element is a maximum element.
\end{proof}

\begin{remark}[Maximum element]
From the point of view of $B_2$-semimodules, the maximum of a bounded event space is the finite sum of all its elements.
\end{remark}

\begin{definition}[Probability measure]
Let $\mathcal{F}$ be a bounded event space over a set $S$. We will speak of a \emph{probability measure} on $\mathcal{F}$ to refer to a function $\wp:\mathcal{F} \to [0,1]$ such that 
\begin{itemize}
\item[1)] both identities $\wp(\emptyset)=0$ and $\wp(\cup \mathcal{F}) = 1$ hold.
\item[2)] if $A,B \in \mathcal{F}$ with $A \cap B = \emptyset$, then the identity $\wp(A \cup B) = \wp(A) + \wp(B)$ holds.
\end{itemize}
\end{definition}

\begin{example}[Haldane's probability measure]\label{exa:Haldane_probability_measure}
This example carries on Example \ref{exa:Haldane_event_space} and discusses the definition of a probability measure on the event space defined thereof. First, recall that, in \cite{Haldane}, Haldane associates the space of pairs of recombined segments (i.e  those pairs containing two complementary segments for the recombination operation) with a Bernouilli distribution \cite[Page 12]{Loeve} as illustrated below.
\[
\begin{array}{l}
\xymatrix@C-25pt{
\mathtt{T}\ar@{}[r]_{\uparrow}\ar@<-2ex>@{}[r]_{0}&\mathtt{C}\ar@{}[r]_{\uparrow}\ar@<-2ex>@{}[r]_{1}&\mathtt{T}\ar@{}[r]_{\uparrow}\ar@<-2ex>@{}[r]_{0}&\mathtt{C}\ar@{}[r]_{\uparrow}\ar@<-2ex>@{}[r]_{0}&\mathtt{C}\ar@{}[r]_{\uparrow}\ar@<-2ex>@{}[r]_{0}&\mathtt{T}\ar@{}[r]_{\uparrow}\ar@<-2ex>@{}[r]_{0}&\mathtt{A}&+&\mathtt{A}\ar@{}[r]_{\uparrow}\ar@<-2ex>@{}[r]_{0}&\mathtt{G}\ar@{}[r]_{\uparrow}\ar@<-2ex>@{}[r]_{1}&\mathtt{C}\ar@{}[r]_{\uparrow}\ar@<-2ex>@{}[r]_{0}&\mathtt{G}\ar@{}[r]_{\uparrow}\ar@<-2ex>@{}[r]_{0}&\mathtt{A}\ar@{}[r]_{\uparrow}\ar@<-2ex>@{}[r]_{0}&\mathtt{A}\ar@{}[r]_{\uparrow}\ar@<-2ex>@{}[r]_{0}&\mathtt{C}
}
\end{array}
\quad\mapsto\quad
p^{1} \cdot (1-p)^{5}
\]
In general, the probability $p$ associated with a recombination event (symbolized, in the previous sum, by 1) is equal to a ratio $x/n$ where $n$ is the number of positions at which a recombination can occur on the segment and $x$ is the expected number of recombination events on this segment. In the previous case, the probability $p$ should be equal to $x/6$.

For its part, the event space $\mathcal{F} \hookrightarrow FE_b^{\varepsilon}(\tau)$ of Example \ref{exa:Haldane_event_space} can be associated with a more refined probability measure by associating every DNA segment of a recombination event with one half of the probability that it was given in Haldane's case. A pair of segments that is complementary for a recombination operation will therefore be endowed with the same probability measure as in Haldane's case.
\[
\begin{array}{lll}
\underbrace{\mathtt{T}\mathtt{C}\mathtt{T}\mathtt{C}\mathtt{C}\mathtt{T}\mathtt{A}}_{\textrm{1 cross-over}}+\underbrace{\mathtt{A}\mathtt{G}\mathtt{C}\mathtt{G}\mathtt{A}\mathtt{A}\mathtt{C}}_{\textrm{1 cross-over}}+\underbrace{\mathtt{T}\mathtt{C}\mathtt{C}\mathtt{G}\mathtt{A}\mathtt{A}\mathtt{C}}_{\textrm{0 cross-over}}&\mapsto&p(1-p)^5+\frac{1}{2}(1-p)^6\\
\underbrace{\mathtt{T}\mathtt{C}\mathtt{T}\mathtt{C}\mathtt{C}\mathtt{T}\mathtt{A}}_{\textrm{1 cross-over}}+\underbrace{\mathtt{A}\mathtt{G}\mathtt{C}\mathtt{G}\mathtt{A}\mathtt{A}\mathtt{C}}_{\textrm{1 cross-over}}+\underbrace{\mathtt{A}\mathtt{G}\mathtt{C}\mathtt{G}\mathtt{A}\mathtt{A}\mathtt{C}}_{\textrm{0 cross-over}}&\mapsto&p(1-p)^5+\frac{1}{2}(1-p)^6\\
\underbrace{\mathtt{T}\mathtt{C}\mathtt{T}\mathtt{C}\mathtt{C}\mathtt{T}\mathtt{A}}_{\textrm{1 cross-over}}+\underbrace{\mathtt{A}\mathtt{G}\mathtt{C}\mathtt{G}\mathtt{A}\mathtt{A}\mathtt{C}}_{\textrm{1 cross-over}}+\underbrace{\mathtt{A}\mathtt{G}\mathtt{T}\mathtt{C}\mathtt{C}\mathtt{T}\mathtt{A}}_{\textrm{0 cross-over}}+\underbrace{\mathtt{T}\mathtt{C}\mathtt{C}\mathtt{G}\mathtt{A}\mathtt{A}\mathtt{C}}_{\textrm{0 cross-over}}&\mapsto&p(1-p)^5+(1-p)^6
\end{array}
\]
In this case, we can verify that if two events are disjoint (i.e. $A \cap B = \emptyset$), then the measure of the union of these is equal to the sum of the measures of each of them (i.e. $\wp(A \cup B) = \wp(A) + \wp(B)$). Of course, other definitions of probability measures are possible.

In \cite{Haldane}, Haldane eventually considers long DNA sequences that are subdivided in small intervals so that the probability measure on a sum of complementary segments that are subject to exactly $t$ recombination operations is taken to be equal to the following limit.
\[
\frac{n!}{t!(n-t)!}\Big(\frac{x}{n}\Big)^t\Big(1-\frac{x}{n}\Big)^{n-t} \mathop{\longrightarrow}\limits_{n \to \infty}  x^t\frac{e^{-x}}{t!}
\]
It follows that an observable cross-over event, which is given by the union of all the odd recombination events acting on the sum $\mathtt{A}\mathtt{G}\mathtt{T}\mathtt{C}\mathtt{C}\mathtt{T}\mathtt{A}+\mathtt{T}\mathtt{C}\mathtt{C}\mathtt{G}\mathtt{A}\mathtt{A}\mathtt{C}$, is equal to $(1-e^{-2x})/2$.
\end{example}

\subsection{Toward more formalism}\label{Toward_more_formalism}
We can already see in Examples \ref{exa:Haldane_event_space} and \ref{exa:Haldane_probability_measure} that the computation of mapping functions lack of a certain formalism. This lack of formalism is also mentioned in the literature \cite{McPeekSpeed,Speed_GMF}.

%1) For instance, it is strange that the event space on which the probability measure associated with the mapping function is computed need to consider pairs of segments \cite{Haldane} while the observation of homologous recombination is actually done through the effect of a unique segment, namely one chromatid for each meiosis (see \cite{McPeekSpeed}).

1) A first instance in which more formalism seems to be needed lies in the fact that homologous recombination between more than two DNA strands is not well-understood from the point of view of mapping functions \cite[Page 1033]{McPeekSpeed}. Indeed, while such a recombination event could be noticed from a pair of DNA strands, the actual recombination would only be determined by the knowledge of all the strands resulting from it. This is one of the reasons why the event space of Example \ref{exa:Haldane_event_space} was not required to be stable under intersection. This also suggests that while homologous recombination between pairs of segments can be modelled via regular probability models, multi-locus recombination between more than two segments might need to be studied from a structural point of view. Pedigrads and chromologies could therefore play a certain role toward this prospect.

2) A second instance that could be mentioned is that %how the probability measure of an event space for recombination is defined. We can notice that the parameter $x$ used in Example \ref{exa:Haldane_probability_measure} is not clearly defined. This parameter is in fact determined \emph{a posteriori} from the number of observed recombination events; in \cite{Haldane}, this parameter is said to be a distance (usually measured in Morgan units \cite{McPeek}), but this assumption gives the probability $p$ a dimensional nature, which does not make sense with the fact that $p$ is supposed to be a ratio. 
Haldane's Poisson model also fails at modelling cross-over interference \cite{Speed_GMF}. Cross-over interference occurs when a recombination event influences the number of recombination happening around it. Cross-over interference is measured via the so-called \emph{coincidence coefficient}, which is the ratio of the recombination rate of two intervals $A$ and $B$ (on the DNA segment) over the simultaneous recombination rate on $A$ and $B$ (see \cite{Speed_GMF}).
When this ratio is less than 1, the recombination events are more likely to be clustered on the segment while when it is greater than 1, these events are more likely to be evenly spaced \cite{BromanWeber}.

The desire to incorporate cross-over interference in recombination models \cite{Speed_GMF,McPeekSpeed} could motivate the consideration of topologies as one of the varying parameters. This would be something that the language of chronologies could fit quite well. For instance, instead of considering the cone $\rho$ of Example \ref{exa:Haldane_event_space}, we could now consider recombination cones induced by the following cones of segments.
\[
\xymatrix@C-30pt@R-15pt{
&&&&&(\bullet&\bullet)&(\bullet&\bullet)\ar[lld]\ar[rrd]&(\bullet&\bullet&\bullet)&&&&&\\
(\bullet&\bullet)&(\circ&\circ)&(\bullet&\bullet&\bullet)&&&&(\circ&\circ)&(\bullet&\bullet)&(\circ&\circ&\circ)
}
\quad\quad\quad
\xymatrix@C-30pt@R-15pt{
&&&&&(\bullet&\bullet&\bullet)&(\bullet\ar[lld]\ar[rrd]&\bullet&\bullet&\bullet)&&&&&\\
(\circ&\circ&\circ)&(\bullet&\bullet&\bullet&\bullet)&&&&(\bullet&\bullet&\bullet)&(\circ&\circ&\circ&\circ)
}
\]

The cone given on the left would model the fact that a recombination event happening on an interval can influence another one that is located farther on the chromosome. This would be due to the fact that a chromosome can pack its DNA, so that the physical distance between two DNA patches is not necessarily related to their distance on the chromosome. On the other hand, the cone given above, on the right could, be used to model a spaced recombination rate. Of course, in order to take into account any possible recombination events of this type on a chromosome, one would need to consider what one could like to call the set of  `permutations' of these cones (see below).
\begin{center}
\scalebox{0.8}{
$$
\xymatrix@C-30pt@R-15pt{
&&&&&(\bullet&\bullet&\bullet)&(\bullet\ar[lld]\ar[rrd]&\bullet&\bullet&\bullet)&&&&&\\
(\bullet&\bullet&\bullet)&(\circ&\circ&\circ&\circ)&&&&(\circ&\circ&\circ)&(\bullet&\bullet&\bullet&\bullet)
}
\quad
\xymatrix@C-30pt@R-15pt{
&&&&&(\bullet&\bullet)&(\bullet&\bullet\ar[lld]\ar[rrd]&\bullet&\bullet)&(\bullet)&&&&&\\
(\bullet&\bullet)&(\circ&\circ&\circ&\circ)&(\circ)&&\dots&&(\circ&\circ)&(\circ&\circ&\circ&\circ)&(\bullet)
}
\quad
\xymatrix@C-30pt@R-15pt{
&&&&&(\bullet)&(\bullet&\bullet&\bullet\ar[lld]\ar[rrd]&\bullet)&(\bullet&\bullet)&&&&&\\
(\bullet)&(\circ&\circ&\circ&\circ)&(\circ&\circ)&&\dots&&(\circ)&(\circ&\circ&\circ&\circ)&(\bullet&\bullet)
}
$$
}
\end{center}
\begin{center}
\scalebox{0.8}{
$$
\xymatrix@C-30pt@R-15pt{
&&&&&(\bullet&\bullet&\bullet&\bullet)\ar[lld]\ar[rrd]&(\bullet&\bullet&\bullet)&&&&&\\
(\bullet&\bullet&\bullet)&(\circ&\circ&\circ&\circ)&&&&(\circ&\circ&\circ)&(\bullet&\bullet&\bullet&\bullet)
}
$$
}
\end{center}
These sets of cones could then be organized in the form of a chromology and the study of the related recombination events would be done with respect to the peak of each of these cones. Making the cones vary would amount to studying a category whose objects are recombination schemes.

\section{Conclusion}

Pedigrads and chromologies provide a language in which it is possible to talk about many biological phenomena, such as homologous recombination, parsing, mutations, duplications, inversions and CRISPR. Pedigrads in the category of $B_2$-semimodules were shown to provide a framework in which it is possible to talk about recombination events and genetic linkage.

%\newpage
%\tableofcontents

%    Bibliographies can be prepared with BibTeX using amsplain,
%    amsalpha, or (for "historical" overviews) natbib style.

\bibliographystyle{plain}

\end{document}